\documentclass[a4paper,11pt,reqno]{amsart}
\usepackage{amssymb}
\usepackage{enumerate}
\usepackage{amscd,eucal}
\usepackage{graphicx}
\newtheorem{theorem}{Theorem}[section]
\newtheorem{lemma}[theorem]{Lemma}
\newtheorem{proposition}[theorem]{Proposition}
\newtheorem{corollary}[theorem]{Corollary}
\newtheorem{theorem-definition}[theorem]{Theorem-Definition}

\theoremstyle{definition}
\newtheorem{definition}[theorem]{Definition}

\newtheorem{notation}[theorem]{Notation}
\newtheorem{problem}[theorem]{Problem}
\newtheorem{question}[theorem]{Question}
\newtheorem{example}[theorem]{Example}

\theoremstyle{remark}
\newtheorem{remark}[theorem]{Remark}

\numberwithin{equation}{section}

\allowdisplaybreaks

\def\aa{{\mathfrak{B}}}
\def\bb{{\mathcal{B}}}
\def\dd{{\mathcal{D}}}
\def\cc{{{C}}}

\def\gh{{\mathcal{H}}}

\def\hc{{\mathfrak{H}}}

\def\mk{{\mathfrak{m}}}
\def\mm{{{\widetilde{\mathfrak{m}}}_c}}
\def\m2{{{\widetilde{\mathfrak{m}}}_2}}
\def\mb{{\boldsymbol{l}}}
\def\nb{{\boldsymbol{m}}}
\def\nn{{\mathcal{N}}}
\def\pp{{\mathcal{P}}}

\def\rs{{\delta}}
\def\sss{{\mathcal{S}}}

\DeclareFontEncoding{OT2}{}{}
\DeclareFontSubstitution{OT2}{cmr}{m}{n}
\DeclareSymbolFont{cyss}{OT2}{wncyss}{m}{n}
\DeclareMathSymbol{\sh}{\mathbin}{cyss}{`x}

\allowdisplaybreaks
\begin{document}

\baselineskip 16pt 

\title[Complex multiple zeta-functions and $p$-adic multiple $L$-functions]{Desingularization of complex multiple zeta-functions, fundamentals of $p$-adic multiple $L$-functions, and 
evaluation of their special values}


\author[H. Furusho, Y. Komori, K. Matsumoto, and H. Tsumura]{Hidekazu Furusho, Yasushi Komori, Kohji Matsumoto, and Hirofumi Tsumura}

\date{February 16, 2014}

\subjclass[2010]{Primary 11M32, 11S40, 11G55; Secondary 11M41, 11S80}
\keywords{Complex multiple zeta-function, desingularization, 
$p$-adic multiple $L$-function, 
$p$-adic multiple polylogarithm, 
multiple Bernoulli numbers,
multiple Kummer congruence,
Coleman's $p$-adic integration theory}

\maketitle

\begin{abstract}
This paper deals with a multiple version of zeta- and $L$-functions
both in the complex case and in the $p$-adic case:\\
(I). Our motivation in the complex case is to find 
suitable rigorous meaning
of the values of multivariable multiple zeta-functions at non-positive integer points.\\
\ \ (a). A desingularization of multiple
zeta-functions (of the generalized Euler-Zagier type):
We reveal that multiple zeta-functions (which are known to be 
meromorphic in the whole space
with  whose singularities lying on infinitely many hyperplanes)
turn to be entire on the whole space after taking the desingularization.
Further we show that the desingularized function is given by
a suitable finite `linear' combination of multiple zeta-functions with some arguments shifted.
It is also shown that specific combinations of Bernoulli numbers attain the
special values at their non-positive integers of the desingularized
ones.\\
\ \ (b). Twisted multiple zeta-functions:  Those can be continued to entire functions,
and their special values at non-positive integer points can be explicitly calculated.\\
(II). Our work in the $p$-adic case is to develop the study on
analytic side of the Kubota-Leopoldt $p$-adic $L$-functions into the
multiple setting.
We construct $p$-adic multiple $L$-functions, multivariable versions of
their $p$-adic $L$-functions, by using a specific $p$-adic
measure.
We establish their various fundamental properties: \\
\ \ (a). Evaluation at non-positive integers: 
We establish their intimate connection with
the above complex multiple zeta-functions 
by showing that the special values of the $p$-adic multiple $L$-functions at non-positive integers
are expressed by 
the twisted multiple Bernoulli numbers, the special values of the complex multiple  zeta-functions
at non-positive integers.  \\
\ \ (b). Multiple Kummer congruence: We extend Kummer congruence for  Bernoulli numbers 
to congruences for the twisted multiple Bernoulli numbers.\\
\ \ (c). Functional relations with a parity condition: We extend the vanishing
property of the Kubota-Leopoldt $p$-adic $L$-functions with odd
characters to our $p$-adic multiple $L$-functions. \\
\ \ (d). Evaluation at positive integers: 
We establish their close relationship with  the $p$-adic twisted multiple polylogarithms
by showing that the special values of the $p$-adic multiple $L$-functions at positive integers
are described by those of the $p$-adic twisted multiple polylogarithms at roots of
unity,
which generalizes the previous result of Coleman  in the single variable case.
\end{abstract} 
\tableofcontents

\setcounter{section}{-1}
\section{Introduction} \label{sec-1}

The aim of the present paper is to consider multiple zeta- and $L$-functions
both in the complex case and in the $p$-adic case, and especially study their special values
at integer points.

Let our story begin with the {\bf multiple zeta-function of 
the generalized Euler-Zagier type} defined by
\begin{align}
&\zeta_r((s_j);(\gamma_j))= 
\zeta_r(s_1,\ldots,s_r;\gamma_1,\ldots,\gamma_r):=\sum_{\substack{m_1=1}}^\infty\cdots \sum_{\substack{m_r=1}}^\infty
    \prod_{j=1}^r
    \left(m_1\gamma_1+\cdots+m_j\gamma_j\right)^{-s_j}   
\label{gene-EZ}
\end{align}
for complex variables $s_1,\ldots,s_r$, where
$\gamma_1,\ldots,\gamma_r$ are complex parameters whose real parts are all positive. 
Series \eqref{gene-EZ} converges absolutely 
in the region
\begin{equation}                                                                      
\dd_r=\{(s_1,\ldots,s_r)\in \mathbb{C}^r~|~\Re (s_{r-k+1}+\cdots+s_r)>k\ 
(1\leqslant k\leqslant r)\}. \label{region-Z}                                         
\end{equation}
The first work which established the meromorphic continuation of 
\eqref{gene-EZ} is Essouabri's thesis \cite{Ess}. 
The third-named author \cite[Theorem 1]{MaJNT} showed that \eqref{gene-EZ} can be continued meromorphically to the whole complex space with infinitely many 
(possible) 
singular hyperplanes.   

A special case of \eqref{gene-EZ} is the {\bf multiple zeta-function of 
Euler-Zagier type} defined by

\begin{equation}                                                                        
\zeta_r((s_j))=
\zeta_r(s_1,s_2,\ldots,s_r)=\sum_{m_1,\ldots,m_r=1}^\infty \prod_{j=1}^{r}
\left(m_1+\cdots+m_j\right)^{-s_j}, \label{MZF-def}                              
\end{equation}
which is absolutely convergent in $\dd_r$.

Note that $\zeta_r((s_j))=\zeta_r((s_j);(1)).$
Its special value $\zeta_r(n_1,\dots,n_r)$ when $n_1,\dots,n_r$ are positive integers 
makes sense when $n_r>1$.
It is called the {\bf multiple zeta value} (abbreviated as MZV),
history of whose  study goes back to the work of Euler \cite{Eu} published in 1776
\footnote{                                                                              
You can find  several literatures which cite the paper  saying as if it were published 
in 1775.                                                                                
But according to Euler archive                                                          
{\tt http://eulerarchive.maa.org/},                                                     
it was written in 1771, presented in 1775 and published in 1776.                        
}.
For a couple of these decades, it has been intensively studied
in various fields including  number theory, algebraic geometry,
low dimensional topology
and mathematical physics.

On the other hand, after the work of meromorphic continuation of \eqref{gene-EZ} 
mentioned above, 
it is natural to ask how is the behavior of $\zeta_r(-n_1,\dots,-n_r)$ when 
$n_1,\dots,n_r$ are positive (or non-negative) integers.
However, as we will mention in Section \ref{sec-2-1}, in most cases these points are
on singular loci, and hence they are points of indeterminacy.
Therefore we can raise the following fundamental problem.
\begin{problem}\label{prob}
Are there any \lq rigorous' ways to give a meaning of $\zeta_r(-n_1,\dots,-n_r)$
for $n_1,\ldots,n_r\in{\mathbb Z}_{\geqslant 0}$?
\end{problem}

Several approaches to this problem have been done so far.
Guo and Zhang \cite{GZ},  Manchon and Paycha \cite{MP} 
and also Guo, Paycha and Zhang \cite{GPZ}
discussed a kind of renormalization method.
In the present paper we will develop yet another approach,
called the \textit{desingularization}, in Subsection \ref{c-1-zeta}.
The Riemann zeta-function $\zeta(s)$ is a meromorphic function on the complex plane
$\mathbb{C}$ with a simple and unique pole at $s=1$.   Hence $(s-1)\zeta(s)$ is an 
entire function. This simple fact may be regarded as a technique to resolve a 
singularity of $\zeta(s)$ and yield an entire function.  
Our desingularization method is motivated by this simple observation.
For $r\geqslant 2$, multiple
zeta-functions have infinitely many singular loci (see {\sc{Figure}} \ref{fig:0} and
\eqref{ex-04-3} for the case $r=2$). 
We will show that
a suitable \textit{finite} sum of multiple zeta-functions will cause cancellations of all of those singularities to produce an entire function
whose special values at non-positive integers are described explicitly
in terms of Bernoulli numbers.

\begin{figure}[h]
  \centering
  \includegraphics[bb=0 0 232 113]{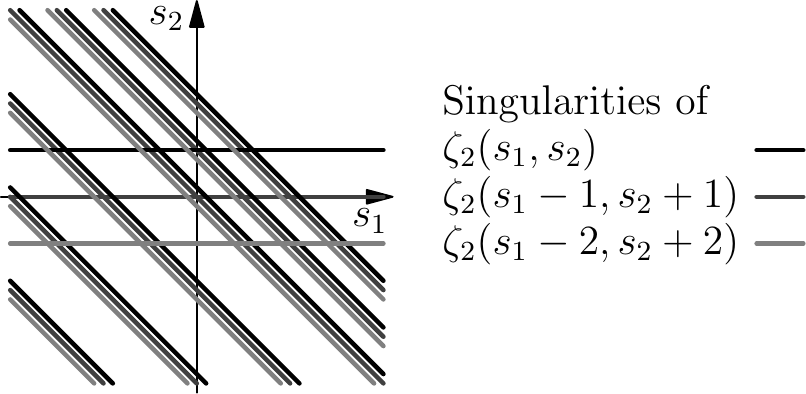}
  \caption{Singularities of $\zeta_2$'s}
  \label{fig:0}
\end{figure}

Another possible approach to the above Problem \ref{prob} is to consider the twisted
multiple series.
Let $\xi_1,\ldots,\xi_r\in \mathbb{C}$ be roots of unity. For 
$\gamma_1,\ldots,\gamma_r\in \mathbb{C}$ with
$\Re \gamma_j >0$ ($1\leqslant j\leqslant r$), 
define the {\bf multiple zeta-function of the generalized Euler-Zagier-Lerch type} by
\begin{equation}                                                                        
\label{Barnes-Lerch}                                                                    
\zeta_r((s_j);(\xi_j);(\gamma_j)):=                                                     
    \sum_{\substack{m_1=1}}^\infty\cdots \sum_{\substack{m_r=1}}^\infty                 
    \prod_{j=1}^r \xi_j^{m_j}
    (m_1\gamma_1+\cdots+m_j\gamma_j )^{-s_j},                               
\end{equation}
which is absolutely convergent in the region $\dd_r$ defined by \eqref{region-Z}. 
We note that the multiple zeta-function of  the generalized Euler-Zagier type 
\eqref{gene-EZ}
is its special case, that is,
$$\zeta_r((s_j);(\gamma_j))=\zeta_r((s_j);(1);(\gamma_j)).$$
Because of the existence of the twisting factor $\xi_1,\ldots,\xi_r$, we can see 
(in Theorem \ref{T-multiple} below) that,
if all $\xi_j$ is not equal to 1, series \eqref{Barnes-Lerch} can be continued to
an entire function, hence its values at non-positive integer points have a rigorous
meaning.    Moreover we will show that those values can be written explicitly in
terms of twisted multiple Bernoulli numbers.

The above second approach naturally leads us to the theory of $p$-adic multiple
$L$-functions.   Our next main theme in the present paper is to search for a
multiple analogue of Kubota-Leopoldt $p$-adic $L$-functions.

In the 1960s, Kubota and Leopoldt \cite{K-L} gave the first construction of the $p$-adic analogue of the Dirichlet $L$-function $L(s,\chi)$ associated with the Dirichlet character $\chi$, which is called the $p$-adic $L$-function denoted by $L_p(s;\chi)$. This can be regarded as a $p$-adic interpolation of values of the Dirichlet $L$-function at non-positive integers, based on Kummer's congruences for Bernoulli numbers. 

Iwasawa \cite{Iwasawa69} proposed a different way of constructing 
$p$-adic $L$-functions, based on the study of the arithmetic of Galois modules associated with certain towers of algebraic number fields called $\mathbb{Z}_p$-extensions (see also \cite{Iwasawa72}). In particular, his method shows that the $p$-adic $L$-function can be expressed by use of the power series. His study, called the Iwasawa theory, developed spectacularly with recognition of the importance of $p$-adic $L$-functions. 

In the 1970s, other constructions of $p$-adic $L$-functions were given, for example, by Amice and Fresnel \cite{AF}, Coates \cite{Co}, Koblitz \cite{Kob79} and Serre \cite{Se}. More generally $p$-adic $L$-functions over totally real number fields were constructed by Barsky \cite{Ba}, Cassou-Nogu\`es \cite{Cass}, 
Deligne and Ribet \cite{DR}. 
These works are based on the $p$-adic properties of values of abelian $L$-functions 
over totally real number fields at non-positive integers. In fact, the approach of 
Deligne and Ribet were built on the theory of $p$-adic Hilbert modular forms, while
those of Barsky and of Cassou-Nogu\`es were built on the pioneering work of 
Shintani \cite{Shi}.
Shintani's work also inspired the theory of multivariable $p$-adic $L$-functions.
In fact, motivated by \cite{Shi}, Imai \cite{Imai81} constructed certain 
multivariable $p$-adic $L$-functions and Hida \cite{Hida} constructed $p$-adic 
analogues of Shintani's $L$-functions.

Recently, the second, the third and the fourth named author \cite{KMT-IJNT} 
introduced  a certain double analogue of the Kubota-Leopoldt $p$-adic $L$-function,
which could also be seen as a $p$-adic analogue of the double ($r=2$) zeta function
of a specific case.
Its evaluation at non-positive integers and a certain functional relation with
the Kubota-Leopoldt $p$-adic $L$-function were shown.


In the present paper we generalize the method in \cite{KMT-IJNT} to construct 
the multivariable {\bf $p$-adic multiple $L$-function}, 
\begin{equation}\label{pMLF}
L_{p,r}(s_1,\ldots,s_r;\omega^{k_1},\ldots,\omega^{k_r};\gamma_1,\ldots,\gamma_{r};c)
\end{equation}
It is a $p$-adic valued function for $p$-adic variables $s_1,\dots, s_r$,
which is attached to each 
$\omega^{k_1},\ldots,\omega^{k_r}$
($\omega$: the Teichm\"{u}ller character, $k_1,\ldots,k_r\in \mathbb{Z}$),
$\gamma_1,\ldots,\gamma_r\in \mathbb{Z}_p$
and $c\in \mathbb{N}_{>1}$ with $(c,p)=1$.

It can be seen 
as a multiple analogue of
the Kubota-Leopoldt $p$-adic $L$-function
and the above mentioned $p$-adic double $L$-function
and also regarded as a $p$-adic 
analogues of the complex multivariable multiple zeta-functions \eqref{Barnes-Lerch}
in a sense.
Actually it can be constructed by a multiple analogue of 
the $p$-adic gamma transform of a $p$-adic  measure of Koblitz \cite{Kob79}.
We  investigate $p$-adic properties  of \eqref{pMLF},
particularly its $p$-adic  analyticity  in the parameter $s_1,\ldots,s_r$ and 
its $p$-adic continuity in the parameter $c$.

Evaluation of \eqref{pMLF} at integral points is one of our main themes of this paper:
By explicitly describing its special values at all non-positive  integer points
in terms on twisted Bernoulli numbers,
we show that our $p$-adic multiple $L$-function \eqref{pMLF} is closely connected to
the complex multiple zeta function \eqref{Barnes-Lerch}
of the generalized Euler-Zagier-Lerch type
via their special values at non-positive integers.
Our evaluation yields  two particular results;
multiple Kummer congruence and functional relations. 
The multiple Kummer congruence is a certain $p$-adic congruence among the special values,
which includes  an ordinal Kummer congruence for Bernoulli numbers as a very special case
and where we also discover a newly-looked (or not?)  type of congruence for Bernoulli numbers.
The functional relations are linear relations among $p$-adic multiple $L$-functions
as a function, 
which reduce \eqref{pMLF} as a linear sum of \eqref{pMLF} with lower $r$.
The relations is attached to a parity of $k_1+\cdots+k_r$.
It extends  the known fact in the single variable case that 
the Kubota-Leopoldt $p$-adic $L$-function with odd character
is identically zero function.
As the double variable case, we recover the result of \cite{KMT-IJNT}
that a certain $p$-adic double $L$-function  is  equal to 
the Kubota-Leopoldt $p$-adic $L$-function 
up to a minor factor. 

Whilst, as for evaluation of \eqref{pMLF}  at all positive integers,
we show that  the special value of $p$-adic multiple $L$-function \eqref{pMLF} 
with $\gamma_1=\cdots=\gamma_r=1$
at  any positive integer points $(n_1,\dots,n_r)$ is given by  
the special values of
{\bf $p$-adic twisted multiple polylogarithms}
($p$-adic TMPL's, in short)
\begin{equation}\label{pTMPL}
Li^{(p)}_{n_1,\dots,n_r}(\xi_1,\dots,\xi_r ;z)
\end{equation}
at unity.
The above  $p$-adic TMPL  \eqref{pTMPL} is a $p$-adic valued function for $p$-adic variable $z$,
which is attached to each  positive integers $n_1,\dots,n_r$
and certain $p$-adic parameters $\xi_1,\dots,\xi_r$
(which are occasionally roots of unity.)
We construct the function  by using Coleman's $p$-adic iterated integration theory \cite{C}.
The construction generalizes that of $p$-adic polylogarithm by Coleman \cite{C}
and that of $p$-adic multiple polylogarithm and $p$-adic multiple zeta ($L$-)values
by the first-named author \cite{Fu1, F2} and  Yamashita \cite{Y}.
Our result here shows that there is a close connection between the theory of $p$-adic
multiple $L$-functions of the Kubota-Leopoldt type initiated in \cite{KMT-IJNT},  
and the theory of $p$-adic multiple polylogarithms developed by
the first-named author \cite{Fu1,F2}
which were introduced under a very different motivation.
It also generalizes the previous result of Coleman \cite{C}
which connects the Kubota-Leopoldt $p$-adic $L$-function with his $p$-adic polylogarithm.
In Remark 4.4 of \cite{KMT-IJNT} it is written that it is unclear whether there is 
some connection between these two theories or not. 
Our results in the present paper give an answer to this question.

We remark that our multivariable $p$-adic multiple $L$-functions are different from 
$p$-adic multiple zeta functions by Tangedal and Young \cite{TY},
which are one variable $p$-adic functions stemming from the works of
Shintani, Cassou-Nogu\`{e}s, Yoshida and Kashio,
though pursuing relationship between our works and theirs might be 
a significant research.
A relationship to
non-commutative Iwasawa theory (Coates, Fukaya, Kato, Sujatha  and Venjakob \cite{CFKSV}, etc)
might be another direction.
The theory is 
a generalization of Iwasawa theory but it
looks lacking  its analytic side.
It is not clear whether our work in this paper lead any direction  related to this or not,
which might be worthy to discuss further.

Here is the plan of this paper:
In Section \ref{sec-2}, we will introduce the \textit{twisted $r$-ple Bernoulli numbers} 
$\{\aa((n_j);( \xi_j);( \gamma_j))\}$ associated with 
integers $\{n_j\}_{1\leqslant j\leqslant r}$ such that $n_j\geqslant -1$,
roots of unity $\{\xi_j\}_{1\leqslant j\leqslant r}$
and complex numbers $\gamma_j$ $(1\leqslant j\leqslant r)$
(Definition \ref{Def-M-Bern})
as generalizations of ordinary Bernoulli numbers. 
We will show that $\zeta_r((s_j);(\xi_j);(\gamma_j))$ 
is analytically continued to the whole space as
an entire function and interpolates $\{\aa((n_j);( \xi_j);( \gamma_j))\}$
at non-positive integers when all $\xi_j$ are not $1$ (Theorem \ref{T-multiple}). 
On the other hand, when $\xi_j$ are all equal to $1$,
then $\zeta_r((s_j);(1);(\gamma_j))=\zeta_r((s_j);(\gamma_j))$
is meromorphically continued to the whole space
with  whose singularities lying on infinitely many hyperplanes.
We will introduce and develop the method of \textit{desingularization},
which is to resolve all singularities of $\zeta_r((s_j);(\gamma_j))$.
We will construct a function $\zeta^{\rm des}_r((s_j);(\gamma_j))$,
called \textit{the desingularized multiple zeta-function}
(Definition \ref{def-MZF-2}), which will be shown to be entire 
(Theorem \ref{T-c-1-zeta}). 
Furthermore  we will show that $\zeta^{\rm des}_r((s_j);(\gamma_j))$ can be expressed 
as a finite \lq linear' combination of $\zeta_r((s_j+m_j);(\gamma_j))$,
the multiple zeta-functions of the generalized Euler-Zagier type
whose arguments are appropriately shifted by integers $m_j$ 
$(1\leqslant j\leqslant r)$ (Theorem \ref{Th-ex}).
It is where we stress that
summing up suitable {\it finite} combinations of $\zeta_r((s_j+m_j);(\gamma_j))$ causes 
that a marvelous cancellation of all of their  {\it infinitely} many 
singular hyperplanes 
occurs and consequently $\zeta^{\rm des}_r((s_j);(\gamma_j))$ turns to be entire. 
We will  also describe the special values of $\zeta^{\rm des}_r((s_j);(\gamma_j))$ 
at non-positive integers 
in terms of ordinary Bernoulli numbers (Theorem \ref{C-Zr}).

In Section \ref{sec-3}, 
we will construct the \textit{$p$-adic $r$-ple $L$-function} 
$L_{p,r}((s_j);(\omega^{k_j});(\gamma_j);c)$ 
associated  with $\gamma_j\in \mathbb Z_p$, $k_j\in\mathbb Z$
$(1\leqslant j\leqslant r)$ and a positive integer $c(\geqslant 2)$ with $(c,p)=1$
(Definition \ref{Def-pMLF}).
The case $r=1$ essentially coincides with the Kubota-Leopoldt $p$-adic $L$-function 
(Example \ref{example for r=1}). 
Also the case $(r,c)=(2,2)$ coincides with the $p$-adic double $L$-function introduced in \cite{KMT-IJNT} as mentioned above
(Example \ref{example for r=2}).
Our main technique of construction is due to Koblitz \cite{Kob79}. 
By using a specific $p$-adic measure, we will define $L_{p,r}((s_j);(\omega^{k_j});(\gamma_j);c)$ as a multiple version of the $p$-adic $\Gamma$-transform that provides a $p$-adic analyticity in $(s_j)$ (Theorem \ref{Th-pMLF}). 
We will further show a non-trivial fact that the map $c\mapsto L_{p,r}((s_j);(\omega^{k_j});(\gamma_j);c)$ can be continuously extended to any $p$-adic integer $c$ 
as a $p$-adic continuous function (Theorem \ref{continuity theorem}).

In Section \ref{sec-4}, we will describe the values of $L_{p,r}((s_j);(\omega^{k_j});(\gamma_j);c)$ at non-positive integers as a sum of $\{\aa((n_j);( \xi_j);( \gamma_j))\}$ 
(Theorem \ref{T-main-1}). 
We will see that our $L_{p,r}((s_j);(\omega^{k_j});(\gamma_j);c)$ is a $p$-adic interpolation of
a certain sum \eqref{Int-P} of the complex multiple zeta-functions  $\zeta_r((s_j);(\xi_j);(\gamma_j))$
(Remark \ref{rem-intpln}).
As an application of the theorem, we will obtain a multiple version of the Kummer congruences (Theorem \ref{Th-Kummer}). 
In fact, the case $r=1$ coincides with the ordinary Kummer congruences for
Bernoulli numbers. 
Also we will give some functional relations with a parity condition among $p$-adic multiple $L$-functions (Theorem \ref{T-6-1}). 
The functional relations can be regarded as multiple versions of the well-known fact that $L_p(s;\omega^{2k+1})$ is the zero-function (see Example \ref{Rem-zero-2}). 
We will see that they also recover 
the functional relations shown in \cite{KMT-IJNT} as a special case $(r,c)=(2,2)$
(Example \ref{P-Func-eq-1}).

In Section \ref{sec-5},  we will describe the special values of
$L_{p,r}((s_j);(\omega^{k_j});(1);c)$ at positive integers.
In Theorem \ref{L-Li theorem-2}
we will establish their close  relation to those
of \textit{$p$-adic TMPL's} 
$Li^{(p)}_{n_1,\dots,n_r}(\xi_1,\dots,\xi_r ;z)$
(cf. Definition \ref{Def-TMPL})
at roots of unity,
which is an extension of the previous result of Coleman \cite{C}.
For this aim, we will introduce \textit{$p$-adic rigid TMPL's}
$\ell^{(p)}_{n_1,\dots,n_r}(\xi_1,\dots,\xi_{r};z)$
and 
\textit{$p$-adic partial TMPL's}
$\ell^{\equiv (\alpha_1,\dots,\alpha_r),(p)}_{n_1,\dots,n_r}(\xi_1,\dots,\xi_{r};z)$
(Definition \ref{def of pMMPL} and  \ref{def of pPMPL} respectively)
as in-between functions.
In Subsection \ref{sec-5-3}
the above special values at positive integers
will be shown to be connected with  the special values of $p$-adic rigid
TMPL's at roots unity
(Theorem \ref{L-ell theorem}).
Basic properties of these in-between functions will be presented in Subsection \ref{sec-5-4}.
We will show
an explicit relationship between $p$-adic rigid TMPL's
and $p$-adic TMPL's (Theorem \ref{ell-Li theorem})
by transmitting through their connections with  $p$-adic partial TMPL's
to obtain Theorem \ref{L-Li theorem-2}
in Subsection \ref{sec-5-5}.

\ 

\section{Complex multiple zeta-functions}\label{sec-2}

In this section, we will first recall the analytic properties of complex multiple zeta-functions of Euler-Zagier type \eqref{MZF-def}
and of the zeta-function of Lerch type \eqref{Lerch-zeta}
which interpolates the twisted Bernoulli numbers \eqref{def-tw-Ber}.
Next we will introduce multiple
twisted Bernoulli numbers (Definition \ref{Def-M-Bern}) which are connected with
multiple zeta-functions of the generalized Euler-Zagier-Lerch type \eqref{Barnes-Lerch}.
It will be shown that
the functions in the non-unity case \eqref{non-unity assumption}
are analytically continued to the whole space as entire functions
and
interpolate these numbers at non-positive integers (Theorem \ref{T-multiple}).
In the final subsection, 
we will develop our method of desingularization  (Definition \ref{def-MZF-2}) of 
multiple zeta-functions of  the generalized Euler-Zagier type
\eqref{gene-EZ}.
They are meromorphically continued to the whole space
with  whose singularities lying on infinitely many hyperplanes.
Our desingularization is a method to reduce them into entire functions (Theorem \ref{T-c-1-zeta}).
We will further show that the desingularized
functions are given by
a suitable finite `linear' combination of multiple zeta-functions
\eqref{gene-EZ} with some arguments shifted (Theorem \ref{Th-ex}).
It is where we see 
a miraculous cancellation of all of their {\it infinitely} many 
singular hyperplanes 
occurring there
by taking a suitable {\it finite} combination of these functions.
We will prove that certain combinations of Bernoulli numbers attain the
special values at their non-positive integers of the desingularized
functions (Theorem \ref{C-Zr}).
Our method might be said as a multiple series analogue of
the procedure reducing the Riemann zeta function $\zeta(s)$
into the entire function $(s-1)\zeta(s)$ (Example \ref{Exam-SH}).
These observations lead to the construction of $p$-adic multiple $L$-functions 
which will be discussed in the next section.

\subsection{Basic facts}\label{sec-2-1}

Let $\mathbb{N}$, $\mathbb{N}_0$, $\mathbb{Z}$, $\mathbb{Q}$, $\mathbb{R}$ and $\mathbb{C}$ be the set of natural numbers, non-negative integers, rational integers, rational numbers, real numbers and complex numbers, respectively. Let $\overline{\mathbb{Q}}$ be the algebraic closure of $\mathbb{Q}$. For $s\in \mathbb{C}$, denote by $\Re s$ and $\Im s$ the real and the imaginary parts of $s$, respectively.

Let $\chi$ be a primitive Dirichlet character and denote the conductor of $\chi$ by $f_\chi$. The Dirichlet $L$-function associated with $\chi$ is defined by 
$$L(s,\chi)=\sum_{m=1}^\infty \frac{\chi(m)}{m^s}.$$
In the case $\chi=\chi_0$, namely the trivial character with $f_{\chi_0}=1$, $L(s,\chi_0)$ is equal to $\zeta(s)$. 

It is well-known that $L(s,\chi)$ is an entire function when $\chi\not=\chi_0$, and $\zeta(s)$ is a meromorphic function on $\mathbb{C}$ with a simple pole at $s=1$, and satisfies
\begin{equation}
\begin{split}
&\zeta(1-k)=
\begin{cases} 
-\frac{B_k}{k} & (k\in \mathbb{N}_{>1})\\
-\frac{1}{2}  & (k=1),
\end{cases}
\\
&L(1-k,\chi)=-\frac{B_{k,\chi}}{k}\quad (k\in \mathbb{N};\ \chi\not=\chi_0),
\end{split}
\label{1-1-2}
\end{equation}
where $\{B_n\}$ and $\{B_{n,\chi}\}$ are the Bernoulli numbers
\footnote{                                                                               
or better to be called Seki-Bernoulli numbers, because
Takakazu (Kowa) Seki published the work on these numbers, independently,
before Jakob Bernoulli.
}
and the generalized Bernoulli numbers associated with $\chi$ defined by 
\begin{align*}
& \frac{t}{e^t-1}=\sum_{n=0}^\infty B_n\frac{t^n}{n!},\\
& \sum_{a=1}^{f} \frac{\chi(a)te^{at}}{e^{f t}-1}=\sum_{n=0}^\infty B_{n,\chi}\frac{t^n}{n!}\qquad (f=f_\chi),
\end{align*}
respectively (see \cite[Theorem 4.2]{Wa}). Note that $B_{n,\chi_0}=B_n$ ($n\in \mathbb{N}_0$) except for $B_{1,\chi_0}=-B_1=\frac{1}{2}$. 


The multiple zeta-function of Euler-Zagier type is defined by \eqref{MZF-def}.
As was mentioned in the Introduction, the research of \eqref{MZF-def} goes back to
a paper of Euler.
In the late 1990s, several authors 
investigated its analytic properties, 
though their results have not been published (for the details, see the survey article \cite{M2010}). 
In the early 2000s, Zhao \cite{Zh2000} and Akiyama, Egami and Tanigawa \cite{AET} independently showed that the multiple zeta-function \eqref{MZF-def} can be meromorphically continued to $\mathbb{C}^r$. Furthermore, 
the \textit{exact} locations of singularities of \eqref{MZF-def} were explicitly determined as follows.

\begin{theorem}[{\cite[Theorem 1]{AET}}]\label{T-AET}
The multiple zeta-function \eqref{MZF-def} can be meromorphically continued to $\mathbb{C}^r$ with infinitely many 
singular hyperplanes
\begin{align}
& s_r=1,\quad s_{r-1}+s_{r}=2,1,0,-2,-4,-6, \ldots, \notag\\
& s_{r-k+1}+s_{r-k+2}+\cdots+s_r=k-n\quad (3\leqslant k\leqslant r,\ n\in \mathbb{N}_0).\label{EZ-sing}
\end{align}
\end{theorem}

Multiple zeta values, namely
the special values of  \eqref{MZF-def}
at positive integers,
are equal to the special values of multiple polylogarithm
$$
Li_{n_1,\dots,n_r}(z_1,\dots,z_r):=
{\underset{0<k_1<\cdots<k_{r}}{\sum}}
\frac{z_1^{k_1}\cdots z_r^{k_r}}{k_1^{n_1}\cdots k_r^{n_r}}
$$
(which is an $r$-variable complex analytic function converging on polyunit disk)
at unity, namely
\begin{equation}\label{zeta=Li}
\zeta_r(n_1,\dots,n_r)=Li_{n_1,\dots,n_r}(1,\dots,1)
\end{equation}
for $n_1,\dots,n_r\in {\mathbb N}$ with $n_r>1$.

In the last section of this paper 
a $p$-adic analogue of the equality \eqref{zeta=Li}
will be attained.

As stated above, the multiple zeta-function 
\eqref{MZF-def} but for the single variable case ($r=1$) has infinitely many singular 
hyperplanes and almost all non-positive integer points lie there.
It causes an indeterminacy 
of its special values at non-positive integers.
For example, according to \cite{AET,AT},
\begin{align*}
\lim_{\varepsilon_1\to 0}\lim_{\varepsilon_2\to 0}\zeta_2(\varepsilon_1,\varepsilon_2)&=\frac{1}{3},\\
\lim_{\varepsilon_2\to 0}\lim_{\varepsilon_1\to 0}\zeta_2(\varepsilon_1,\varepsilon_2)&=\frac{5}{12},\\
\lim_{\varepsilon\to 0}\zeta_2(\varepsilon,\varepsilon)&=\frac{3}{8}.
\end{align*}

There are some other explicit formulas for the values at those
non-positive integer points as limit values when the way of  
approaching those points are fixed (\cite{AET}, \cite{Ko2010},
\cite{Onozuka} et al.)

\subsection{Twisted multiple Bernoulli numbers}\label{sec-2-2}
We will review Koblitz' definition of twisted Bernoulli numbers.
Then we  will introduce twisted multiple Bernoulli numbers,
their multiple analogue, in Definition \ref{Def-M-Bern}
and investigate their expression  as combinations of twisted Bernoulli numbers
in Proposition \ref{prop-M-Bern}.

\begin{definition}[{\cite[p.\,456]{Kob79}}]
For any root of unity $\xi$, we define the {\bf twisted Bernoulli numbers} $\{ \aa_n(\xi)\}$ by
\begin{equation}
\hc(t;\xi)=\frac{1}{1-\xi e^{t}}=\sum_{n=-1}^\infty \aa_n(\xi)\frac{t^n}{n!}, \label{def-tw-Ber}
\end{equation}
where we formally let $(-1)!=1$.
\end{definition}

\begin{remark}
Koblitz \cite{Kob79} generally defined the twisted Bernoulli numbers associated with primitive 
Dirichlet characters. 
The above $\{\aa_n(\xi)\}$ correspond to $\chi_0$. 
\end{remark}

In the case $\xi=1$, we have 
\begin{equation}
\aa_{-1}(1)=-1,\qquad \aa_n(1)=-\frac{B_{n+1}}{n+1}\quad (n\in \mathbb{N}_0). \label{Ber-01}
\end{equation}
In the case $\xi\not=1$, 
we have $\aa_{-1}(\xi)=0$ and 
$\aa_n(\xi)=\frac{1}{1-\xi}H_n\left(\xi^{-1}\right)$ $(n\in \mathbb{N}_0)$, where $\{H_n(\lambda)\}_{n\geqslant 0}$ are what is called the Frobenius-Euler numbers associated with $\lambda$ defined by
$$\frac{1-\lambda}{e^t-\lambda}=\sum_{n=0}^\infty H_n(\lambda)\frac{t^n}{n!}$$
(see Frobenius \cite{Fro}). 
We obtain from \eqref{def-tw-Ber} that $\aa_n(\xi)\in \mathbb{Q}(\xi)$. For example, 
\begin{equation}
\begin{split}
& \aa_0(\xi)=\frac{1}{1-\xi},\quad \aa_1(\xi)=\frac{\xi}{(1-\xi)^2},\quad \aa_2(\xi)=\frac{\xi(\xi+1)}{(1-\xi)^3},\\
& \aa_3(\xi)=\frac{\xi(\xi^2+4\xi+1)}{(1-\xi)^4},\quad \aa_4(\xi)=\frac{\xi(\xi^3+11\xi^2+11\xi+1)}{(1-\xi)^5},\ldots
\end{split}
 \label{TBN-exam}
\end{equation}

Let $\mu_k$ be the group of $k$th roots of unity. 
Using the relation
\begin{equation}
\frac{1}{X-1}-\frac{k}{X^{k}-1}=\sum_{\xi\in \mu_k\atop \xi\not=1}\frac{1}{1-\xi X}\qquad (k\in \mathbb{N}_{>1}) \label{log-der}
\end{equation}
for an indeterminate $X$, 
we obtain the following.

\begin{proposition}
Let $c\in \mathbb{N}_{>1}$. For $n\in \mathbb{N}_0$, 
\begin{equation}
\left(1-c^{n+1}\right)\frac{B_{n+1}}{n+1}=\sum_{\xi^c=1\atop \xi\not=1}\aa_n(\xi). \label{2-0-1}
\end{equation}
\end{proposition}

\begin{remark}
Let $\xi$ be a root of unity. 
As an analogue of \eqref{1-1-2}, it holds that
\begin{equation}
\phi(-k;\xi)=\aa_k(\xi)\quad (k\in \mathbb{N}_0),\label{phi-val}
\end{equation}
where $\phi(s;\xi)$ is the {\bf zeta-function of Lerch type}
defined by the meromorphic continuation of the series
\begin{equation}
\phi(s;\xi)=\sum_{m\geqslant 1}\xi^{m}m^{-s} \qquad (\Re s>1)
\label{Lerch-zeta}
\end{equation}
(cf.\ {\cite[Chapter 2,\,Section 1]{Kob}}).
\end{remark}

We see that \eqref{2-0-1} can also be given from the relation 
\begin{equation}
\left(c^{1-s}-1\right)\zeta(s)=\sum_{\xi^c=1\atop \xi\not=1} \phi(s;\xi). \label{rel-phi}
\end{equation}

Now we define certain multiple analogues of twisted Bernoulli numbers. 

\begin{definition}\label{Def-M-Bern}
Let $r\in \mathbb{N}$, 
$\gamma_1,\ldots,\gamma_r\in \mathbb{C}$ 
and let $\xi_1,\ldots,\xi_r\in \mathbb{C}$ be roots of unity. 
Set 
\begin{align}
    \mathfrak{H}_r  (( t_j );( \xi_j); ( \gamma_j))&:=
    \prod_{j=1}^{r} \mathfrak{H}(\gamma_j (\sum_{k=j}^r t_k);\xi_j)=\prod_{j=1}^{r} \frac{1}{1-\xi_j \exp\left(\gamma_j \sum_{k=j}^r t_k\right)}\label{Def-Hr}
\end{align}
and define
{\bf twisted multiple Bernoulli numbers}
\begin{footnote}{We are not sure which is better, ``twisted multiple'', or ``multiple twisted''. 
But we will skip this problem because
it looks that these two adjectives are ``commutative'' here.}
\end{footnote}$\{\aa(n_1,\ldots,n_r;( \xi_j);( \gamma_j))\}$
by 
\begin{align}
    \mathfrak{H}_r  (( t_j );( \xi_j); ( \gamma_j))
    =\sum_{n_1=-1}^\infty
    \cdots
    \sum_{n_r=-1}^\infty
    \aa(n_1,\ldots,n_r;( \xi_j);( \gamma_j))
    \frac{t_1^{n_1}}{n_1!}
    \cdots
    \frac{t_r^{n_r}}{n_r!},
\label{Fro-def-r}
\end{align}
where we note $(-1)!=1$ as mentioned before. 
In the case $r=1$, we have $\aa_n(\xi_1)=\aa(n;\xi_1;1)$. Note that if $\xi_j\not=1$ $(1\leqslant j\leqslant r)$ then $\mathfrak{H}_r  (( t_j );( \xi_j); ( \gamma_j))$ is holomorphic around the origin with respect to the parameters $t_1,\dots, t_r$,
hence the singular part does not appear on the right-hand side of \eqref{Fro-def-r}.
\end{definition}

We immediately obtain the following from \eqref{def-tw-Ber}, \eqref{Def-Hr} and \eqref{Fro-def-r}.

\begin{proposition}\label{prop-M-Bern}
Let $\gamma_1,\ldots,\gamma_r\in \mathbb{C}$ 
and $\xi_1,\ldots,\xi_r\in \mathbb{C}$ be roots of unity. Then 
$\aa(n_1,\ldots,n_r;( \xi_j);( \gamma_j))$ can be expressed as a polynomial in $\{ \aa_{n}(\xi_j)\,|\,1\leqslant j\leqslant r,~{n\geqslant 0}\}$ and $\{\gamma_1,\ldots,\gamma_r\}$ with $\mathbb{Q}$-coefficients, that is, a rational function in $\{\xi_j\}$ and $\{\gamma_j\}$ with $\mathbb{Q}$-coefficients. 
\end{proposition}

\begin{example}\label{Exam-DH}
We consider the case $r=2$ and $\xi_j\neq 1$ $(j=1,2)$. 
Substituting \eqref{def-tw-Ber} into \eqref{Def-Hr} in the case $r=2$, we have
\begin{align*}
& \mathfrak{H}_2(t_1,t_2;\xi_1,\xi_2;\gamma_1,\gamma_2)=\frac{1}{1-\xi_1 \exp\left(\gamma_1(t_1+t_2)\right)}\frac{1}{1-\xi_2 \exp\left(\gamma_2 t_2\right)}\\
    &=\left(\sum_{m=0}^\infty \aa_m(\xi_1)\frac{\gamma_1^m (t_1+t_2)^m}{m!}\right) \left(\sum_{n=0}^\infty \aa_n(\xi_2)\frac{\gamma_2^n t_2^n}{n!}\right)
    \\
    &=\sum_{m=0}^\infty\sum_{n=0}^\infty \aa_m(\xi_1)\aa_n(\xi_2)\left(\sum_{k,j\geqslant 0 \atop k+j=m}\frac{t_1^k t_2^j}{k!j!}\right)\gamma_1^m\gamma_2^n\frac{t_2^n}{n!}.
\end{align*}
Putting $l:=n+j$, we have 
\begin{align*}
\mathfrak{H}_2(t_1,t_2;\xi_1,\xi_2;\gamma_1,\gamma_2)
    &=\sum_{k=0}^\infty\sum_{l=0}^\infty \sum_{j=0}^{l}\binom{l}{j}\aa_{k+j}(\xi_1)\aa_{l-j}(\xi_2)\gamma_1^{k+j}\gamma_2^{l-j}\frac{t_1^k}{k!}\frac{t_2^l}{l!},
\end{align*}
which gives 
\begin{equation}
\begin{split}
\aa(k,l;\xi_1,\xi_2;\gamma_1,\gamma_2)&=\sum_{j=0}^{l}\binom{l}{j}\aa_{k+j}(\xi_1)\aa_{l-j}(\xi_2)\gamma_1^{k+j}\gamma_2^{l-j}\quad (k,l\in \mathbb{N}_0).
\end{split}
\label{Eur-exp1}
\end{equation}
For example, we can obtain from \eqref{TBN-exam} that 
\begin{align*}
&\aa(0,0;\xi_1,\xi_2;\gamma_1,\gamma_2)=\frac{1}{(1-\xi_1)(1-\xi_2)},\quad \aa(1,0;\xi_1,\xi_2;\gamma_1,\gamma_2)=\frac{\xi_1\gamma_1}{(1-\xi_1)^2(1-\xi_2)},\\
&\aa(0,1;\xi_1,\xi_2;\gamma_1,\gamma_2)=\frac{\xi_1\gamma_1+\xi_2\gamma_2-\xi_1\xi_2(\gamma_1+\gamma_2)}{(1-\xi_1)^2(1-\xi_2)^2},\\
&\aa(1,1;\xi_1,\xi_2;\gamma_1,\gamma_2)=\frac{\xi_1^2\gamma_1(\gamma_1-\xi_2(\gamma_1+\gamma_2))+\xi_1\gamma_1(\gamma_1-\xi_2(\gamma_1-\gamma_2))}{(1-\xi_1)^3(1-\xi_2)^2},\ldots
\end{align*}
\end{example}

The following series will be treated in our desingularization method in subsection \ref{c-1-zeta}.

\begin{definition}\label{define-tilde-H}
For $c\in \mathbb{R}$ and 
$\gamma_1,\ldots,\gamma_r\in \mathbb{C}$ with $\Re \gamma_j >0 \ (1\leqslant j\leqslant r)$, define
\begin{align}
\widetilde{\mathfrak{H}}_r  (( t_j ); ( \gamma_j);c)
& =\prod_{j=1}^{r} \left( \frac{1}{\exp\left(\gamma_j \sum_{k=j}^r t_k\right)-1}-\frac{c}{\exp\left(c\gamma_j \sum_{k=j}^r t_k\right)-1}\right)\notag\\
& =\prod_{j=1}^{r} \left(\sum_{m=1}^\infty  \left(1-c^m\right)B_m\frac{\left(\gamma_j \sum_{k=j}^r t_k\right)^{m-1}}{m!}\right).\label{def-tilde-H}
\end{align}
In particular when $c\in \mathbb{N}_{>1}$, by use of \eqref{log-der}, we have
\begin{align}
\widetilde{\mathfrak{H}}_r  (( t_j ); ( \gamma_j);c)&=\prod_{j=1}^{r} \sum_{\xi_j^c=1 \atop \xi_j\not=1}\frac{1}{1-\xi_j \exp\left(\gamma_j \sum_{k=j}^r t_k\right)}\notag\\
& =\sum_{\xi_1^c=1 \atop \xi_1\not=1}\cdots \sum_{\xi_r^c=1 \atop \xi_r\not=1}\mathfrak{H}_r  (( t_j );( \xi_j); ( \gamma_j)).\label{tilde-H}
\end{align}
\end{definition}

\begin{remark}
We note that $\widetilde{\mathfrak{H}}_r  (( t_j ); ( \gamma_j);c)$ 
is holomorphic around the origin with respect to the parameters $(t_j)$,
and tends to $0$ as $c \to 1$. 
We also note that 
the Bernoulli numbers appear in the Maclaurin expansion of the limit
\begin{equation*}
\lim_{c\to 1}\frac{1}{(c-1)^r}\widetilde{\mathfrak{H}}_r  (( t_j ); ( \gamma_j);c).\label{limit-H}
\end{equation*}
These are important points in our arguments  on desingularization methods
developed in  Subsection \ref{c-1-zeta}.
\end{remark}

\begin{example}\label{Exam-DH-2}
Similarly to Example \ref{Exam-DH}, we obtain from
\eqref{def-tilde-H} 
with any $c\in \mathbb{R}$ that
\begin{align}
& \widetilde{\mathfrak{H}}_2  (t_1,t_2; \gamma_1,\gamma_2;c) \notag\\
& =\sum_{k,l=0}^\infty \left\{\sum_{j=0}^{l}\binom{l}{j}\left(1-c^{k+j+1}\right)\left(1-c^{l-j+1}\right)\frac{B_{k+j+1}}{k+j+1}\frac{B_{l-j+1}}{l-j+1}\gamma_1^{k+j}\gamma_2^{l-j}\right\} \frac{t_1^kt_2^l}{k! l!}. \label{Convo-Bern}
\end{align}
Therefore it follows from \eqref{Fro-def-r} and \eqref{tilde-H} that
\begin{equation}
\begin{split}
&\sum_{\xi_1\in \mu_c\atop \xi_1\not=1}\sum_{\xi_2\in \mu_c\atop \xi_2\not=1}\aa(k,l;\xi_1,\xi_2;\gamma_1,\gamma_2)\\
&\quad =\sum_{j=0}^{l}\binom{l}{j}\left(1-c^{k+j+1}\right)\left(1-c^{l-j+1}\right)\frac{B_{k+j+1}}{k+j+1}\frac{B_{l-j+1}}{l-j+1}\gamma_1^{k+j}\gamma_2^{l-j}\quad (k,l\in \mathbb{N}_0)
\end{split}
\label{Eur-exp2}
\end{equation}
for $c\in \mathbb{N}_{>1}$. 
\end{example}


\begin{remark}\label{poly-Bernoulli}
Kaneko \cite{Kaneko1997} defined the poly-Bernoulli numbers $\{B_n^{(k)}\}_{n\in \mathbb{N}_0}$ $(k\in \mathbb{Z})$ by use of the polylogarithm of order $k$. 
Explicit relations between twisted multiple Bernoulli numbers and poly-Bernoulli numbers are not clearly known. 
It is noted that, for example, 
$$  B_l^{(2)} = \sum_{j=0}^l \binom{l}{j} \frac{B_{l-j}B_j}{j+1} \quad (l\in \mathbb{N}_0),$$
which resembles \eqref{Eur-exp1} and \eqref{Eur-exp2}.
\end{remark}

\subsection{Multiple zeta-functions}\label{MZF}

Corresponding to the twisted multiple Bernoulli numbers
$\{\aa((n_j);( \xi_j);( \gamma_j))\}$ is the
multiple zeta-function of the generalized Euler-Zagier-Lerch type
\eqref{Barnes-Lerch} defined in the Introduction, 
which is a multiple analogue of $\phi(s;\xi)$.
This function can be continued analytically to the
whole space and interpolates $\aa((n_j);( \xi_j);( \gamma_j))$ at non-positive integers (Theorem \ref{T-multiple}). 

Assume $\xi_j \neq 1$ $(1\leqslant j \leqslant r)$. 
Using the well-known relation 
\begin{equation*}
  u^{-s}=\frac{1}{\Gamma(s)}\int_0^\infty e^{-ut}{t^{s-1}} dt,
\end{equation*}
we obtain
  \begin{align}
    & \zeta_r((s_j);(\xi_j);(\gamma_j))\notag\\
    &=
    \sum_{\substack{m_1=1}}^\infty\cdots \sum_{\substack{m_r=1}}^\infty
    \Bigl(    
    \prod_{j=1}^r \xi_j^{ m_j}\Bigr) 
    \left(\prod_{k=1}^r
    \frac{1}{\Gamma(s_k)}\right)\int_{[0,\infty)^{r}}
    \prod_{k=1}^r
    \exp(-t_k(\sum_{\substack{j\leqslant k}} m_j\gamma_j)) 
    \prod_{k=1}^r t_k^{s_k-1}dt_k\notag
    \\
    &=
    \left(\prod_{k=1}^r \frac{1}{\Gamma(s_k)}\right)\int_{[0,\infty)^{r}}\prod_{j=1}^r
    \frac{\xi_j \exp(-\gamma_j (\sum_{k=j}^r t_k))}{1-\xi_j \exp(-\gamma_j (\sum_{k=j}^r t_k))}
    \prod_{k=1}^r t_k^{s_k-1}dt_k\notag
    \\
    &=
    \left(\prod_{k=1}^r \frac{1}{(e^{2\pi i s_k}-1)\Gamma(s_k)}\right)\int_{\mathcal{C}^{r}}
    \prod_{j=1}^r 
    \frac{\xi_j \exp(-\gamma_j (\sum_{k=j}^r t_k))}{1-\xi_j \exp(-\gamma_j (\sum_{k=j}^r t_k))}
    \prod_{k=1}^r t_k^{s_k-1}dt_k\notag\\
    &=(-1)^r 
    \left(\prod_{k=1}^r \frac{1}{(e^{2\pi i s_k}-1)\Gamma(s_k)}\right)\int_{\mathcal{C}^{r}}
    \mathfrak{H}_r  (( t_j ),( \xi_j^{-1}), ( \gamma_j))\prod_{k=1}^r t_k^{s_k-1}dt_k,\label{Cont}
  \end{align}
where $\mathcal{C}$ is the Hankel contour, that is, the path consisting of the positive real axis (top side), a circle around the origin of radius $\varepsilon$ (sufficiently small), and the positive real axis (bottom side) 
(see  {\sc{Figure}} \ref{fig:1}).
Note that the third equality holds because we can let $\varepsilon \to 0$ 
on the fourth member of \eqref{Cont}. In fact, 
the integrand of the fourth member is holomorphic around the origin with respect to the parameters $(t_j)$ because of $\xi_j \neq 1$ $(1\leqslant j \leqslant r)$. Here we can easily show that the integral on the last member of \eqref{Cont} is absolutely convergent in a usual manner with respect to the Hankel contour. Hence 
we obtain the following.
\begin{figure}
  \centering
  \includegraphics[bb=0 0 113 73]{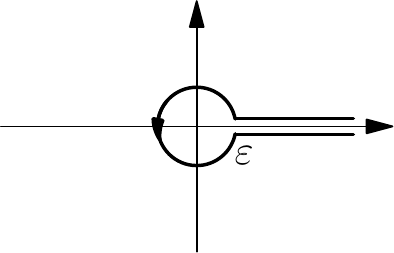}
  \caption{The Hankel contour $\mathcal{C}$}
  \label{fig:1}
\end{figure}

\begin{theorem}\label{T-multiple}
Let $\xi_1,\ldots,\xi_r\in \mathbb{C}$ 
be roots of unity and 
$\gamma_1,\ldots,\gamma_r\in \mathbb{C}$ with $\Re \gamma_j >0 \ (1\leqslant j\leqslant r)$. 
Assume that 
\begin{equation}\label{non-unity assumption}
\xi_j\neq 1  \quad\text{ for all } j \ (1\leqslant j\leqslant r).
\end{equation}
Then, with the above notation, 
$\zeta_r((s_j);(\xi_j);(\gamma_j))$ can be analytically continued to $\mathbb{C}^r$ as an entire function in $(s_j)$. For $n_1,\ldots,n_r\in \mathbb{N}_0$, 
\begin{equation}
\zeta_r((-n_j);(\xi_j);(\gamma_j))=(-1)^{r+n_1+\cdots+n_r}\aa((n_j);( \xi_j^{-1});( \gamma_j)). \label{multi-val}
\end{equation}
\end{theorem}

\begin{proof}
Since the contour integral on the right-hand side of \eqref{Cont} is holomorphic for all $(s_k)\in \mathbb{C}^r$, we see that $\zeta_r((s_j);(\xi_j);(\gamma_j))$ can be meromorphically continued to $\mathbb{C}^r$ and its possible singularities are located on hyperplanes $s_k=l_k\in \mathbb{N}$ $(1\leqslant k \leqslant r)$ outside of the region of convergence because $(e^{2\pi i s_k}-1)\Gamma(s_k)$ does not vanish at $s_k\in \mathbb{Z}_{\leqslant 0}$. Furthermore, for $s_k=l_k\in \mathbb{N}$, the integrand of the contour integral with respect to $t_k$ on the last member of \eqref{Cont} is holomorphic around $t_k=0$. Therefore, for $l_k\in \mathbb{N}$, we see that
\begin{align*}
    &\lim_{s_k\to l_k}\int_{\mathcal{C}}
    \mathfrak{H}_r  (( t_j ),( \xi_j^{-1}), ( \gamma_j))
    t_k^{s_k-1}dt_k
    =\int_{C_\varepsilon}
\mathfrak{H}_r  (( t_j ),( \xi_j^{-1}), ( \gamma_j))
    t_k^{l_k-1}dt_k
     =0,
\end{align*}
because of the residue theorem, where $C_\varepsilon=\{\varepsilon e^{i\theta}\,|\,0\leqslant \theta\leqslant 2\pi\}$ for any sufficiently small $\varepsilon$. Consequently this implies that $\zeta_r((s_j);(\xi_j);(\gamma_j))$ has no singularity on $s_k=l_k$, namely $\zeta_r((s_j);(\xi_j);(\gamma_j))$ is entire. 
Finally, substituting \eqref{Fro-def-r} into \eqref{Cont}, setting $(s_j)=(-n_j)$ and 
using
\begin{equation*}
  \lim_{s\to -n}\frac{1}{(e^{2\pi is}-1)\Gamma(s)}=\frac{(-1)^n n!}{2\pi i}\quad (n\in \mathbb{N}_0),
\end{equation*}
we obtain \eqref{multi-val}. 
Thus we complete the proof of Theorem \ref{T-multiple}.
\end{proof}

The partial cases of Theorem \ref{T-multiple} 
can be recovered by the special cases of the results in 
Matsumoto-Tanigawa \cite{MT2003}, Matsumoto-Tsumura \cite{MaTsAA}, and de Crisenoy \cite{Cr2006}.

In \cite[Theorem 1]{MaJNT}, it is shown that the multiple zeta-function $\zeta_r((s_j);(\xi_j);(\gamma_j))$ 
of  the generalized Euler-Zagier-Lerch type \eqref{Barnes-Lerch} with all $\xi_j=1$ 
is meromorphically continued to the whole space $\mathbb C^r$
with \textit{possible} singularities. 
A more general type of multiple zeta-function is 
treated in \cite{Ko2010}, 
where equation \eqref{multi-val} without the assertion of being an entire function 
is shown in the case of 
$\xi_j\neq 1$ for all $j$ 
and the meromorphic continuation of 
$\zeta_r((s_j);(\xi_j);(\gamma_j))$ is also given. 



\begin{remark}\label{Rem-forthcoming}
Without the assumption \eqref{non-unity assumption},
it should be noted that \eqref{Cont} does not hold generally, more strictly the third equality on the right-hand side does not hold because the Hankel contours necessarily 
cross the singularities of the integrand. 
\end{remark}

In the forthcoming paper \cite{FKMT02}, we will show 
the necessary and sufficient condition that 
$\zeta_r((s_j);(\xi_j);(\gamma_j))$ is entire, and will determine the exact locations of singularities when it is not entire:

\begin{theorem}\label{Forth-coming} 
Let $\xi_1,\ldots,\xi_r\in \mathbb{C}$ 
be roots of unity and 
$\gamma_1,\ldots,\gamma_r\in \mathbb{C}$ with $\Re \gamma_j >0 \ (1\leqslant j\leqslant r)$. Then 
$\zeta_r((s_j);(\xi_j);(\gamma_j))$ can be entire if and only if the condition \eqref{non-unity assumption} holds. 
When it is not entire, one of the following cases occurs:
\begin{enumerate}[{\rm (i)}]
\item The function $\zeta_r((s_j);(\xi_j);(\gamma_j))$ has infinitely many simple singular hyperplanes when $\xi_j=1$ for some $j$ $(1\leqslant j \leqslant r-1)$.
\item The function $\zeta_r((s_j);(\xi_j);(\gamma_j))$ has a unique simple singular hyperplane $s_r=1$ when $\xi_j\neq 1$ for all $j$ $(1\leqslant j \leqslant r-1)$ and $\xi_r=1$.
\end{enumerate}
\end{theorem}

\subsection{Desingularization of multiple zeta-functions
}\label{c-1-zeta}
In this subsection we introduce and develop our method of  desingularization.
In our previous subsection we saw that 
the multiple zeta-function 
$\zeta_r((s_j);(\gamma_j))$
of  the generalized Euler-Zagier type \eqref{gene-EZ}
is meromorphically  continued to the whole space with `true' singularities 
whilst the multiple zeta function $\zeta_r((s_j);(\xi_j);(\gamma_j))$
of  the generalized Euler-Zagier-Lerch type \eqref{Barnes-Lerch} 
under the non-unity assumption \eqref{non-unity assumption} 
is analytically continued to $\mathbb C^r$ as an entire function.
Our desingularization is a technique to resolve all singularities  of $\zeta_r((s_j);(\gamma_j))$
by making use of the holomorphicity of $\zeta_r((s_j);(\xi_j);(\gamma_j))$
and to produce an entire function  $\zeta^{\rm des}_r((s_j);(\gamma_j))$.
Consider the following expression:
\begin{equation}\label{nonsense equation}
\zeta^{\rm des}_r((s_j);(\gamma_j)):=\lim_{c \to 1}
\frac{1}{(c-1)^r}
\sum_{\xi_1^c=1 \atop \xi_1\not=1}\cdots \sum_{\xi_r^c=1 \atop \xi_r\not=1}
\zeta_r((s_j);(\xi_j);(\gamma_j))
\end{equation}
This is surely nonsense, because $c\in \mathbb{N}_{>1}$ on the right-hand side.
However our fundamental idea is symbolized in this primitive expression. 
Our idea is motivated from a very simple observation 
$$(1-s)\zeta(s)=
\lim_{c\to 1}\,\frac{1}{c-1}\,\left(c^{1-s}-1\right)\zeta(s).
$$
Here on  the left-hand side we find  an entire function $(1-s)\zeta(s)$,
which is merely a product of $(1-s)$ and the meromorphic function $\zeta (s)$
with a simple pole at $s=1$.
While 
on the right hand-side, when $c\in \mathbb{N}_{>1}$,  we may associate a decomposition
\begin{equation*}
\frac{1}{c-1}\,\left(c^{1-s}-1\right)\zeta(s)=
\frac{1}{c-1}\sum_{\xi^c=1 \atop \xi\not=1}\phi(s;\xi)
\end{equation*}
into a sum of entire functions $\phi(s;\xi)=\zeta_1(s;\xi;1)$.

Our desingularization method, 
a rigorous mathematical formulation to
give  a meaning of \eqref{nonsense equation}
will be settled in Definition \ref{def-MZF-2}.
An application of desingularization to the Riemann zeta function $\zeta(s)$
is given in Example \ref{Exam-SH}.
We will see in Theorem \ref{T-c-1-zeta}
that our $\zeta^{\rm des}_r((s_j);(\gamma_j))$  is entire on the whole space ${\mathbb C}^r$. 
We stress that $\zeta_r^{\rm des}((s_j);(\gamma_j))$ is worthy of 
an important object from the viewpoint of the analytic theory of 
multiple zeta-functions. In fact, its values at not only all positive 
or all non-positive integer points but also arbitrary integer points are 
fully determined (see Example \ref{Exam-1-33}).  

Theorem \ref{C-Zr} will prove that suitable combinations of Bernoulli numbers attain the 
special values at non-positive integers of $\zeta^{\rm des}_r((s_j);(\gamma_j))$.
Theorem \ref{Th-ex}, which is the most important theorem in this subsection, 
will reveal that our desingularized  multiple zeta-function $\zeta^{\rm des}_r((s_j);(\gamma_j))$
is actually given by a finite \lq linear' combination of 
the multiple zeta-function $\zeta_r((s_j+m_j);(\gamma_j))$ 
with some arguments appropriately shifted by $m_j\in{\mathbb Z}$ ($1\leqslant j \leqslant r$).
We will see there that {\it infinitely} many 
singular hyperplanes 
of  $\zeta_r((s_j);(\gamma_j))$
are miraculously canceled by taking only  the {\it finite} \lq linear' combination of
appropriately shifted ones.
Example \ref{Z-EZ-double} and Remark \ref{EZ-double}
are our specific observations for double variable case.
%

\begin{definition} \label{def-MZF-2}
For $\gamma_1,\ldots,\gamma_r\in \mathbb{C}$ with $\Re \gamma_j >0 \quad (1\leqslant j\leqslant r)$, the \textbf{desingularized multiple zeta-function},
which we also call the \textbf{desingularization of} $\zeta_r((s_j);(\gamma_j))$,
is defined by
  \begin{align}
    & \zeta^{\rm des}_r((s_j);(\gamma_j)) \notag\\
    &:=\underset{c\in \mathbb{R}\setminus \{1\}}
{\lim_{c\to 1}}\frac{(-1)^r}{(c-1)^r}
    \prod_{k=1}^r \frac{1}{(e^{2\pi i s_k}-1)\Gamma(s_k)}\int_{\mathcal{C}^{r}}
    \widetilde{\mathfrak{H}}_r  (( t_j ); ( \gamma_j);c)\prod_{k=1}^r t_k^{s_k-1}dt_k\label{Cont-2-1} 
\end{align}
for $(s_j)\in \mathbb{C}^r$, 
where $\mathcal{C}$ is the Hankel contour used in \eqref{Cont}. Note that \eqref{Cont-2-1} is well-defined because the convergence of the contour integral and of the limit with respect to $c\to 1$ can be justified from Theorem \ref{T-c-1-zeta} (see below).
\end{definition}

\begin{remark}\label{conceptual idea}
By \eqref{Cont} and \eqref{tilde-H}, we may say that
equation \eqref{Cont-2-1} is a rigorous way to make sense of the nonsense equation
\eqref{nonsense equation}. 
We will discuss 
equation \eqref{nonsense equation}
again in the $p$-adic case (Subsection \ref{sec-3-2}).
\end{remark}

\begin{example}\label{Exam-SH}
In the case $r=1$, set $(r,\gamma_1)=(1,1)$ in \eqref{Cont-2-1}. Similarly to \cite[Theorem 4.2]{Wa}, 
we can easily see that 
\begin{align}
 \zeta^{\rm des}_1(s;1)& =\lim_{c \to 1}\frac{(-1)}{c-1}\cdot 
    \frac{1}{(e^{2\pi i s}-1)\Gamma(s)}\int_{\mathcal{C}}
\left(\frac{1}{e^t-1}-\frac{c}{e^{ct}-1}\right)t^{s-1}dt\notag\\
& =\lim_{c\to 1}\frac{(-1)}{c-1}\left(\zeta(s)-c\sum_{m=1}^\infty \frac{1}{(cm)^s}\right) \notag\\
         & =\lim_{c\to 1}\frac{(-1)}{c-1}\left(1-c^{1-s}\right)\zeta(s)=(1-s)\zeta(s).\label{def-Z1}
\end{align}
Hence $\zeta^{\rm des}_1(s;1)$ can be analytically continued to $\mathbb{C}$.
\end{example}

More generally we can prove the following theorem.

\begin{theorem}\label{T-c-1-zeta}
For $\gamma_1,\ldots,\gamma_r\in \mathbb{C}$ with $\Re \gamma_j >0 \quad (1\leqslant j\leqslant r)$, 
\begin{align}
    & \zeta^{\rm des}_r((s_j);(\gamma_j)) \notag\\
    &=
    \prod_{k=1}^r \frac{1}{(e^{2\pi i s_k}-1)\Gamma(s_k)}\int_{\mathcal{C}^{r}}
    \lim_{c \to 1}\frac{(-1)^r}{(c-1)^r}\widetilde{\mathfrak{H}}_r  (( t_j ); ( \gamma_j);c)\prod_{k=1}^r t_k^{s_k-1}dt_k\notag\\
    &=\prod_{k=1}^r \frac{1}{(e^{2\pi i s_k}-1)\Gamma(s_k)}\notag\\
    & \ \times \int_{\mathcal{C}^{r}}
    \prod_{j=1}^{r} \lim_{c \to 1}\frac{(-1)}{c-1}\left( \frac{1}{\exp\left(\gamma_j \sum_{k=j}^r t_k\right)-1}-\frac{c}{\exp\left(c\gamma_j \sum_{k=j}^r t_k\right)-1}\right)\prod_{k=1}^r t_k^{s_k-1}dt_k,\label{Cont-2}
  \end{align}
which can be analytically continued to $\mathbb{C}^r$ as an entire function in $(s_j)$. 
\end{theorem}

For the proof of \eqref{Cont-2}, it is enough to prove that 
if $|c-1|$ is sufficiently small, then there exists a function $F:\,\mathcal{C}^{r} \to \mathbb{R}_{>0}$ independent of $c$ such that 
\begin{gather}
  |(c-1)^{-r}\widetilde{\mathfrak{H}}_r  ((t_j);(\gamma_j);c)|\leqslant F((t_j)) \qquad((t_j) \in \mathcal{C}^{r}), \label{eq-01}\\
  \int_{\mathcal{C}^{r}} F((t_j))\prod_{k=1}^r |t_k^{s_k-1}dt_k|<\infty.\label{eq-02}
\end{gather}
Now we aim to construct $F((t_j))$ which satisfies these conditions. 
Let $\nn(\varepsilon)=\{z\in\mathbb{C}~|~|z|\leqslant \varepsilon\}$ and 
$\sss(\theta)=\{z\in\mathbb{C}~|~|\arg z|\leqslant \theta\}$.

Let $\gamma_1,\ldots,\gamma_r\in \mathbb{C}$ with $\Re \gamma_j >0 \quad (1\leqslant j\leqslant r)$. 
Then the following lemma is obvious.

\begin{lemma}\label{L-1-4-1}
There exist $\varepsilon>0$ and $0<\theta<\pi/2$ such that 
\begin{equation}
  \gamma_j \sum_{k=j}^r t_k\in \nn(1)\cup \sss(\theta) 
\end{equation}
for any $(t_j)\in \mathcal{C}^{r}$, where $\mathcal{C}$ is the Hankel contour involving a circle around the origin of radius $\varepsilon$ (see \eqref{Cont}). 
\end{lemma}

Further we prove the following lemma.

\begin{lemma}\label{L-1-4-2}
Let $c\in \mathbb{R}\setminus \{1\}$ satisfying that $|c-1|$ is sufficiently small. Then there exists a constant $A>0$ independent of $c$ such that 
  \begin{equation}
    |c-1|^{-1}\Bigl|\frac{1}{e^y-1}-\frac{c}{e^{cy}-1}\Bigr|<Ae^{-\Re y/2}
  \end{equation}
for any 
  $y\in \nn(1)\cup \sss(\theta)$.
\end{lemma}

\begin{proof}
It is noted that there exists a constant $C>0$ such that 
  \begin{equation*}
    |c-1|^{-1}\Bigl|\frac{1}{e^y-1}-\frac{c}{e^{cy}-1}\Bigr|<C\quad (y\in \nn(1)),
  \end{equation*}
where we interpret this inequality for $y=0$ as that for $y\to 0$. 
Also, for any $y\in \sss(\theta)\setminus \nn(1)$, we have
  \begin{equation*}
    \begin{split}
      |c-1|^{-1}
      \Bigl|\frac{1}{e^y-1}-\frac{c}{e^{cy}-1}\Bigr|
      &=
      |c-1|^{-1}
      \Bigl|\frac{e^{cy}-ce^y+c-1}{(e^y-1)(e^{cy}-1)}\Bigr|
      \\
      &=
      |c-1|^{-1}
      \Bigl|\frac{e^{cy}-e^y+(1-c)(e^y-1)}{(e^y-1)(e^{cy}-1)}\Bigr|
      \\
      &\leqslant \
      |c-1|^{-1}
      \frac{|e^{cy}-e^y|}{|e^y-1||e^{cy}-1|}
      +\frac{1}{|e^{cy}-1|}.
    \end{split}
  \end{equation*}
Hence it is necessary to estimate 
\begin{equation*}
  |c-1|^{-1}
  \frac{|e^{cy}-e^y|}{|e^y-1||e^{cy}-1|}.
\end{equation*}
We note that
\begin{equation*}
  \begin{split}
    \Bigl|\frac{e^{ay}-1}{a}\Bigr|&=\Bigl|\sum_{j=1}^\infty \frac{a^{j-1}y^j}{j!}\Bigr|
    \leqslant \
    |y|\sum_{l=0}^\infty \frac{|ay|^{l}}{l!}
\leqslant \
    |y|e^{|ay|}.
  \end{split}
\end{equation*}
Since $|y|\leqslant \Re y/\cos\theta$, we have
  \begin{equation*}
    \begin{split}
      |c-1|^{-1}
      \frac{|e^{cy}-e^y|}{|e^y-1||e^{cy}-1|}
      &=
      \frac{1}{|1-e^{-y}||e^{cy}-1|}
      \frac{|e^{(c-1)y}-1|}{|c-1|}
      \\
      &\leqslant
      \frac{1}{|1-e^{-y}||e^{cy}-1|}
      |y|e^{|(c-1)y|}
      \\
      &\leqslant
      \frac{|y|e^{\Re y(|c-1|/\cos\theta)}}{|1-e^{-y}||e^{cy}-1|}.
    \end{split}
  \end{equation*}
Therefore, if $|c-1|$ is sufficiently small, then there exists a constant $A>0$ such that 
  \begin{equation*}
      |c-1|^{-1}
      \frac{|e^{cy}-e^y|}{|e^y-1||e^{cy}-1|}
      \leqslant
      Ae^{-\Re y/2}.
  \end{equation*}
This completes the proof.
\end{proof}

\begin{proof}[Proof of Theorem \ref{T-c-1-zeta}]
With the notation provided in Lemmas \ref{L-1-4-1} and \ref{L-1-4-2}, we set
\begin{equation*}
  \begin{split}
    F((t_j))&=A^r \prod_{j=1}^r \exp\left(-\Re (\gamma_j \sum_{k=j}^r t_k/2)\right)
    =A^r \exp\left(-\sum_{j=1}^r\Re (\gamma_j \sum_{k=j}^r t_k/2)\right)
    \\
    &=A^r \exp\left(-\sum_{k=1}^r\Re (t_k(\sum_{j=1}^k \gamma_j /2))\right)
    =A^r \prod_{k=1}^r\exp\left(-\Re (t_k(\sum_{j=1}^k \gamma_j /2))\right).
  \end{split}
\end{equation*}
Then it is clear that $F((t_j))$ satisfies 
\eqref{eq-01} and \eqref{eq-02}.
Hence, by Lebesgue's convergence theorem we see that \eqref{Cont-2} holds. 

Similarly to the proof of Theorem \ref{T-multiple}, 
since the contour integral on the right-hand side of \eqref{Cont-2} is holomorphic for all $(s_k)\in \mathbb{C}^r$, we see that $\zeta^{\rm des}_r((s_j);(\gamma_j))$ can be meromorphically continued to $\mathbb{C}^r$ and its possible singularities are located on hyperplanes $s_k=l_k\in \mathbb{N}$ $(1\leqslant k \leqslant r)$ outside of the region of convergence because $(e^{2\pi i s_k}-1)\Gamma(s_k)$ does not vanish at $s_k\in \mathbb{Z}_{\leqslant 0}$. Furthermore, for $s_k=l_k\in \mathbb{N}$, the integrand of the contour integral with respect to $t_k$ on the right-hand side of \eqref{Cont-2} is holomorphic around $t_k=0$. Therefore, for $l_k\in \mathbb{N}$, we see that
\begin{align*}
    &\lim_{s_k\to l_k}\int_{\mathcal{C}}
\lim_{c \to 1}\frac{(-1)}{c-1}\left( \frac{1}{\exp\left(\gamma_j \sum_{\nu=j}^r t_\nu\right)-1}-\frac{c}{\exp\left(c\gamma_j \sum_{\nu=j}^r t_\nu\right)-1}\right) t_k^{s_k-1}dt_k\\
    &=-\int_{C_\varepsilon}
\lim_{c \to 1}\frac{1}{c-1}\left( \frac{1}{\exp\left(\gamma_j \sum_{\nu=j}^r t_\nu\right)-1}-\frac{c}{\exp\left(c\gamma_j \sum_{\nu=j}^r t_\nu\right)-1}\right) t_k^{l_k-1}dt_k\\
& =0,
  \end{align*}
because of the residue theorem, where $C_\varepsilon=\{\varepsilon e^{i\theta}\,|\,0\leqslant \theta\leqslant 2\pi\}$ for any sufficiently small $\varepsilon$. Consequently this implies that $\zeta^{\rm des}_r((s_j);(\gamma_j))$ has no singularity on $s_k=l_k$, namely $\zeta^{\rm des}_r((s_j);(\gamma_j))$ is entire. 
Thus we complete the proof of Theorem \ref{T-c-1-zeta}.
\end{proof}

\begin{theorem}\label{C-Zr}
For $\gamma_1,\ldots,\gamma_r\in \mathbb{C}$ with $\Re \gamma_j >0 \quad (1\leqslant j\leqslant r)$, 
\begin{equation}
\begin{split}
& \prod_{j=1}^{r} \frac{\left(1-\gamma_j \sum_{k=j}^r t_k\right)\exp\left(\gamma_j \sum_{k=j}^r t_k\right)-1}{\left( \exp\left(\gamma_j \sum_{k=j}^r t_k\right)-1\right)^2} \\
& \quad =\sum_{m_1,\ldots,m_r=0}^\infty (-1)^{m_1+\cdots+m_r}\zeta^{\rm des}_r((-m_j);(\gamma_j))\prod_{j=1}^{r}\frac{t_j^{m_j}}{m_j!}.
\end{split}
\label{cont-gene}
\end{equation}
Hence, for $(k_j)\in \mathbb{N}_0^r$, 
\begin{align}
    \zeta^{\rm des}_r((-k_j);(\gamma_j))& =\prod_{l=1}^r (-1)^{k_l}k_l!\notag\\
    &\quad \times 
\sum_{\nu_{11}\geqslant 0 \atop {\nu_{12},\,\nu_{22}\geqslant 0 \atop {\,\cdots \atop {\nu_{1r},\,\ldots,\,\nu_{rr}\geqslant 0 \atop {\sum_{d=1}^j \nu_{dj}=k_j \atop (1\leqslant j\leqslant r)}}}}}\prod_{j=1}^{r} \left(B_{1+\sum_{l=j}^r\nu_{jl}}\,\gamma_j^{\sum_{l=j}^r \nu_{jl}}\frac{1}{\prod_{d=1}^{j}\nu_{dj}!}\right). \label{Cont-2-02}
  \end{align}
\end{theorem}

\begin{proof}
By \eqref{def-tilde-H}, we have
\begin{align*}
& \lim_{c\to 1}\frac{(-1)^r}{(c-1)^r}\widetilde{\mathfrak{H}}_r  (( t_j )_{j=1}^{r}; ( \gamma_j)_{j=1}^{r};c)\\
&\quad =\lim_{c \to 1}\prod_{j=1}^{r}\frac{(-1)}{c-1}\left( \frac{1}{\exp\left(\gamma_j \sum_{k=j}^r t_k\right)-1}-\frac{c}{\exp\left(c\gamma_j \sum_{k=j}^r t_k\right)-1}\right)\\
& \quad =\prod_{j=1}^{r}\frac{\left(1-\gamma_j \sum_{k=j}^r t_k\right)\exp\left(\gamma_j \sum_{k=j}^r t_k\right)-1}{\left( \exp\left(\gamma_j \sum_{k=j}^r t_k\right)-1\right)^2}.
\end{align*}
Hence we obtain \eqref{cont-gene} from \eqref{Cont-2}. 
Also, by \eqref{def-tilde-H}, we have 
\begin{align*}
& \lim_{c\to 1}\frac{(-1)^r}{(c-1)^r}\widetilde{\mathfrak{H}}_r  (( t_j )_{j=1}^{r}; ( \gamma_j)_{j=1}^{r};c)\\
& =\lim_{c\to 1}\prod_{j=1}^{r} \left(\sum_{m_j=1}^\infty  \frac{c^{m_j}-1}{c-1}B_{m_j}\frac{\left(\gamma_j \sum_{l=j}^r t_l\right)^{m_j-1}}{m_j!}\right)\\
& =\prod_{j=1}^{r} \left(\sum_{m_j=1}^\infty  B_{m_j}\frac{\left(\gamma_j \sum_{l=j}^r t_l\right)^{m_j-1}}{(m_j-1)!}\right)\\
& =\prod_{j=1}^{r} \left(\sum_{n_j=0}^\infty  B_{n_j+1}\gamma_j^{n_j}\sum_{\nu_{jj},\ldots,\nu_{jr}\geqslant 0 \atop \sum_{l=j}^r \nu_{jl}=n_j}\frac{t_j^{\nu_{jj}}}{\nu_{jj}!}\cdots \frac{t_r^{\nu_{jr}}}{\nu_{jr}!}\right)\\
& =\sum_{\nu_{11}\geqslant 0 \atop {\nu_{12},\,\nu_{22}\geqslant 0 \atop {\,\cdots \atop \nu_{1r},\,\nu_{2r},\ldots,\,\nu_{rr}\geqslant 0}}}\prod_{j=1}^{r} \left(B_{1+\sum_{l=j}^r\nu_{jl}}\,\gamma_j^{\sum_{l=j}^r \nu_{jl}}\frac{t_j^{\sum_{d=1}^{j}\nu_{dj}}}{\prod_{d=1}^{j}\nu_{dj}!}\right).
\end{align*}
Hence, substituting the above relation into \eqref{Cont-2} and using the residue theorem with 
$$\lim_{s\to -k}\left(e^{2\pi is}-1\right)\Gamma(s)=\frac{(2\pi i)(-1)^k}{k!}\quad (k\in \mathbb{N}_0),$$
we have
\begin{align*}
    \zeta^{\rm des}_r((-k_j);(\gamma_j))& =\prod_{l=1}^r \frac{(-1)^{k_l}k_l!}{2\pi i}\notag\\
    &\times {(2\pi i)^r}
\sum_{\nu_{11}\geqslant 0 \atop {\nu_{12},\,\nu_{22}\geqslant 0 \atop {\,\cdots \atop {\nu_{1r},\,,\ldots,\,\nu_{rr}\geqslant 0 \atop {\sum_{d=1}^j \nu_{dj}=k_j \atop (1\leqslant j\leqslant r)}}}}}\prod_{j=1}^{r} \left(B_{1+\sum_{l=j}^r\nu_{jl}}\,\gamma_j^{\sum_{l=j}^r \nu_{jl}}\frac{1}{\prod_{d=1}^{j}\nu_{dj}!}\right). 
  \end{align*}
Thus we obtain the assertion.
\end{proof}




Now we give a certain expression of $\zeta^{\rm des}_r((s_j);(\gamma_j))$ in terms of $\zeta_r((s_j);(1);(\gamma_j))$ (Theorem \ref{Th-ex}), which can be regarded as a multiple version of $\zeta^{\rm des}_1(s;1)=(1-s)\zeta(s)$ in the case $r=1$ (see Examples \ref{Exam-SH} and \ref{Ex-ex-1}). 

For $s_j\in \mathbb{C}$ with $\Re s_j>1$ $(1\leqslant j\leqslant r)$ and $\gamma_1,\ldots,\gamma_r\in \mathbb{C}$ with $\Re \gamma_j>0$ $(1\leqslant j \leqslant r)$, we set
\begin{multline}
  I_{c,r}(s_1,\ldots,s_r;\gamma_1,\ldots,\gamma_r):=\frac{1}{\prod_{j=1}^r\Gamma(s_j)}
  \int_{[0,\infty)^r}
  \prod_{j=1}^r dt_j
  \prod_{j=1}^rt_j^{s_j-1}
  \\
  \times\prod_{j=1}^r\Biggl(\frac{1}{\exp\Bigl(\gamma_j\sum_{k=j}^r t_k\Bigr)-1}
  -
  \frac{c}{\exp\Bigl(c\gamma_j\sum_{k=j}^r t_k\Bigr)-1}
  \Biggr). \label{ex-01}
\end{multline}
From Definition \ref{def-MZF-2}, we see that 
$$\zeta^{\rm des}_r((s_j);(\gamma_j))=\lim_{c\to 1}\frac{(-1)^r}{(c-1)^r}I_{c,r}((s_j);(\gamma_j)).$$
For indeterminates $u_j,v_j$ $(1\leqslant j \leqslant r)$, we set 
\begin{equation}
\mathcal{G}((u_j),(v_j)):=\prod_{j=1}^r\Bigl( 1-(u_jv_j+\cdots+u_rv_r)(v_j^{-1}-v_{j-1}^{-1})\Bigr)  \label{def-GG}
\end{equation}
with the
convention $v_0^{-1}=0$, and also define the set of integers $\{ a_{\mb,\nb}\}$ by 
\begin{equation}
  \mathcal{G}((u_j),(v_j))=\sum_{\mb=(l_j) \in \mathbb{N}_0^r \atop {\nb=(m_j) \in \mathbb{Z}^r \atop \sum_{j=1}^r m_j=0}} a_{\mb,\nb}\prod_{j=1}^r u_j^{l_j}v_j^{m_j}, \label{ex-02}
\end{equation}
where the sum on the right-hand side is obviously a finite sum. 
Note that the condition $\sum_{j=1}^r m_j=0$ for the summation indices $\nb=(m_j)$ can be deduced from the fact that the right-hand side of \eqref{def-GG} is a homogeneous polynomial of degree $0$ in $(v_j)$, namely so is that of \eqref{ex-02}. 

\begin{theorem}\label{Th-ex} 
For $\gamma_1,\ldots,\gamma_r\in \mathbb{C}$ with $\Re \gamma_j>0$ $(1\leqslant j \leqslant r)$,
\begin{align}
  \zeta^{\rm des}_r((s_j);(\gamma_j))
& =\sum_{\mb=(l_j) \in \mathbb{N}_0^r \atop {\nb=(m_j) \in \mathbb{Z}^r \atop \sum_{j=1}^r m_j=0}} a_{\mb,\nb}\Bigl(\prod_{j=1}^{r}(s_j)_{l_j}\Bigr)\zeta_r(s_1+m_1,\ldots,s_r+m_r;(1);(\gamma_j))\label{ex-03}
\end{align}
holds for all $(s_j)\in \mathbb{C}^r$, 
where $(s)_0=1$ and $(s)_k=s(s+1)\cdots(s+k-1)$ $(k\in \mathbb{N})$ are the Pochhammer symbols.
\end{theorem}

We emphasize here that
each term of the right-hand side of \eqref{ex-03} is meromorphic
with infinitely many singularities but
taking the above {\it finite} sum of the shifted functions causes \lq miraculous' cancellations 
of all the {\it infinitely} many singularities
to conclude an entire function.

\begin{remark}
In \eqref{ex-03}, the condition $\sum_{j=1}^{r}m_j=0$ implies that all zeta-functions appearing on the both sides have the same \textit{weight} $s_1+\cdots+s_r$. 
\end{remark}


\begin{proof}[Proof of Theorem \ref{Th-ex}]
First we assume that $\Re s_j$ is sufficiently large for $1\leqslant j\leqslant r$. 
From \eqref{Cont} with $(\xi_j)=(1)$, we have
\begin{align}
& \zeta_r((s_j);(1);(\gamma_j))
=\frac{1}{\prod_{j=1}^r\Gamma(s_j)}
\int_{[0,\infty)^r}\prod_{j=1}^r\frac{t_j^{s_j-1}}{\exp\Bigl(\gamma_j
\sum_{k=j}^r t_k\Bigr)-1}\prod_{j=1}^r dt_j. \label{ex-04}
\end{align}
Using the relation
\begin{align*}
&\lim_{c\to1}
\frac{(-1)}{c-1}\Bigl(  \frac{1}{e^y-1}-\frac{c}{e^{cy}-1}\Bigr)\\
& \ =\frac{-1+e^y-ye^y}{(e^y-1)^2}=\frac{1}{e^y-1}-\frac{ye^y}{(e^y-1)^2}\ =E(y)\ (say),
\end{align*}
we have
\begin{multline}
  \zeta^{\rm des}_r((s_j);(\gamma_j))=\lim_{c\to1}\frac{(-1)^r}{(c-1)^r}{I_{c,r}((s_j);(\gamma_j))}
  \\
  \begin{aligned}
    &=
    \lim_{c\to1}\frac{1}{\prod_{j=1}^r\Gamma(s_j)}
    \int_{[0,\infty)^r}
    \prod_{j=1}^r dt_j
    \prod_{j=1}^r t_j^{s_j-1}
    \\
    &
    \qquad\times\prod_{j=1}^r\frac{(-1)}{c-1}\Biggl(\frac{1}{\exp\Bigl(\gamma_j\sum_{k=j}^r t_k\Bigr)-1}
    -
    \frac{c}{\exp\Bigl(c\gamma_j \sum_{k=j}^r t_k\Bigr)-1}
    \Biggr)
    \\
    &=
    \frac{1}{\prod_{j=1}^r\Gamma(s_j)}
    \int_{[0,\infty)^r}
    \prod_{j=1}^r dt_j
    \prod_{j=1}^r t_j^{s_j-1}
\prod_{j=1}^r
E\Bigl(\gamma_j\sum_{k=j}^r t_k\Bigr).
  \end{aligned}
\label{ex-04-02}
\end{multline}
We calculate the last product of \eqref{ex-04-02}. Using the relations 
\begin{align*}
  \frac{1}{e^y-1}&=\sum_{n=1}^\infty e^{-ny}, \quad \frac{e^y}{(e^y-1)^2}=\sum_{n=1}^\infty ne^{-ny},
\end{align*}
we have, for $J\subset\{1,\ldots,r\}$, 
\begin{align}
&  \int_{[0,\infty)^r}
  \prod_{j=1}^r dt_j
  \prod_{j=1}^r t_j^{s_j-1}
  \prod_{j\notin J}\frac{1}{\exp\Bigl(\gamma_j\sum_{k=j}^r t_k\Bigr)-1}
  \prod_{j\in J}\frac{\Bigl(\gamma_j\sum_{k=j}^r t_k\Bigr)\exp\Bigl(\gamma_j\sum_{k=j}^r t_k\Bigr)}{\Bigl(\exp\Bigl(\gamma_j\sum_{k=j}^r t_k\Bigr)-1\Bigr)^2}
\notag\\
&=
  \int_{[0,\infty)^r}
  \prod_{j=1}^r dt_j
  \prod_{j=1}^r t_j^{s_j-1}
  \prod_{j\notin J}\sum_{n_j=1}^\infty\exp\Bigl(-n_j\gamma_j\sum_{k=j}^r t_k\Bigr)\notag\\
& \qquad \times 
  \prod_{j\in J}\sum_{n_j=1}^\infty n_j\exp\Bigl(-n_j\gamma_j\sum_{k=j}^r t_k\Bigr)
  \prod_{j\in J}\Bigl(\gamma_j\sum_{k=j}^r t_k\Bigr)
\notag\\
&=
\sum_{\substack{n_1,\ldots,n_r\geqslant1}}
  \Bigl(\prod_{j\in J}n_j\gamma_j\Bigr)
  \int_{[0,\infty)^r}
  \prod_{j=1}^r dt_j
  \prod_{j=1}^r t_j^{s_j-1}
  \prod_{j=1}^r
  \exp\Bigl(-t_j\sum_{k=1}^j n_k\gamma_k\Bigr)
  \prod_{j\in J}\Bigl(\sum_{k=j}^r t_k\Bigr)
\notag\\
&=
\sum_{\substack{n_1,\ldots,n_r\geqslant 1}}
  \Bigl(\prod_{j\in J}\Bigl(\sum_{k=1}^j n_k\gamma_k-\sum_{k=1}^{j-1} n_k\gamma_{k}\Bigr)\Bigr)\notag\\
& \qquad \times  \int_{[0,\infty)^r}
  \prod_{j=1}^r dt_j
  \prod_{j=1}^r t_j^{s_j-1}
  \prod_{j=1}^r
  \exp\Bigl(-t_j\sum_{k=1}^j n_k\gamma_{k}\Bigr)
  \prod_{j\in J}\Bigl(\sum_{k=j}^r t_k\Bigr)
\notag\\
&=
\sum_{\mb \in \mathbb{N}_0^r}b_{J,\mb}
\sum_{\substack{n_1,\ldots,n_r\geqslant 1}}
  \Bigl(\prod_{j\in J}\Bigl(\sum_{k=1}^j n_k\gamma_k-\sum_{k=1}^{j-1} n_k\gamma_k\Bigr)\Bigr)\notag\\
&\qquad \times  \int_{[0,\infty)^r}
  \prod_{j=1}^r dt_j
  \prod_{j=1}^r t_j^{s_j+l_j-1}
  \prod_{j=1}^r
  \exp\Bigl(-t_j\sum_{k=1}^j n_k\gamma_k\Bigr)
\notag\\
&=
\sum_{\mb\in \mathbb{N}_0^r}b_{J,\mb}
\sum_{\substack{n_1,\ldots,n_r\geqslant 1}}
  \Bigl(\prod_{j\in J}\Bigl(\sum_{k=1}^j n_k\gamma_k-\sum_{k=1}^{j-1} n_k\gamma_k\Bigr)\Bigr)
  \prod_{j=1}^r\Gamma(s_j+l_j)\frac{1}{\Bigl(\sum_{k=1}^j n_k\gamma_k\Bigr)^{s_j+l_j}}
\notag\\
&=
\sum_{\mb\in \mathbb{N}_0^r}b_{J,\mb}
\prod_{j=1}^r\Gamma(s_j+l_j)
\sum_{\substack{n_1,\ldots,n_r\geqslant 1}}
\sum_{K\subset J\setminus\{1\}}(-1)^{|K|}
  \prod_{j=1}^r\frac{1}{\Bigl(\sum_{k=1}^j n_k\gamma_k\Bigr)^{s_j+l_j-\delta_{j\in J\setminus K}-\delta_{j+1\in K}}}
\notag\\
&=
\sum_{\mb\in \mathbb{N}_0^r}b_{J,\mb}
\sum_{K\subset J\setminus\{1\}}(-1)^{|K|}
\Bigl(\prod_{j=1}^r\Gamma(s_j+l_j)\Bigr)
\zeta_r((s_j+l_j-\delta_{j\in J\setminus K}-\delta_{j+1\in K});(1);(\gamma_j)),
\label{ex-05-2}
\end{align}
where $|K|$ implies the number of elements of $K$, 
\begin{equation*}
\delta_{i\in I}=
\begin{cases}
1 & (i\in I)\\
0 & (i\not\in I)
\end{cases}
\end{equation*}
for $I\subset J$, and 
\begin{equation}
  \label{ex-06}
  \prod_{j\in J}\Bigl(\sum_{k=j}^r t_k\Bigr)=\sum_{\mb \in \mathbb{N}_0^r} b_{J,\mb} \prod_{j=1}^rt_j^{l_j}.
\end{equation}
Hence, by \eqref{ex-04-02} we have
\begin{multline}
  \zeta^{\rm des}_r((s_j);(\gamma_j))
  \\
  \begin{aligned}
    &=
    \sum_{J\subset\{1,\ldots,r\}}(-1)^{|J|}
    \sum_{\mb\in \mathbb{N}_0^r}b_{J,\mb}
    \sum_{K\subset J\setminus\{1\}}(-1)^{|K|}
    \Bigl(\prod_{j=1}^r\frac{\Gamma(s_j+l_j)}{\Gamma(s_j)}
    \Bigr)
    \zeta_r((s_j+l_j-\delta_{j\in J\setminus K}-\delta_{j+1\in K});(1);(\gamma_j))
    \\
    &=
    \sum_{J\subset\{1,\ldots,r\}}
    \sum_{K\subset J\setminus\{1\}}(-1)^{|J\setminus K|}
    \sum_{\mb\in \mathbb{N}_0^r}b_{J,\mb}
    \Bigl(\prod_{j=1}^r(s_j)_{l_j}
    \Bigr)
    \zeta_r((s_j+l_j-\delta_{j\in J\setminus K}-\delta_{j+1\in K});(1);(\gamma_j)).
  \end{aligned}
\label{ex-07}
\end{multline}
Finally we set
\begin{equation*}
H((u_j),(v_j)):=
  \sum_{J\subset\{1,\ldots,r\}}
  \sum_{K\subset J\setminus\{1\}}(-1)^{|J\setminus K|}
  \sum_{\mb \in \mathbb{N}_0^r}b_{J,\mb}
  \prod_{j=1}^r
  u_j^{l_j}v_j^{l_j-\delta_{j\in J\setminus K}-\delta_{j+1\in K}}
\end{equation*}
and aim to prove that 
\begin{equation}
\mathcal{G}((u_j),(v_j))=H((u_j),(v_j)). \label{ex-claim}
\end{equation}

It follows from \eqref{ex-06} that
\begin{equation*}
  \begin{split}
    H((u_j),(v_j))
    &=
    \sum_{J\subset\{1,\ldots,r\}}
    \sum_{K\subset J\setminus\{1\}}(-1)^{|J\setminus K|}
    \Bigl(\prod_{j\in J}\sum_{k=j}^r u_kv_k\Bigr)
    \prod_{j=1}^rv_j^{-\delta_{j\in J\setminus K}-\delta_{j+1\in K}}
    \\
    &=
    \sum_{J\subset\{1,\ldots,r\}}
    \Bigl(\prod_{j\in J}\sum_{k=j}^r u_kv_k\Bigr)
    \sum_{K\subset J\setminus\{1\}}
    \prod_{j\in J\setminus K}(-v_j^{-1})
    \prod_{j\in K}v_{j-1}^{-1}.
  \end{split}
\end{equation*}
Since $v_0^{-1}=0$, we have
\begin{equation*}
      \sum_{K\subset J\setminus\{1\}}
    \prod_{j\in J\setminus K}(-v_j^{-1})
    \prod_{j\in K}v_{j-1}^{-1}=\prod_{j\in J}(-v_j^{-1}+v_{j-1}^{-1}).
\end{equation*}
Hence we obtain
\begin{equation}
  \begin{split}
    H((u_j),(v_j))
    &=
    \sum_{J\subset\{1,\ldots,r\}}
    \prod_{j\in J}\Bigl(\sum_{k=j}^r u_kv_k\Bigr)
    (-v_j^{-1}+v_{j-1}^{-1})
    \\
    &=
    \prod_{j=1}^r\Bigl(\Bigl(\sum_{k=j}^r u_kv_k\Bigr)
    (-v_j^{-1}+v_{j-1}^{-1})+1\Bigr)
    \\
    &=
    \prod_{j=1}^r\Bigl(1-\Bigl(\sum_{k=j}^r u_kv_k\Bigr)
    (v_j^{-1}-v_{j-1}^{-1})\Bigr)=\mathcal{G}((u_j),(v_j)).
  \end{split}
\label{ex-08}
\end{equation}
Combining \eqref{ex-02}, \eqref{ex-07} and \eqref{ex-08}, and regarding 
$(s_j)_{l_j}$ and $\zeta_r((s_j+l_j-\delta_{j\in J\setminus K}-\delta_{j+1\in K});(1);(\gamma_j))$ as indeterminates $u_j^{l_j}$ and $v_j^{l_j}$, 
we see that \eqref{ex-03} holds when 
$\Re s_j$ is sufficiently large for $1\leqslant j\leqslant r$. 
It is known that each function on the right-hand side can be continued meromorphically to $\mathbb{C}^r$ (see \cite[Theorem 1]{MaJNT}). Since $\zeta^{\rm des}_r((s_j);(\gamma_j))$ is entire, we see that \eqref{ex-03} holds for all $(s_j)\in \mathbb{C}^r$. 
Thus we complete the proof of Theorem \ref{Th-ex}.
\end{proof}

\begin{example}\label{Ex-ex-1}
In the case $r=1$ and $\gamma_1=1$, we have
  \begin{equation*}
    \mathcal{G}(u_1,v_1)=1- u_1v_1v_1^{-1}=1-u_1,
  \end{equation*}
namely, $a_{0,0}(1)=1$ and $a_{1,0}(1)=-1$. Hence we have
\begin{equation*}
  \zeta^{\rm des}_1(s;1)=\lim_{c\to1}\frac{(-1)}{c-1}{I_{c,1}(s)}=(s)_0\zeta_1(s;1;1)-(s)_1\zeta_1(s;1;1)=(1-s)\zeta(s),
\end{equation*}
which coincides with \eqref{def-Z1}. We see that 
$$\zeta^{\rm des}_1(1;1)=-1$$ 
and 
$$\zeta^{\rm des}_1(-k;1)=(-1)^{k}B_{k+1} \qquad (k\in \mathbb{N}_0).$$ 
\end{example}

\begin{example}\label{Z-EZ-double}
In the case $r=2$, we can easily check that 
  \begin{equation*}
    \begin{split}
      \mathcal{G}((u_j),(v_j))
      &=(1-(u_1v_1+u_2v_2)v_1^{-1})(1-u_2v_2(v_2^{-1}-v_1^{-1}))
      \\
      &=(1-u_1)(1-u_2)+(u_2^2-u_1u_2)v_1^{-1}v_2-u_2^2v_1^{-2}v_2^2.
    \end{split}
  \end{equation*}
Then \eqref{ex-03} implies that
\begin{align}
  \zeta^{\rm des}_2(s_1,s_2;\gamma_1,\gamma_2)& =(s_1-1)(s_2-1)\zeta_2(s_1,s_2;(1);\gamma_1,\gamma_2)\notag
  \\
  & \quad +s_2(s_2+1-s_1)\zeta_2(s_1-1,s_2+1;(1);\gamma_1,\gamma_2)\notag\\
  & \quad -s_2(s_2+1)\zeta_2(s_1-2,s_2+2;(1);\gamma_1,\gamma_2). \label{ex-04-2}
\end{align}
Let $k,l\in \mathbb{N}_0$. 
By \eqref{Cont-2-02}, 
we obtain
\begin{equation}
\zeta^{\rm des}_2(-k,-l;\gamma_1,\gamma_2)=(-1)^{k+l}\sum_{\nu=0}^{l}\binom{l}{\nu}B_{k+\nu+1}B_{l-\nu+1}\gamma_1^{k+\nu}\gamma_2^{l-\nu}. \label{EZ-val}
\end{equation}
\end{example}

\begin{remark}\label{EZ-double}
Setting $(\gamma_1,\gamma_2)=(1,1)$ in \eqref{ex-04-2}, we obtain
\begin{align}
  \zeta^{\rm des}_2(s_1,s_2;1,1)
  & =(s_1-1)(s_2-1)\zeta_2(s_1,s_2)\notag\\
  & \quad +s_2(s_2+1-s_1)\zeta_2(s_1-1,s_2+1)-s_2(s_2+1)\zeta_2(s_1-2,s_2+2). \label{ex-04-3}
\end{align}
%
From Theorem \ref{T-c-1-zeta}, 
we see that $\zeta^{\rm des}_2(s_1,s_2;1,1)$ on the left-hand side of \eqref{ex-04-3} 
is entire, 
though each double zeta-function (defined by \eqref{MZF-def}) on the right-hand side of \eqref{ex-04-3} has infinitely many singularities (see Theorem \ref{T-AET}).
In fact, we can explicitly write $\zeta^{\rm des}_2(-m,-n;1,1)$ in terms of Bernoulli numbers by \eqref{EZ-val}, though the values of $\zeta_2(s_1,s_2)$ at non-positive integers (except for regular points) cannot be determined uniquely because they are irregular singularities (see \cite{AET}).
%
\end{remark}

\begin{example}
In the case $r=3$, we can see that
\begin{align*}
    \mathcal{G}((u_j),(v_j))
    &=(1-(u_1v_1+u_2v_2+u_3v_3)v_1^{-1})(1-(u_2v_2+u_3v_3)(v_2^{-1}-v_1^{-1}))
    \\
    &\qquad\times
    (1-u_3v_3(v_3^{-1}-v_2^{-1}))
    \\
    &=
    -(u_1-1) (u_2-1) (u_3-1)     
    +(u_1-1) (u_2 - u_3) u_3 v_2^{-1}v_3   \\
    &\qquad
    + (u_1-1) u_3^2 v_2^{-2}v_3^2    
    +(u_1 - u_2) u_2 (u_3-1) v_1^{-1}v_2 \\
    &\qquad
    + u_3 (-u_1 + 2 u_2 - u_1 u_2 + u_2^2 + u_1 u_3 - 2 u_2 u_3) v_1^{-1}v_3 \\
    &\qquad
    - u_3^2 (-1 + u_1 - 2 u_2 + u_3) v_1^{-1}v_2^{-1}v_3^2
    + u_3^3 v_1^{-1}v_2^{-2}v_3^3     \\
    &\qquad
    +u_2^2 (u_3-1)v_1^{-2} v_2^2 
    -  u_2 (2 + u_2 - 2 u_3) u_3 v_1^{-2}v_2 v_3 \\
    &\qquad
    + u_3^2 (-1 - 2 u_2 + u_3) v_1^{-2}v_3^2 
    - u_3^3 v_1^{-2}v_2^{-1}v_3^3.
\end{align*}
Therefore we obtain
\begin{align*}
    &\zeta_3^{\rm des}(s_1,s_2,s_3;\gamma_1,\gamma_2,\gamma_3)\\
    &\quad =
    -(s_1-1) (s_2-1) (s_3-1) \zeta_3(s_1,s_2,s_3;(1);\gamma_1,\gamma_2,\gamma_3)\\
    &\qquad
    + (s_1-1) (-1 + s_2 -  s_3) s_3 \zeta_3(s_1,s_2-1,s_3+1;(1);\gamma_1,\gamma_2,\gamma_3)\\
    &\qquad
    + (s_1-1) s_3 (s_3+1) \zeta_3(s_1,s_2-2,s_3+2;(1);\gamma_1,\gamma_2,\gamma_3)\\
    &\qquad
    +(-1 + s_1 - s_2) s_2 (s_3-1) \zeta_3(s_1-1,s_2+1,s_3;(1);\gamma_1,\gamma_2,\gamma_3)\\
    &\qquad
    + s_3 (s_2 - s_1 s_2 + s_2^2 + s_1 s_3 - 2 s_2 s_3) \zeta_3(s_1-1,s_2,s_3+1;(1);\gamma_1,\gamma_2,\gamma_3)\\
    &\qquad
    - s_3 (s_3+1) (1 + s_1 - 2 s_2 + s_3) \zeta_3(s_1-1,s_2-1,s_3+2;(1);\gamma_1,\gamma_2,\gamma_3)\\
    &\qquad
    + s_3 (s_3+1) (s_3+2) \zeta_3(s_1-1,s_2-2,s_3+3;(1);\gamma_1,\gamma_2,\gamma_3)\\
    &\qquad
    +s_2 (s_2+1) (s_3-1) \zeta_3(s_1-2,s_2+2,s_3;(1);\gamma_1,\gamma_2,\gamma_3)\\
    &\qquad
    -  s_2 (1 + s_2 - 2 s_3) s_3 \zeta_3(s_1-2,s_2+1,s_3+1;(1);\gamma_1,\gamma_2,\gamma_3)\\
    &\qquad    
    + s_3 (s_3+1) (1 - 2 s_2 + s_3) \zeta_3(s_1-2,s_2,s_3+2;(1);\gamma_1,\gamma_2,\gamma_3)\\
    &\qquad
    - s_3 (s_3+1) (s_3+2) \zeta_3(s_1-2,s_2-1,s_3+3;(1);\gamma_1,\gamma_2,\gamma_3).
\end{align*}
Let $k,l,m\in \mathbb{N}_0$. 
By \eqref{Cont-2-02}, we have
\begin{align*}
    \zeta^{\rm des}_3(-k,-l,-m;\gamma_1,\gamma_2,\gamma_3)
    &=(-1)^{k+l+m}\sum_{\nu=0}^m
    \sum_{\rho=0}^{m-\nu}\sum_{\kappa=0}^l\binom{l}{\kappa}\binom{m}{\nu\ \rho}
    \\
    &\quad\times
    B_{k+\nu+\kappa+1}B_{l-\kappa+\rho+1}B_{m-\nu-\rho+1}
    \gamma_1^{k+\nu+\kappa+1}\gamma_2^{l-\kappa+\rho+1}\gamma_3^{m-\nu-\rho+1},
\end{align*}
where $\binom{m}{\nu\ \rho}=\frac{m!}{\nu!\, \rho!\, (m-\nu-\rho)!}$.

\end{example}

\begin{remark}\label{rem-general-form}
Our desingularization method in this paper is for multiple zeta-functions of the generalized Euler-Zagier type \eqref{gene-EZ}.
In our forthcoming paper \cite{FKMT02}, we will extend our desingularization method into more general multiple series.
\end{remark}

\begin{remark}\label{AK-xi}
Arakawa and Kaneko \cite{AK} defined an entire function $\xi_k(s)$ associated with poly-Bernoulli numbers $\{B_n^{(k)}\}$ mentioned in Remark \ref{poly-Bernoulli}. It is known that, for example, 
\begin{align*}
& \xi_1(s)=s\zeta(s+1),\\
& \xi_2(s)=-\zeta_2(2,s)+s\zeta_2(1,s+1)+\zeta(2)\zeta(s).
\end{align*}
Comparing these formulas with \eqref{def-Z1} and \eqref{ex-04-2}, and using the well-known formula
$$\zeta_2(0,s)=\sum_{m,n=1}\frac{1}{(m+n)^s}=\sum_{N=2}^\infty \frac{N-1}{N^s}=\zeta(s-1)-\zeta(s),$$
we obtain
\begin{equation*}
\xi_1(s)=-\zeta^{\rm des}_1(s+1;1), 
\end{equation*}
and 
\begin{align*}
&(1-s)\xi_2(s)=\zeta^{\rm des}_2(2,s;1,1)-\zeta^{\rm des}_1(2)\zeta^{\rm des}_1(s;1)-(s+1)\zeta^{\rm des}_1(s+1;1)+s\zeta^{\rm des}_1(s+2;1), \\
&\zeta^{\rm des}_2(2,s;1,1)=(1-s)\xi_2(s)+\xi_1(1)\xi_1(s-1)-(s+1)\xi_1(s)+s\xi_1(s+1). 
\end{align*}
Note that the both sides of the above relations are entire. 
It seems quite interesting if we acquire explicit relations between $\xi_k(s)$ and $\zeta^{\rm des}_k((s_j);(1))$ for any $k\geqslant 3$.
\end{remark}

Related to Connes-Kreimer's renormalization procedure  in quantum field theory,
Guo and Zhang \cite{GZ} and  Manchon and Paycha \cite{MP} 
introduced methods using certain Hopf algebras
to give  well-defined special values of the multiple zeta-functions 
at non-positive integers.

\begin{example}
According to their computation table (in loc.cit.),
Guo-Zhang's renormalized value $\zeta_2^{\rm GZ}(0,-2)$
of $\zeta_2(s_1,s_2)$ at its singularity $(s_1,s_2)=(0,-2)$ is
$$\zeta_2^{\rm GZ}(0,-2)=\frac{1}{120},$$
while Manchon-Paycha's value $\zeta_2^{\rm MP}(0,-2)$ is
$$\zeta_2^{\rm MP}(0,-2)=\frac{7}{720}.$$
On the other hand, our desingularized method gives
$$\zeta_2^{\rm des}(0,-2;1,1)=\frac{1}{18}.$$
\end{example}

Several methods yield several different values.

\begin{question}
Are there any relationships between our desingularization method and
their renormalization methods?
\end{question}

We emphasize that since our $\zeta^{\rm des}_r((s_j);(1))$ is entire,
their special values at integers points which are neither all positive nor all non-positive
are  well-determined.
These values might be also worthy to study.

\begin{example}\label{Exam-1-33}
As explicit examples which we mentioned above, we can give 
\begin{align*}
& \zeta_2^{\rm des}(-1,1;1,1) =\frac{1}{8},\\
& \zeta_2^{\rm des}(-1,4;1,1) =\zeta(3)-\zeta(4),\\
& \zeta_2^{\rm des}(3,-3;1,1)=\frac{3}{4}-\frac{1}{15}\zeta(3),\\
& \zeta_2^{\rm des}(4,-3;1,1)=\frac{1}{2}+\frac{1}{2}\zeta(2)-\frac{1}{10}\zeta(4).
\end{align*}
Also we can give the following examples for non-admissible indices:
\begin{align*}
\zeta_2^{\rm des}(1,1;1,1)&=\frac{1}{2},\\
\zeta_2^{\rm des}(2,1;1,1)&=-\zeta(2)+2\zeta(3),\\
\zeta_2^{\rm des}(3,1;1,1)&=2\zeta(3)-\frac{5}{4}\zeta(4).
\end{align*}
For the details, see our forthcoming paper \cite{FKMT02}.
\end{example}

\ 

\section{$p$-adic multiple $L$-functions}\label{sec-3}

In this section, based on the consideration in the previous section, 
we will construct certain $p$-adic multiple $L$-functions in Definition \ref{Def-pMLF}
which can be regarded as multiple analogues of the Kubota-Leopoldt $p$-adic $L$-functions
and also $p$-adic analogues of complex multiple  $L$- (zeta-)functions.
We will investigate several $p$-adic properties of them, in particular,
a $p$-adic analyticity on certain variables
in Theorem \ref{Th-pMLF} and 
a $p$-adic continuity on  other variables
in Theorem \ref{continuity theorem}.

\subsection{Kubota-Leopoldt $p$-adic $L$-functions}\label{sec-3-1}
Here we will review of  the Kubota-Leopoldt $p$-adic $L$-functions.

First we prepare ordinary notation. For a prime number $p$, let $\mathbb{Z}_p$, $\mathbb{Q}_p$, $\overline{\mathbb{Q}}_p$, $\mathbb{C}_p$, ${\mathcal O}_{\mathbb{C}_p}$ and ${\frak M}_{\mathbb{C}_p}$
be the set of $p$-adic integers, $p$-adic numbers, the algebraic closure of $\mathbb{Q}_p$, the $p$-adic completion of $\overline{\mathbb{Q}}_p$, the ring of integers of $\mathbb{C}_p$
and its maximal ideal.
Fixing embeddings $\overline{\mathbb{Q}}\hookrightarrow \mathbb{C}$,\ $\overline{\mathbb{Q}}\hookrightarrow \mathbb{C}_p$, we identify $\overline{\mathbb{Q}}$ simultaneously as a subset of $\mathbb{C}$ and $\mathbb{C}_p$. 
Denote by $|\cdot |_p$ the $p$-adic absolute value, and by $\mu_c$ 
the group of $c$\,th roots of unity in $\mathbb{C}_p$ for $c\in\mathbb{N}$. 
Let $\textbf{P}^{1}(\mathbb{C}_p)=\mathbb{C}_p \cup \{\infty\}$ with $|\infty|_p=\infty$. 

Let 
$\mathbb{Z}_p^\times$ be the unit group in $\mathbb{Z}_p$ and 
\begin{equation*}
W:=
\begin{cases}
\{\pm 1\} & (\text{if\ }p=2)\\
\{ \xi\in \mathbb{Z}_p^\times\,|\,\xi^{p-1}=1\} & (\text{if\ }p\geqslant 3),
\end{cases}
\end{equation*}
\begin{equation}
q:=
\begin{cases}
4 & (\text{if\ }p=2)\\
p & (\text{if\ }p\geqslant 3).
\end{cases}
\label{q-def}
\end{equation}
Then it is well-known that 
$W$ forms a set of complete representatives of $\left(\mathbb{Z}_p/q\mathbb{Z}_p\right)^\times$, and 
$$\mathbb{Z}_p^\times = W\times (1+q\mathbb{Z}_p),$$
where we denote this correspondence by $x=\omega(x)\cdot\langle x \rangle$ for $x\in \mathbb{Z}_p^\times$ (see \cite[Section\,3]{Iwasawa72}). 
Noting $\omega(m)\equiv m\ ({\rm mod\ }q\mathbb{Z}_p)$ for $(m,q)=1$, we see that 
\begin{equation}
\omega\,:\,\left(\mathbb{Z}/q\mathbb{Z}\right)^{\times} \to W \label{Teich-Char}
\end{equation}
is the primitive Dirichlet character of conductor $q$, which is called the 
{\it Teichm\"uller character}.


Now we recall the Kubota-Leopoldt $p$-adic $L$-function which is defined by use of $p$-adic integral with respect to $p$-adic measure (for the details, see \cite[Chapter 2]{Kob}; also \cite{Kob79},~\cite[Chapter 5]{Wa}). 
Following Koblitz \cite[Chapter 2]{Kob} we consider a $p$-adic measure $\mk_z$ 
on $\mathbb{Z}_p$ for
$z\in\textbf{P}^{1}(\mathbb{C}_p)$ with $|z-1|_p\geqslant 1$  by
\begin{align}
\mk_z\left(j+p^N\mathbb{Z}_p\right)&:=\frac{z^{j}}{1-z^{p^N}}
\quad\in {\mathcal O}_{\mathbb{C}_p}
\quad (0\leqslant j\leqslant p^N-1). \label{Koblitz-measure}
\end{align}
Corresponding to this measure we can define the $p$-adic integral 
$$\int_{\mathbb{Z}_p} f(x)d{\mathfrak{m}}_z(x)  :=\lim_{N\to \infty}\sum_{a=0}^{p^N-1} f(a) \mk_z\left(a+p^N\mathbb{Z}_p\right)
\quad\in  \mathbb{C}_p$$
for any continuous function $f\,:\,\mathbb{Z}_p \to \mathbb{C}_p$. 
The $p$-adic integral over a compact open subset $U$ of $\mathbb{Z}_p$ is defined by
$$
\int_{U} f(x)d{\mathfrak{m}}_z(x)  :=
\int_{\mathbb{Z}_p} f(x)\chi_U(x)d{\mathfrak{m}}_z(x) \quad\in \mathbb{C}_p
$$
where $\chi_U(x)$ is the characteristic function of $U$.

For an arbitrary $t\in \mathbb{C}_p$ with $|t|_p<p^{-1/(p-1)}$, we see that $e^{tx}$ is a continuous function in $x\in \mathbb{Z}_p$. We obtain the following proposition by regarding \eqref{def-tw-Ber} as the Maclaurin expansion of $\hc(t;\xi)$ in $\mathbb{C}_p$ and using Koblitz' result \cite[Chapter 2,\ Equations (1.4), (2.4) and (2.5)]{Kob}. 

\begin{proposition}\label{Kob-Ch2}
For any primitive root of unity $\xi\neq 1\in \mathbb{C}_p$ of order prime to $p$,
\begin{equation}
\hc(t;\xi)=\int_{\mathbb{Z}_p}e^{tx}d\mk_\xi(x)
\quad\in \mathbb{C}_p[[t]]. \label{gene-01}
\end{equation}
Therefore
\begin{equation}
\int_{\mathbb{Z}_p}x^n d\mk_\xi(x)=\aa_n(\xi)
\qquad (n\in \mathbb{N}_0). \label{gene-02}
\end{equation}
\end{proposition}

For any $c\in \mathbb{N}_{>1}$ with $(c,p)=1$,
we consider
\begin{align}
\widetilde{\mathfrak{m}}_c\left(j+p^N\mathbb{Z}_p\right)&:=\sum_{\xi \in \mathbb{C}_p\setminus \{1\} \atop \xi^{c}=1}\mk_\xi\left(j+p^N\mathbb{Z}_p\right), \label{KM-measure}
\end{align}
which is a $p$-adic measure on $\mathbb{Z}_p$ introduced by Mazur (see \cite[Chapter 2]{Kob}). 
Considering the Galois action of ${\rm Gal}(\overline{\mathbb{Q}}_p/{\mathbb{Q}}_p)$, we see that $\widetilde{\mathfrak{m}}_c$ is a $\mathbb{Q}_p$-valued measure, therefore a $\mathbb{Z}_p$-valued measure.

Note that 
\begin{equation}
\begin{split}
\mm\left(pj+p^N\mathbb{Z}_p\right)&=\sum_{\xi \in \mathbb{C}_p\setminus \{1\} \atop \xi^{c}=1}\frac{\xi^{pj}}{1-\xi^{p^N}}\\
&=\sum_{\rho \in \mathbb{C}_p\setminus \{1\} \atop \rho^{c}=1}\frac{\rho^{j}}{1-\rho^{p^{N-1}}}=\mm\left(j+p^{N-1}\mathbb{Z}_p\right),
\end{split}
\label{2-0}
\end{equation}
because $(c,p)=1$. 
From \eqref{2-0-1} and \eqref{gene-02}, we have the following. 

\begin{lemma}[{cf. \cite[Chapter 2,\,(2.6)]{Kob}}]
For any $c\in \mathbb{N}_{>1}$ with $(c,p)=1$,
\begin{align}
\int_{\mathbb{Z}_p} x^{n}d\mm(x) &=\left(c^{n+1}-1\right)\zeta(-n)=\left(1-c^{n+1}\right)\frac{B_{n+1}}{n+1}\in \mathbb{Q}\quad  (n\in \mathbb{N}), \label{2-1}
\end{align}
and
\begin{equation}
\int_{\mathbb{Z}_p} 1\,d\mm(x) =\frac{c-1}{2}\in \mathbb{Z}_p. \label{2-1-2}
\end{equation}
Note that $c$ is odd when $p=2$.
\end{lemma}

Here we define the {\bf Kubota-Leopoldt $p$-adic $L$-function}.

\begin{definition}[{\cite[Chapter 2, (4.5)]{Kob}}] \label{Def-Kubota-L}
For $k\in \mathbb{Z}$ and  $c\in \mathbb{N}_{>1}$ with $(c,p)=1$,
\begin{equation*}
\begin{split}
L_p(s;\omega^{k})& :=\frac{1}{\langle c \rangle^{1-s}\omega^{k}(c)-1}\int_{\mathbb{Z}_p^\times}\langle x \rangle^{-s}\omega^{k-1}(x)d\mm(x)
\quad \in \mathbb{C}_p
\end{split}
\end{equation*}
for  $s\in \mathbb{C}_p$ with $|s|_p<qp^{-1/(p-1)}$.
\end{definition}

\begin{remark}
We note that $\langle x \rangle^{-s}$ can be defined as an ${\mathcal O}_{\mathbb{C}_p}$-valued rigid-analytic function (cf. Notation \ref{rigid-basics} below) in $s\in \mathbb{C}_p$ with $|s|_p<qp^{-1/(p-1)}$ (see \cite[p.\,54]{Wa}). 
\end{remark}

\begin{proposition}[{\cite[Chapter 2, Theorem]{Kob}, \cite[Theorem 5.11 and Corollary 12.5]{Wa}}] 
For $k\in \mathbb{Z}$, $L_p(s;\omega^{k})$ is rigid-analytic in the sense of Notation 
\ref{rigid-basics} below (except for a pole at $s=1$ when $k \equiv 0$ mod $p-1$) and satisfies that
\begin{align}
L_p(1-m;\omega^k)& =\left(1-\omega^{k-m}(p)p^{m-1}\right)L\left(1-m,\omega^{k-m}\right)\notag\\
& =-\left(1-\omega^{k-m}(p)p^{m-1}\right)\frac{B_{m,\omega^{k-m}}}{m} \quad (m\in \mathbb{N}). \label{p-L-vals}
\end{align}
\end{proposition}

We know that $B_{m,\omega^{k-m}}=0$ $(m\in \mathbb{N})$ when $k$ is odd. Hence we obtain the following.

\begin{proposition}[{\cite[p.57, Remark]{Wa}}] \label{Prop-zero}
When $k$ is any odd integer, $L_p(s;\omega^k)$ is the zero-function.
\end{proposition}

\begin{remark}\label{Rem-zero-ft}
The above fact will be generalized to the multiple case as functional relations between the $p$-adic multiple $L$-functions defined in the following section (see Theorem \ref{T-6-1} and Example \ref{Rem-zero-2}).
\end{remark}

\subsection{Construction and properties of $p$-adic multiple $L$-functions}\label{sec-3-2}

We defined a certain $p$-adic double $L$-function associated with the {Teichm\"uller character} for any odd prime $p$ in \cite{KMT-IJNT}. 
In this section, as its generalization with slight modification, 
we will introduce a $p$-adic multiple $L$-functions 
${L_{p,r}(s_1,\ldots,s_r;\omega^{k_1},\ldots,\omega^{k_r};\gamma_1,\ldots,\gamma_{r};c)}$
in Definition \ref{Def-pMLF}
and investigate their basic properties:
the $p$-adic rigid analyticity with respect to  the parameter 
$s_1,\ldots,s_r$ (Theorem \ref{Th-pMLF})
and the $p$-adic continuity with respect to  the parameter 
$c$ (Theorem \ref{continuity theorem})
which yields a non-trivial $p$-adic property of
their special values (Corollary \ref{non-trivial property}).
At the end of this subsection we will give a discussion toward a construction of
a $p$-adic analogue of the entire zeta-function
$\zeta^{\rm des}_r(s_1,\ldots,s_r;\gamma_1,\ldots,\gamma_{r})$
previously constructed in Subsection \ref{c-1-zeta}.

Let $r\in \mathbb{N}$ and $\gamma_1,\ldots,\gamma_r\in \mathbb{Z}_p$. 
Let $\xi_1,\ldots,\xi_r\in \mathbb{C}_p$ be roots of unity other than $1$
whose orders are prime to $p$.

Then, combining \eqref{gene-01} and \eqref{Def-Hr}, we have
\begin{align*}
    & \mathfrak{H}_r  (( t_j ),( \xi_j), ( \gamma_j))=
    \prod_{j=1}^{r} \int_{\mathbb{Z}_p}e^{\gamma_j (\sum_{k=j}^r t_k)x_j}d\mk_{\xi_j}(x_j)\\
    & =\int_{\mathbb{Z}_p^r} \exp\left( \sum_{k=1}^{r} \left(\sum_{j=1}^{k}x_j\gamma_j\right)t_k\right)\prod_{j=1}^{r}d\mk_{\xi_j}(x_j)
    \\
    & =\sum_{n_1,\ldots,n_r=0}^\infty \int_{\mathbb{Z}_p^r} \prod_{k=1}^{r}\left( \sum_{j=1}^{k}x_j\gamma_j \right)^{n_k}\prod_{j=1}^r d\mk_{\xi_j}(x_j)\frac{t_1^{n_1}\cdots t_r^{n_r}}{n_1!\cdots n_r!}.
\end{align*}
Hence as a multiple analogue of Proposition \ref{Kob-Ch2}
we have the following.

\begin{proposition}\label{T-multi-2}
For $n_1,\ldots,n_r\in \mathbb{N}_0$, $\gamma_1,\ldots,\gamma_r\in \mathbb{Z}_p$ and
roots $\xi_1,\ldots,\xi_r\in\mathbb{C}_p^\times\setminus\{1\}$ of unity
whose orders are prime to $p$,
\begin{equation}
\int_{\mathbb{Z}_p^r} \prod_{k=1}^{r}( \sum_{j=1}^{k}x_j\gamma_j )^{n_k}\prod_{j=1}^r d\mk_{\xi_j}(x_j)=\aa((n_j);( \xi_j);( \gamma_j)). \label{multi-val2}
\end{equation}
\end{proposition}

We set
\begin{align}
\left( \mathbb{Z}_p^r\right)'_{\{\gamma_j\}}:=\bigg\{ (x_j)\in \mathbb{Z}_p^r& \,\bigg|\,p\nmid x_1\gamma_1,\ p\nmid (x_1\gamma_1+x_2\gamma_2),\ldots,\ p\nmid \sum_{j=1}^{r}x_j\gamma_j\,\bigg\}, \label{region}
\end{align}
and 
\label{pMLF-region}
\begin{equation}
\mathfrak{X}_r(d):=\left\{(s_1,\ldots,s_r)\in \mathbb{C}_p^r\,\big|\,|s_j|_p<d^{-1}p^{-1/(p-1)}\ (1\leqslant j\leqslant r)\right\}\label{region-d}
\end{equation}
for $d\in \mathbb{R}_{>0}$.

\begin{definition} \label{Def-pMLF}
For $r\in \mathbb{N}$, $k_1,\ldots,k_r,\in \mathbb{Z}$,
$\gamma_1,\ldots,\gamma_r\in \mathbb{Z}_p$
and $c\in \mathbb{N}_{>1}$ with $(c,p)=1$,
the {\bf $p$-adic multiple $L$-function} of depth $r$ is the $\mathbb{C}_p$-valued function
on  $(s_j)\in \mathfrak{X}_r\left(q^{-1}\right)$  defined by
\begin{equation}
\begin{split}
&{L_{p,r}(s_1,\ldots,s_r;\omega^{k_1},\ldots,\omega^{k_r};\gamma_1,\ldots,\gamma_{r};c)}\\
&\quad :=\int_{\left( \mathbb{Z}_p^r\right)'_{\{\gamma_j\}}}\langle x_1\gamma_1 \rangle^{-{s_1}}\langle x_1\gamma_1+ x_2\gamma_2 \rangle^{-{s_2}}\cdots \langle \sum_{j=1}^{r}x_{j}\gamma_{j} \rangle^{-{s_r}}\\
& \qquad \qquad \times \omega^{k_1}(x_1\gamma_1)\cdots\omega^{k_r}( \sum_{j=1}^{r}x_{j}\gamma_{j}) \prod_{j=1}^{r}d\mm(x_j).
\end{split}
\label{e-6-1}
\end{equation}
\end{definition}

\begin{remark}\label{Rem-gamma1}
In the case $\gamma_1\in p\mathbb{Z}_p$, we see that $\left( \mathbb{Z}_p^r\right)'_{\{\gamma_j\}}$ is an empty set, hence we regard ${L_{p,r}((s_j);(\omega^{k_j});(\gamma_j);c)}$ as the zero-function.
\end{remark}

\begin{remark}
Note that we can define, more generally, 
\begin{align*}
{L_{p,r}(s_1,\ldots,s_r;(\omega^{k_j});(\gamma_j);(c_j))}
&:=\int_{\left( \mathbb{Z}_p^r\right)'_{\{\gamma_j\}}}\langle x_1\gamma_1 \rangle^{-{s_1}}\langle x_1\gamma_1+ x_2\gamma_2 \rangle^{-{s_2}}\cdots \langle \sum_{j=1}^{r}x_{j}\gamma_{j} \rangle^{-{s_r}}\\
& \qquad \qquad \times \omega^{k_1}(x_1\gamma_1)\cdots\omega^{k_r}( \sum_{j=1}^{r}x_{j}\gamma_{j}) \prod_{j=1}^{r}d\widetilde{\mathfrak{m}}_{c_j}(x_j)
\end{align*}
for $c_j\in \mathbb{N}_{>1}$ with $(c_j,p)=1$ $(1\leqslant j \leqslant r)$. 
Then we can naturally generalize the following arguments in the remaining sections.
\end{remark}

In the next section, 
we will see that the above $p$-adic multiple $L$-function can
be seen as a $p$-adic interpolation of
a certain finite sum \eqref{Int-P} of complex multiple zeta functions
(cf. Theorem \ref{T-main-1} and Remark \ref{rem-intpln}).

The following example shows that the Kubota-Leopoldt $p$-adic $L$-function
is essentially a special case of our $p$-adic multiple $L$-function with $r=1$.

\begin{example}\label{example for r=1}
For $s\in\mathbb{C}_p$ with $|s|_p<qp^{-1/(p-1)}$, $\gamma_1\in \mathbb{Z}_p^\times $, $k\in \mathbb Z$
and $c\in \mathbb{N}_{>1}$ with $(c,p)=1$,
we have
\begin{equation}
\begin{split}
L_{p,1}(s;\omega^{k-1};\gamma_1;c)& =\int_{\mathbb{Z}_p^\times}\langle x\gamma_1\rangle^{-{s}} \omega^{k-1}(x\gamma_1)d\mm(x)\\
& =\langle \gamma_1\rangle^{-s} \omega^{k-1}(\gamma_1) (\langle c\rangle^{1-s}\omega^k(c)-1)\cdot L_p(s;\omega^k). 
\end{split}
\label{1-p-LF-gamma}
\end{equation}
\end{example}

The next example shows that 
when $r=2$, we recover the notion of the $p$-adic double $L$-function
which has been studied by the second-, the third- and the fourth-named authors.

\begin{example}\label{example for r=2}
Let $p$ be an odd prime, $r=2$, $c=2$ and $\eta\in p\mathbb{Z}_p$. 
Then, for $s_1,s_2\in\mathbb{C}_p$ with $|s_j|_p<p^{1-1/(p-1)}$ $(j=1,2)$ and $k_1,k_2\in \mathbb Z$,
our $p$-adic $L$-function is given by
\begin{equation}
\begin{split}
& L_{p,2}(s_1,s_2;\omega^{k_1},\omega^{k_2};1,\eta;2)\\
& \quad =\int_{\mathbb{Z}_p^\times\times \mathbb{Z}_p}\langle x_1 \rangle^{-{s_1}} \langle x_1+x_2\eta \rangle^{-{s_2}} \omega^{k_1}(x_1)\omega^{k_2}(x_1+x_2\eta)d\m2(x_1)d\m2(x_2),
\end{split}
\label{double-p-LF-eta}
\end{equation}
which is the $p$-adic double $L$-function introduced in \cite{KMT-IJNT}.
\end{example}

The next theorem claims that our $p$-adic multiple $L$-function is rigid-analytic
(cf. Notation \ref{rigid-basics} below)
with respect to the parameters $s_1,\ldots,s_r$:

\begin{theorem}\label{Th-pMLF}
Let $k_1,\ldots,k_r\in \mathbb{Z}$,
$\gamma_1,\ldots,\gamma_r\in \mathbb{Z}_p$
and $c\in \mathbb{N}_{>1}$ with $(c,p)=1$.
Then 
$L_{p,r}((s_j);(\omega^{k_j});(\gamma_j);c)$
has the following expansion
\begin{align*}
&{L_{p,r}(s_1,\ldots,s_r;\omega^{k_1},\ldots,\omega^{k_r};\gamma_1,\ldots,\gamma_{r};c)}\\
&=\sum_{n_1,\ldots,n_r=0}^\infty \cc\left(n_1,\ldots,n_r;(\omega^{k_j});(\gamma_j);c\right)s_1^{n_1}\cdots s_r^{n_r}
\end{align*}
for $(s_j)\in \mathfrak{X}_r\left(q^{-1}\right)$, where $\cc\left((n_j);(\omega^{k_j});(\gamma_j);c\right)\in \mathbb{Z}_p$ satisfies 
\begin{equation}
\left| \cc\left(n_1,\ldots,n_r;(\omega^{k_j});(\gamma_j);c\right)\right|_p\leqslant \left(qp^{-1/(p-1)}\right)^{-n_1-\cdots -n_r}. \label{C-esti}
\end{equation}
\end{theorem}

In order to prove this theorem, we prepare the following two lemmas which are generalizations of \cite[Proposition 5.8 and its associated lemma]{Wa}. 

\begin{lemma}\label{multi-lemma1}
Let $f$ be a continuous function from $\mathbb{Z}_p^r$ to $\mathbb{Q}_p$ defined by 
$$f(X_1,\ldots,X_r)=\sum_{n_1,\ldots,n_r=0}^\infty a(n_1,\ldots,n_r)\prod_{j=1}^{r}\binom{X_j}{n_j},$$
where $a(n_1,\ldots,n_r)\in \mathbb{Z}_p$. Suppose there exist $d,M\in \mathbb{R}_{>0}$ with $d<p^{-1/(p-1)}<1$ such that $|a(n_1,\ldots,n_r)|_p\leqslant Md^{\sum_{j=1}^r n_j}$ for any $(n_j)\in \mathbb{N}_0^r$. Then $f(X_1,\ldots,X_r)$ may be expressed as 
$$f(X_1,\ldots,X_r)=\sum_{n_1,\ldots,n_r=0}^\infty C(n_1,\ldots,n_r)X_1^{n_1}\cdots X_r^{n_r}\in \mathbb{Q}_p[[X_1,\ldots,X_r]]$$
which converges absolutely in $\mathfrak{X}_r(d)$, 
where 
$$\left|C(n_1,\ldots,n_r)\right|_p\leqslant M(d^{-1}p^{-1/(p-1)})^{-n_1-\cdots -n_r}.$$
\end{lemma}

\begin{lemma}\label{multi-lemma2}
Let 
$$P_l(X_1,\ldots,X_r)=\sum_{n_1,\ldots,n_r\geqslant 0}a(n_1,\ldots,n_r;l)\prod_{j=1}^{r}X_j^{n_j}\ \ (l\in \mathbb{N}_0)$$
be a sequence of power series with $\mathbb{C}_p$-coefficients which converges in a fixed subset $D$ of $\mathbb{C}_p^r$ and suppose 
\begin{enumerate}[{\rm (i)}]
\item when $l\to \infty$, $a(n_1,\ldots,n_r;l)\to a(n_1,\ldots,n_r;0)$ for each $(n_j)\in \mathbb{N}_0^r$,
\item for each $(X_j)\in D$ and any $\varepsilon>0$, there exists an $N_0=N_0((X_j),\varepsilon)$ such that 
$$\left|\sum_{(n_j)\in \mathbb{N}_0^r \atop \sum n_j >N_0}a(n_1,\ldots,n_r;l)\prod_{j=1}^{r}X_j^{n_j}\right|_p<\varepsilon$$
uniformly in $l\in \mathbb{N}$. 
\end{enumerate}
Then $P_l((X_j))\to P_0((X_j))$ as $l\to \infty$ for any $(X_j)\in D$.
\end{lemma}

\begin{proof}[Proof of Lemma \ref{multi-lemma2}] 
For any $\varepsilon$ and $(X_j)$, we can choose $N_0$ as above. Then 
$$|P_l((X_j))- P_0((X_j))|_p\leqslant \max_{\sum n_j\leqslant N_0}\left\{\varepsilon, \left|a(n_1,\ldots,n_r;0)-a(n_1,\ldots,n_r;l)\right|_p\cdot \prod_{j=1}^{r}|X_j|_p^{n_j}\right\}\leqslant \varepsilon$$
for any sufficiently large $l$. 
\end{proof}

\begin{proof}[Proof of Lemma \ref{multi-lemma1}] 
We can write
$$\binom{X}{n}=\frac{1}{n!}\sum_{m=0}^n b(n,m)X^m\qquad (n\in \mathbb{N}_0),$$
where $b(n,m)\in \mathbb{Z}$. For any $l\in \mathbb{N}$, let 
\begin{align*}
\pp_l(X_1,\ldots,X_r)&=\sum_{n_1,\ldots,n_r\geqslant 0 \atop n_1+\cdots +n_r\leqslant l}a(n_1,\ldots,n_r)\prod_{j=1}^{r}\binom{X_j}{{n_j}}\\
& =\sum_{0\leqslant m_1,\ldots,m_r\leqslant l \atop m_1+\cdots +m_r\leqslant l}C(m_1,\ldots,m_r;l)\prod_{j=1}^{r}X_j^{m_j},
\end{align*}
say. We aim to show that $\pp_l(X_1,\ldots,X_r)$ satisfies the conditions (i) and (ii) for $P_l(X_1,\ldots,X_r)$ in Lemma \ref{multi-lemma2}. 
In fact, using the notation 
\begin{equation*}
\lambda_n(m)=
\begin{cases}
1 & (m\leqslant n)\\
0 & (m>n),
\end{cases}
\end{equation*}
we have
\begin{align*}
\pp_l(X_1,\ldots,X_r)&=\sum_{n_1+\cdots+n_r\leqslant l}\frac{a(n_1,\ldots,n_r)}{n_1!\cdots n_r!}\sum_{m_1\leqslant n_1}\cdots \sum_{m_r\leqslant n_r}\prod_{j=1}^{r}b(n_j,m_j)X_j^{m_j}\\
&=\sum_{n_1+\cdots+n_r\leqslant l}\frac{a(n_1,\ldots,n_r)}{n_1!\cdots n_r!}\sum_{m_1\leqslant l}\cdots \sum_{m_r\leqslant l \atop m_1+\cdots+m_r \leqslant l}\prod_{j=1}^{r}b(n_j,m_j)\lambda_{n_j}(m_j)X_j^{m_j}\\
&=\sum_{m_1,\ldots,m_r\leqslant l \atop m_1+\cdots +m_r\leqslant l}\left(\sum_{m_1\leqslant n_1,\ldots,m_r\leqslant n_r \atop n_1+\cdots+n_r\leqslant l}\frac{a(n_1,\ldots,n_r)}{n_1!\cdots n_r!}\prod_{j=1}^{r}b(n_j,m_j)\right)\prod_{j=1}^{r}X_j^{m_j}.
\end{align*}
Hence we have
$$C(m_1,\ldots,m_r;l)=\sum_{m_1\leqslant n_1,\ldots,m_r\leqslant n_r \atop n_1+\cdots+n_r\leqslant l}\frac{a(n_1,\ldots,n_r)}{n_1!\cdots n_r!}\prod_{j=1}^{r}b(n_j,m_j).$$
Therefore, noting $b(n_j,m_j)\in \mathbb{Z}$ and using $|n!|_p>p^{-n/(p-1)}$, we obtain
\begin{align}
|C(m_1,\ldots,m_r;l)|_p& \leqslant \max_{m_1\leqslant n_1,\ldots,m_r\leqslant n_r \atop n_1+\cdots+n_r\leqslant l}\frac{Md^{\sum n_j}}{|n_1!\cdots n_r!|_p}\notag\\
& \leqslant \max_{m_1\leqslant n_1,\ldots,m_r\leqslant n_r \atop n_1+\cdots+n_r\leqslant l}M\left(dp^{1/(p-1)}\right)^{\sum n_j}\leqslant M\left(d^{-1}p^{-1/(p-1)}\right)^{-m_1-\cdots -m_r}\label{5-4-1}.
\end{align}
Furthermore, we have
\begin{align*}
& |C(m_1,\ldots,m_r;l+k)-C(m_1,\ldots,m_r;l)|_p\\
& \leqslant \left|\sum_{m_1\leqslant n_1,\ldots,m_r\leqslant n_r \atop l<n_1+\cdots+n_r\leqslant l+k}\frac{a(n_1,\ldots,n_r)}{n_1!\cdots n_r!}\prod_{j=1}^{r}b(n_j,m_j)\right|_p\\
& \leqslant \max_{m_1\leqslant n_1,\ldots,m_r\leqslant n_r \atop l<n_1+\cdots+n_r\leqslant l+k}\frac{Md^{\sum n_j}}{|n_1!\cdots n_r!|_p}\notag\\
& \leqslant \max_{m_1\leqslant n_1,\ldots,m_r\leqslant n_r \atop l<n_1+\cdots+n_r\leqslant l+k}M\left(dp^{1/(p-1)}\right)^{\sum n_j}\leqslant M\left(d^{-1}p^{-1/(p-1)}\right)^{-l-1}
\end{align*}
for any 
$l\in \mathbb{N}$. 
Hence $\{C(m_1,\ldots,m_r;l)\}$ is a Cauchy sequence in $l$. Therefore there exists
$$C(m_1,\ldots,m_r;0)=\lim_{l\to\infty} C(m_1,\ldots,m_r;l)\in \mathbb{Q}_p\quad ((m_j)\in \mathbb{N}_0)$$
with $|C(m_1,\ldots,m_r;0)|_p\leqslant M(d^{-1}p^{-1/(p-1)})^{-m_1-\cdots -m_r}$. 
For $(X_j)\in \mathfrak{X}_r(d)$ defined by \eqref{region-d}, let
$$\pp_0(X_1,\ldots,X_r)=\sum_{n_1,\ldots,n_r\geqslant 0}C(n_1,\ldots,n_r;0)X_1^{n_1}\cdots X_r^{n_r}.$$
Then $\pp_0(X_1,\ldots,X_r)$ converges absolutely in $\mathfrak{X}_r(d)$. Moreover, by \eqref{5-4-1}, 
we have
\begin{align*}
& \left|\sum_{n_1,\ldots,n_r\geqslant 0 \atop n_1+\cdots +n_r\geqslant N_0}C(n_1,\ldots,n_r;l)X_1^{n_1}\cdots X_r^{n_r}\right|_p \\
& \quad \leqslant \max_{n_1+\cdots +n_r\geqslant N_0}\left\{M\prod_{j=1}^{r}\left(d^{-1}p^{-1/(p-1)}\right)^{-n_j}|X_j|_p^{n_j}\right\} \to 0\quad (n_1,\ldots,n_r\to \infty)
\end{align*}
uniformly in $l$, for $(X_j)\in \mathfrak{X}_r(d)$. Therefore, by Lemma \ref{multi-lemma2}, we obtain 
$$f(X_1,\ldots,X_r)=\lim_{l\to \infty}\pp_l(X_1,\ldots,X_r)=\pp_0(X_1,\ldots,X_r)$$
for $(X_j)\in \mathfrak{X}_r(d)$. Thus we complete the proof of Lemma \ref{multi-lemma1}.
\end{proof}

\begin{proof}[Proof of Theorem \ref{Th-pMLF}] 
Considering the binomial expansion in \eqref{e-6-1}, we have, for $(s_j)\in \mathfrak{X}_r=\mathfrak{X}_r\left(q^{-1}\right)$, 
\begin{align*}
&{L_{p,r}(s_1,\ldots,s_r;\omega^{k_1},\ldots,\omega^{k_r};\gamma_1,\ldots,\gamma_{r};c)}\\
&\quad =\sum_{n_1,\ldots,n_r=0}^\infty \bigg\{\int_{\left( \mathbb{Z}_p^r\right)'_{\{\gamma_j\}}}\prod_{\nu=1}^{r} \left(\langle \sum_{j=1}^{\nu}x_{j}\gamma_{j} \rangle-1\right)^{n_\nu}\\
& \qquad \qquad \times \omega^{k_1}(x_1\gamma_1)\cdots\omega^{k_r}( \sum_{j=1}^{r}x_{j}\gamma_{j}) \prod_{j=1}^{r}d\mm(x_j)\bigg\}\prod_{j=1}^{r}\binom{-s_j}{n_j}\\
& \quad =\sum_{n_1,\ldots,n_r=0}^\infty a(n_1,\ldots,n_r)\prod_{j=1}^{r}\binom{-s_j}{n_j},
\end{align*}
say. Since $\langle \sum_{j=1}^{\nu}x_{j}\gamma_{j} \rangle\equiv 1$ mod $q$, we have $|a(n_1,\ldots,n_r)|_p\leqslant q^{-\sum n_j}$. 
Applying Lemma \ref{multi-lemma1} with $d=q^{-1}$ and $M=1$, we obtain the proof of Theorem \ref{Th-pMLF}. 
\end{proof}


Next, we discuss the $p$-adic continuity of our $p$-adic multiple $L$-function
\eqref{e-6-1} with respect to the parameter $c$. 

\begin{theorem}\label{continuity theorem}
Let 
$s_1,\ldots,s_r\in \mathbb{Z}_p$,
$k_1,\ldots,k_r\in \mathbb{Z}$,
$\gamma_1,\ldots,\gamma_r\in\mathbb{Z}_p$ and
$c\in \mathbb{N}_{>1}$ with $(c,p)=1$.
Then 
the map 
$$c\mapsto{L_{p,r}(s_1,\ldots,s_r;\omega^{k_1},\ldots,\omega^{k_r};\gamma_1,\ldots,\gamma_{r};c)}$$
is continuously extended to $c\in\mathbb{Z}_p^\times$
as a $p$-adic continuous function.

Moreover the extension is uniformly continuous
with respect to both parameters $c$ and $(s_j)$.
Namely, 
for any given $\varepsilon>0$,
there always exists $\delta>0$ such that 
$$
|L_{p,r}(s_1,\ldots,s_r;\omega^{k_1},\ldots,\omega^{k_r};\gamma_1,\ldots,\gamma_{r};c)
-L_{p,r}(s'_1,\ldots,s'_r;\omega^{k_1},\ldots,\omega^{k_r};\gamma_1,\ldots,\gamma_{r};c')|_p
<\varepsilon
$$
holds for any $c, c'\in\mathbb Z_p^\times$ with $|c-c'|_p<\delta$
and $s_j, s'_j\in\mathbb Z_p$ with $|s_j-s_j'|_p<\delta$
($1\leqslant j\leqslant r$).
\end{theorem}

By \cite[Subsection\,12.2]{Wa}, the space ${\rm Meas}_{\mathbb Z_p}({\mathbb Z_p})$
of  $\mathbb Z_p$-valued measures on $\mathbb{Z}_p$ 
is identified with the completed group algebra  ${\mathbb Z_p}[[{\mathbb Z_p}]]$;
\begin{equation}\label{identificationI}
{\rm Meas}_{\mathbb Z_p}({\mathbb Z_p})\simeq
{\mathbb Z_p}[[{\mathbb Z_p}]].
\end{equation}
Again by loc.\,cit. Theorem 7.1,
it is identified with the one parameter formal power series ring 
${\mathbb Z_p}[[T]]$ by sending $1\in {\mathbb Z_p}[[{\mathbb Z_p}]]$ to $1+T\in {\mathbb Z_p}[[T]]$;
\begin{equation}\label{identificationII}
{\mathbb Z_p}[[{\mathbb Z_p}]]\simeq{\mathbb Z_p}[[T]].
\end{equation}

By loc.\,cit.\,Subsection\,12.2, we obtain the following.

\begin{lemma}
Let $c\in \mathbb{N}_{>1}$ with $(c,p)=1$.
By the above correspondences,
$\widetilde{\mathfrak{m}}_c\in {\rm Meas}_{\mathbb Z_p}({\mathbb Z_p})$ 
corresponds to
\begin{equation}
g_c(T)=\frac{1}{(1+T)-1}-\frac{c}{(1+T)^c-1}\in {\mathbb Z_p}[[T]].
\end{equation}
\end{lemma}


\begin{lemma}
The map $c\mapsto g_c(T)$ is uniquely extended into
a $p$-adic continuous function
$g:\mathbb Z_p^\times\to{\mathbb Z_p}[[T]]$.
\end{lemma}

\begin{proof}
The series $(1+T)^c$ makes sense in ${\mathbb Z_p}[[T]]$
for $c\in\mathbb Z_p$ and moreover continuous with respect to 
 $c\in\mathbb Z_p$
because we have
$$
{\mathbb Z_p}[[T]]\simeq\underset{\leftarrow}{\lim} \ 
{\mathbb Z_p}[T]/
\bigl((1+T)^{p^N}-1\bigr)
$$
(cf. \cite[Theorem 7.1]{Wa}).
When $c\in\mathbb Z_p^\times$, 
$\frac{c}{(1+T)^c-1}$ belongs to $\frac{1}{T}+{\mathbb Z_p}[[T]]$.
Hence $g_c(T)$ belongs to  ${\mathbb Z_p}[[T]]$.
\end{proof}

For $c\in\mathbb Z_p^\times$
we denote 
$\widetilde{\mathfrak{m}}_c$
to be the $\mathbb Z_p$-valued measure on $\mathbb{Z}_p$ 
which corresponds to $g_c(T)$
by \eqref{identificationI} and \eqref{identificationII}.
We note that it coincides with \eqref{KM-measure}
when $c\in \mathbb{N}_{>1}$ with $(c,p)=1$. 

\begin{remark}\label{measure extension}
For the parameters 
$s_1,\ldots,s_r\in \mathbb{C}_p$ with $|s_j|_p<qp^{-1/(p-1)}$ $(1\leqslant j\leqslant r)$,
$k_1,\ldots,k_r\in \mathbb{Z}$,
$\gamma_1,\ldots,\gamma_r\in\mathbb Z_p$ and
$c\in \mathbb Z_p^\times$,
the $p$-adic multiple $L$-function
$${L_{p,r}(s_1,\ldots,s_r;\omega^{k_1},\ldots,\omega^{k_r};\gamma_1,\ldots,\gamma_{r};c)}$$
is defined by \eqref{e-6-1}
with the above measure $\widetilde{\mathfrak{m}}_c$.
\end{remark}

Let $C( \mathbb{Z}^r_p, \mathbb{C}_p)$ be the $ \mathbb{C}_p$-Banach space
of continuous  $\mathbb{C}_p$-valued functions on $ \mathbb{Z}^r_p$
with $||f||:=\sup_{x\in\mathbb Z^r_p} |f(x)|_p$
and $\textrm{Step}( \mathbb{Z}^r_p)$ be the set of $ \mathbb{C}_p$-valued locally constant functions
on $ \mathbb{Z}^r_p$.
The subspace $\textrm{Step}( \mathbb{Z}^r_p)$ is dense in $C( \mathbb{Z}^r_p, \mathbb{C}_p)$ 
 (cf. \cite[Section\,12.1]{Wa}).

\begin{proposition}
For each function $f\in C( \mathbb{Z}_p, \mathbb{C}_p)$,
the map 
$$\Phi_f:\mathbb{Z}_p^\times\to\mathbb{C}_p$$
sending $c\mapsto\int_{\mathbb{Z}_p}f(z)d\widetilde{\mathfrak{m}}_c(z)$
is $p$-adically continuous.
\end{proposition}

\begin{proof}
First assume that $f=\chi_{j,N}$ where $\chi_{j,N}$
($0\leqslant j<p^N$, $N\in\mathbb N$)
is the characteristic function of
the set $j+p^N\mathbb Z_p$.
Then $\Phi_f(c)$ is calculated to be $a_j(g_c)$ (cf. \cite[Section\,12.2]{Wa}) which is defined by
$$
g_c(T)\equiv\sum_{i=0}^{p^N-1}a_i(g_c)(1+T)^i\pmod{(1+T)^{p^N}-1}.
$$ 
Since the map $c\mapsto g_c$ is continuous, 
$c\mapsto\Phi_f(c)=a_j(g_c)$ is continuous in this case.
This implies the continuity of $\Phi_f$
in $c$ when $f\in \textrm{Step} (\mathbb Z_p)$.

Next assume $f\in C( \mathbb{Z}_p, \mathbb{C}_p)$
and $c\in\mathbb Z_p^\times$.
Then for any  $g\in C( \mathbb{Z}_p, \mathbb{C}_p)$ and
$c'\in\mathbb Z_p^\times$,
we have
\begin{align*}
\Phi_f(c) & -\Phi_f(c')=
\int_{\mathbb{Z}_p}f(z)d\widetilde{\mathfrak{m}}_c(z)-
\int_{\mathbb{Z}_p}f(z)d\widetilde{\mathfrak{m}}_{c'}(z) \\
&=\int_{\mathbb{Z}_p}(f(z)-g(z))\cdot
d(\widetilde{\mathfrak{m}}_c(z)-\widetilde{\mathfrak{m}}_{c'}(z))
+\int_{\mathbb{Z}_p}g(z)\cdot
d(\widetilde{\mathfrak{m}}_c(z)-\widetilde{\mathfrak{m}}_{c'}(z)) \\
&=(\Phi_{f-g}(c) -\Phi_{f-g}(c')) +(\Phi_g(c)-\Phi_g(c')).
\end{align*}
Since $\textrm{Step}(\mathbb Z_p)$ is dense in $C( \mathbb{Z}_p, \mathbb{C}_p)$,
there exists $g_0\in \textrm{Step}(\mathbb Z_p)$ with $||f-g_0||<\varepsilon$
for any given $\varepsilon>0$.
So for any $c\in\mathbb Z_p^\times$
we have
$$
|\Phi_{f-g_0}(c)|_p=
\Bigl|\int_{\mathbb{Z}_p}(f(z)-g_0(z))\cdot d\widetilde{\mathfrak{m}}_c(z)\Bigr|_p
\leqslant 
||f-g_0|| 
<\varepsilon
$$
because  $\widetilde{\mathfrak{m}}_c$ is a $\mathbb Z_p$-valued measure.

On the other hand, since $g_0\in \textrm{Step}(\mathbb Z_p)$,
there exists a $\delta >0$ such that
$$
|\Phi_{g_0}(c)  -\Phi_{g_0}(c')|_p<\varepsilon
$$
holds for any $c'\in\mathbb Z_p^\times$ with $|c-c'|_p<\delta$.

Therefore for $f\in C( \mathbb{Z}_p, \mathbb{C}_p)$, $c\in\mathbb Z_p^\times$
and any $\varepsilon>0$, 
there always exists a $\delta >0$ such that
$$
|\Phi_f(c)  -\Phi_f(c')|_p\leqslant \max \left\{ |\Phi_{f-g_0}(c)|_p ,\,|\Phi_{f-g_0}(c')|_p,\,|\Phi_{g_0}(c)-\Phi_{g_0}(c')|_p\right\}
<\varepsilon 
$$
holds for any $c'\in\mathbb Z_p^\times$ with $|c-c'|_p<\delta$.
\end{proof}

By generalizing our arguments above, we obtain the following.

\begin{proposition}\label{continuity}
For each function $f\in C( \mathbb{Z}_p^r, \mathbb{C}_p)$,
the map 
$$\Phi^r_f:\mathbb{Z}_p^\times\to\mathbb{C}_p$$
sending $c\mapsto\int_{\mathbb{Z}_p^r}f(x_1,\ldots,x_r)
\prod_{j=1}^{r}d\widetilde{\mathfrak{m}}_c(x_j)
$
is $p$-adically continuous.
\end{proposition}

A proof of Theorem \ref{continuity theorem} is attained by the above proposition.
This is the reason why we restrict Theorem \ref{continuity theorem} 
to the case for $s_1,\ldots,s_r\in \mathbb{Z}_p$.

\begin{proof}[Proof of Theorem \ref{continuity theorem}]

Let us fix notation: Put
\begin{equation*}
\begin{split}
&
{W(s_1,\dots s_r;x_1,\dots,x_r)}
\\
&:=
\langle x_1\gamma_1 \rangle^{-{s_1}}\langle x_1\gamma_1+ x_2\gamma_2 \rangle^{-{s_2}}\cdots \langle \sum_{j=1}^{r}x_{j}\gamma_{j} \rangle^{-{s_r}}
\cdot
\omega^{k_1}(x_1\gamma_1)\cdots\omega^{k_r}( \sum_{j=1}^{r}x_{j}\gamma_{j}) \cdot
\chi_{_{\left( \mathbb{Z}_p^r\right)'_{\{\gamma_j\}}}}(x_1,\ldots,x_r)
\end{split}
\end{equation*}
where $\chi_{_{\left( \mathbb{Z}_p^r\right)'_{\{\gamma_j\}}}}(x_1,\ldots,x_r)$
is the characteristic function of 
${\left( \mathbb{Z}_p^r\right)'_{\{\gamma_j\}}}$ (cf. Definition \ref{Def-pMLF}).
Then we have
\begin{equation}\label{integration}
{L_{p,r}(s_1,\ldots,s_r;\omega^{k_1},\ldots,\omega^{k_r};\gamma_1,\ldots,\gamma_{r};c)}
:=\int_{\mathbb Z_p^r}
W(s_1,\dots ,s_r;x_1,\dots,x_r)\cdot
\prod_{j=1}^{r}d\widetilde{\mathfrak{m}}_c(x_j).
\end{equation}

Below we will prove that the map
$$
\varPsi:
(s_1,\ldots,s_r, c)\in \mathbb Z_p^r\times \mathbb Z_p^\times\mapsto
L_{p,r}(s_1,\ldots,s_r;\omega^{k_1},\ldots,\omega^{k_r};\gamma_1,\ldots,\gamma_{r};c)
\in \mathbb Z_p
$$
is continuous:

First, by Proposition \ref{continuity} we see that the function is continuous
with respect to $c\in\mathbb Z_p^\times$ for each fixed $(s_1,\dots,s_r)\in \mathbb Z_p^r$.
Namely, for each fixed $(s_1,\dots,s_r)\in \mathbb Z_p^r$ and $c\in \mathbb Z_p^\times$
and for any given $\varepsilon>0$,
there always exists $\delta>0$ such that
\begin{equation*}
|L_{p,r}(s_1,\ldots,s_r;\omega^{k_1},\ldots,\omega^{k_r};\gamma_1,\ldots,\gamma_{r};c)
-L_{p,r}(s_1,\ldots,s_r;\omega^{k_1},\ldots,\omega^{k_r};\gamma_1,\ldots,\gamma_{r};c')|_p<\varepsilon,
\end{equation*}
equivalently
\begin{equation}\label{bound1}
|\varPsi(s_1,\ldots,s_r; c)-\varPsi(s_1,\ldots,s_r; c')|_p<\varepsilon
\end{equation}
for all $c'\in\mathbb Z_p^\times$ with $|c-c'|_p<\delta$.

Next, take $M\in\mathbb N$ such that $\varepsilon> p^{-M}>0$.
Since
it is easy to see that there exists $\delta'>0$
such that
$$(1+u)^d\equiv 1\pmod{p^M}$$
for any $u\in p\mathbb Z_p$ and
$d\in\mathbb Z_p$ with $|d|_p<\delta'$
(actually you may take $\delta'$ for $\delta'<p^{-M}$),
we have
$$x^s\equiv x^{s'}\pmod{p^M}$$
for all  $x \in (1+p\mathbb Z_p)$ and
$s,s'\in\mathbb Z_p$
with $|s-s'|<\delta'$ for such $\delta'$. 
Therefore
$$
W(s_1,\dots, s_r;x_1,\dots,x_r)\equiv W(s'_1,\dots, s'_r;x_1,\dots,x_r) \pmod{p^M}
$$
holds for all $(x_1,\dots, x_r)\in\mathbb Z_p^r$
when $|s_i-s'_i|_p<\delta'$ ($1\leqslant i\leqslant r$).
So, in that case,  the inequality
\begin{align*}
|L_{p,r}(s_1,\ldots,s_r; \omega^{k_1},\ldots,\omega^{k_r};\gamma_1,\ldots,\gamma_{r};&c)-
L_{p,r}(s'_1,\ldots,s'_r;\omega^{k_1},\ldots,\omega^{k_r};\gamma_1,\ldots,\gamma_{r};c)|_p \\
&\leqslant p^{-M}
<\varepsilon
\end{align*}
holds for any $c\in\mathbb Z_p^\times$
because $\widetilde{\mathfrak{m}}_c$ is a $\mathbb Z_p$-valued measure.
Therefore, for any given $\varepsilon>0$,
there always exists $\delta'>0$ such that
when $|s_i-s'_i|_p<\delta'$ ($1\leqslant i\leqslant r$),
\begin{equation}\label{bound2}
|\varPsi(s_1,\ldots,s_r;c)-\varPsi(s'_1,\ldots,s'_r; c)|_p<\varepsilon
\end{equation}
holds for all $c\in\mathbb Z_p^\times$.

By \eqref{bound1} and \eqref{bound2}, we see that
for each fixed $(s_1,\dots,s_r)\in \mathbb Z_p^r$ and $c\in \mathbb Z_p^\times$,
and for any given $\varepsilon>0$,
there always exist $\delta, \delta'>0$ such that
\begin{align*}
\bigl|\varPsi&(s_1,\ldots,s_r;c)-\varPsi(s'_1,\ldots,s'_r; c')\bigr|_p \\
&=\bigl|\varPsi(s_1,\ldots,s_r;c)-\varPsi(s_1,\ldots,s_r;c')+\varPsi(s_1,\ldots,s_r;c')
-\varPsi(s'_1,\ldots,s'_r; c')\bigr|_p \\
&\leqslant
\max
\Bigl\{
\bigl|\varPsi(s_1,\ldots,s_r;c)-\varPsi(s_1,\ldots,s_r;c')\bigr|_p, 
\bigl|\varPsi(s_1,\ldots,s_r;c')-\varPsi(s'_1,\ldots,s'_r; c')\bigr|_p
\Bigr\}<\varepsilon
\end{align*}
for  $|s_i-s'_i|_p<\delta'$ ($1\leqslant i\leqslant r$) and $|c-c'|_p<\delta$.
Thus we get the desired continuity of $\varPsi(s_1,\ldots,s_r; c)$.

The uniform continuity of $\varPsi(s_1,\ldots,s_r; c)$
is almost obvious because 
now we know that the function
is continuous and the source set $\mathbb Z_p^r\times \mathbb Z_p^\times$ is compact.
\end{proof}

As a corollary of Theorem \ref{continuity theorem},
the following non-trivial property of special values of $p$-adic multiple $L$-functions
at non-positive integer points
is obtained.

\begin{corollary}\label{non-trivial property}
The right-hand side of equation
\eqref{Th-main} of Theorem \ref{T-main-1} (in the next section)
is $p$-adically continuous
with respect to $c$. 
\end{corollary}

Our final discussion here is towards a construction of a $p$-adic analogue of the entire function
$Z(s_1,\ldots,s_r;\gamma_1,\ldots,\gamma_{r})$
which was constructed 
as a multiple analogue of the entire function $(1-s)\zeta(s)$
by equation \eqref{Cont-2-1}:
As was explained in Remark \ref{conceptual idea},
an idea of the equation is conceptually (but not mathematically)
expressed as in
\eqref{nonsense equation}.
By Theorem \ref{T-main-1} and Remark \ref{rem-intpln} (both in the next section),
our
$L_{p,r}(s_1,\ldots,s_r;\omega^{n_1},\ldots,\omega^{n_r};\gamma_1,\ldots,\gamma_{r};c)$
is a $p$-adic interpolation of \eqref{Int-P},
which also has appeared 
on the right-hand side of \eqref{nonsense equation}.

By $\lim_{c\to 1}g_c(T)=0$, we have,
for $s_1,\ldots,s_r\in \mathbb{Z}_p$,
$k_1,\ldots,k_r\in \mathbb{Z}$,
$\gamma_1,\ldots,\gamma_r\in\mathbb Z_p$,
\begin{equation*}
\underset{c\in\mathbb Z_p^\times}{\lim_{c\to 1}}{L_{p,r}(s_1,\ldots,s_r;\omega^{k_1},\ldots,\omega^{k_r};\gamma_1,\ldots,\gamma_{r};c)}=0.
\end{equation*}
While, by
$$\lim_{c\to 1}\frac{g_c(T)}{c-1}\in\mathbb{Q}_p[[T]]\setminus\mathbb{Z}_p[[T]],$$
we can not say that the limit
$\lim_{c\to 1}
\frac{1}{c-1}{\widetilde{\mathfrak{m}}_c}$ 
converges to a measure.
Thus the following is unclear.

\begin{problem}\label{Main-Prob}
For $s_1,\ldots,s_r\in \mathbb{Z}_p$,
$k_1,\ldots,k_r\in \mathbb{Z}$ and
$\gamma_1,\ldots,\gamma_r\in\mathbb Z_p$, does
\begin{equation}\label{limit}
\underset{c\in\mathbb Z_p^\times\setminus\{1\}}
{\lim_{c\to 1}}\frac{1}{(c-1)^r}L_{p,r}(s_1,\ldots,s_r;\omega^{k_1},\ldots,\omega^{k_r};\gamma_1,\ldots,\gamma_{r};c)
\end{equation}
converge?
\end{problem} 

If the limit \eqref{limit} exists and happens to be a rigid analytic function,
we may call it a $p$-adic analogue of the desingularized zeta function
$\zeta^{\rm des}_r(s_1,\ldots,s_r;\gamma_1,\ldots,\gamma_{r})$.
%
We remind that the problem is affirmative  in the case when 
$r=1$ and $\gamma=1$. Actually by \eqref{1-p-LF-gamma}
we have
\begin{equation*}
\lim_{c\to 1}{(c-1)^{-1}L_{p,1}(s;\omega^k;1;c)}=
(1-s)\cdot L_p(s;\omega^{k+1}).
\end{equation*}
We also note  that the limit converges when $(s_1,\ldots,s_r)=(-n_1,\ldots,-n_r)$
for $n_1,\ldots,n_r\in\mathbb{N}_0$
by Theorem \ref{T-5-gene} in the next section.

\

\section{Special values of $p$-adic multiple $L$-functions at non-positive integers}\label{sec-4}

We will consider the special values of our $p$-adic multiple $L$-functions
(Definition \ref{Def-pMLF}) at non-positive integers.
We will express them in terms of twisted multiple Bernoulli numbers
(Definition \ref{Def-M-Bern}) in Theorems \ref{T-main-1} and \ref{T-5-gene}.
We will see that our  $p$-adic multiple $L$-function
is a $p$-adic interpolation of
a certain sum \eqref{Int-P} of the entire  complex multiple zeta-functions  of generalized Euler-Zagier-Lerch type
in Remark \ref{rem-intpln}.
Based on the evaluations, 
we will generalize the well-known Kummer congruence for ordinary Bernoulli numbers
to the multiple Kummer congruence for twisted multiple Bernoulli numbers in Theorem \ref{Th-Kummer}.
We will show certain functional relations with a parity condition
among $p$-adic multiple $L$-functions in Theorem \ref{T-6-1}. 
These functional relations will be  seen as multiple generalizations of 
the vanishing property of the Kubota-Leopoldt $p$-adic $L$-function with odd characters
(cf. Proposition \ref{Prop-zero})  in the single variable case
and of the  functional relations for the $p$-adic double $L$-function (shown in \cite{KMT-IJNT})
in the double variable case under the special condition $c=2$.
Many 
examples will be investigated in this section.

\subsection{Evaluation of $p$-adic multiple $L$-functions at non-positive integers}\label{negative-1}

Based on the consideration in the previous sections, we determine values of $p$-adic multiple $L$-functions at non-positive integers as follows. 

\begin{theorem}\label{T-main-1}
For $n_1,\ldots,n_r\in \mathbb{N}_0$, $\gamma_1,\ldots,\gamma_r\in \mathbb{Z}_p$, 
and $c\in \mathbb{N}_{>1}$ with $(c,p)=1$. 
\begin{align}
& L_{p,r}(-n_1,\ldots,-n_r;\omega^{n_1},\ldots,\omega^{n_r};\gamma_1,\ldots,\gamma_{r};c)\notag\\
& =\sum_{\xi_1^c=1 \atop \xi_1\not=1}\cdots\sum_{\xi_r^c=1 \atop \xi_r\not=1}\aa((n_j);(\xi_j);(\gamma_j)) \notag\\
& +\sum_{d=1}^{r}\left(-\frac{1}{p}\right)^d \sum_{1\leqslant i_1<\cdots<i_d\leqslant r}\sum_{\rho_{i_1}^p=1}\cdots\sum_{\rho_{i_d}^p=1}\sum_{\xi_1^c=1 \atop \xi_1\not=1}\cdots\sum_{\xi_r^c=1 \atop \xi_r\not=1}\aa((n_j);((\prod_{j\leqslant i_l}\rho_{i_l})^{\gamma_j}\xi_j);(\gamma_j)), \label{Th-main}
\end{align}
where the empty product is interpreted as $1$.
\end{theorem}

In Definition \ref{Def-M-Bern}, the twisted multiple Bernoulli number $\aa((n_j);(\xi_j);(\gamma_j))$ was defined for 
$\gamma_1,\ldots,\gamma_r\in \mathbb{C}$ and roots of unity $\xi_1,\ldots,\xi_r\in \mathbb{C}$. Here we define $\aa((n_j);(\xi_j);(\gamma_j))$ for $\gamma,\ldots,\gamma_r\in \mathbb{Z}_p$ and roots of unity $\xi_1,\ldots,\xi_r\in \mathbb{C}_p$ by \eqref{Def-Hr} and \eqref{Fro-def-r}.

\begin{remark}\label{rem-intpln}
When $\gamma_1,\ldots,\gamma_r\in \mathbb{Z}_p\cap \overline{\mathbb{Q}}$ satisfy 
$\Re \gamma_j >0 \quad (1\leqslant j\leqslant r)$, 
we obtain from the above theorem and  Theorem \ref{T-multiple} that
\begin{align*}
& L_{p,r}(-n_1,\ldots,-n_r;\omega^{n_1},\ldots,\omega^{n_r};\gamma_1,\ldots,\gamma_{r};c)
=(-1)^{r+n_1+\cdots+n_r}
\Bigl\{\sum_{\xi_1^c=1 \atop \xi_1\not=1}\cdots\sum_{\xi_r^c=1 \atop \xi_r\not=1}
\zeta_{r}((-n_j);(\xi_j);(\gamma_j)) \notag\\
&\qquad\qquad
 +\sum_{d=1}^{r}\left(-\frac{1}{p}\right)^d \sum_{1\leqslant i_1<\cdots<i_d\leqslant r}\sum_{\rho_{i_1}^p=1}\cdots\sum_{\rho_{i_d}^p=1}\sum_{\xi_1^c=1 \atop \xi_1\not=1}\cdots\sum_{\xi_r^c=1 \atop \xi_r\not=1}
\zeta_{r}((-n_j);((\prod_{j\leqslant i_l}\rho_{i_l})^{\gamma_j}\xi_j);(\gamma_j))\Bigr\}.
\end{align*}
We may say that the $p$-adic multiple $L$-function $L_{p,r}((s_j);(\omega^{k_1});(\gamma_j);c)$
is a $p$-adic interpolation of 
the  following finite sum of multiple zeta-functions $\zeta_r((s_j);(\xi_j);(\gamma_j))$
which are  entire:
\begin{equation}
\sum_{\xi_1^c=1 \atop \xi_1\not=1}\cdots \sum_{\xi_r^c=1 \atop \xi_r\not=1}
\zeta_r((s_j);(\xi_j);(\gamma_j)) \label{Int-P}.
\end{equation}
\end{remark}

{\it Proof of Theorem \ref{T-main-1}}.
We see that for any $\rho\in \mu_p$ and $\xi\in \mu_c\setminus\{1\}$, 
\begin{align*}
\mk_{\rho\xi}(j+p^N\mathbb{Z}_p)& =\frac{(\rho\xi)^j}{1-(\rho\xi)^{p^{N}}}=\rho^{j}\mk_{\xi}(j+p^N\mathbb{Z}_p) \qquad (N\geqslant 1). 
\end{align*}
Hence, using 
\begin{equation*}
\sum_{\rho^p=1}\rho^n=
\begin{cases}
0 & (p\nmid n)\\
p & (p\mid n),
\end{cases}
\end{equation*}
we have
\begin{align*}
& L_{p,r}(-n_1,\ldots,-n_r;\omega^{n_1},\ldots,\omega^{n_r};\gamma_1,\ldots,\gamma_{r};c)\\
& =\int_{\mathbb{Z}_p^r}(x_1\gamma_1)^{n_1} \cdots (\sum_{j=1}^{r}x_{j}\gamma_{j})^{n_r}\prod_{i=1}^{r}\left(1-\frac{1}{p}\sum_{\rho_i^p=1}\rho_i^{\sum_{\nu=1}^{i}x_\nu \gamma_\nu}\right) \prod_{j=1}^{r}\sum_{\xi_j^c=1 \atop \xi_j\not=1}d\mk_{\xi_j}(x_j)\\
& =\int_{\mathbb{Z}_p^r}(x_1\gamma_1)^{n_1} \cdots (\sum_{j=1}^{r}x_{j}\gamma_{j})^{n_r}\sum_{\xi_1^c=1 \atop \xi_1\not=1}\cdots \sum_{\xi_r^c=1 \atop \xi_r\not=1}\prod_{j=1}^{r}d\mk_{\xi_j}(x_j)\\
& \qquad +\sum_{d=1}^{r}\left(-\frac{1}{p}\right)^d \sum_{1\leqslant i_1<\cdots<i_d\leqslant r}\sum_{\rho_{i_1}^p=1}\cdots\sum_{\rho_{i_d}^p=1}\sum_{\xi_1^c=1 \atop \xi_1\not=1}\cdots\sum_{\xi_r^c=1 \atop \xi_r\not=1}\\
& \qquad \qquad \times 
\int_{\mathbb{Z}_p^r}\prod_{l=1}^{r} (\sum_{j=1}^{l}x_{j}\gamma_{j})^{n_l} \prod_{l=1}^{d}\rho_{i_l}^{\sum_{\nu=1}^{i_l}x_\nu\gamma_{\nu}}\prod_{j=1}^{r}d\mk_{\xi_j}(x_j)\\
& =\int_{\mathbb{Z}_p^r}(x_1\gamma_1)^{n_1} \cdots (\sum_{j=1}^{r}x_{j}\gamma_{j})^{n_r}\sum_{\xi_1^c=1 \atop \xi_1\not=1}\cdots \sum_{\xi_r^c=1 \atop \xi_r\not=1}\prod_{j=1}^{r}d\mk_{\xi_j}(x_j)\\
& \qquad +\sum_{d=1}^{r}\left(-\frac{1}{p}\right)^d \sum_{1\leqslant i_1<\cdots<i_d\leqslant r}\sum_{\rho_{i_1}^p=1}\cdots\sum_{\rho_{i_d}^p=1}\sum_{\xi_1^c=1 \atop \xi_1\not=1}\cdots\sum_{\xi_r^c=1 \atop \xi_r\not=1}\\
& \qquad \qquad \times 
\int_{\mathbb{Z}_p^r}\prod_{l=1}^{r} (\sum_{j=1}^{l}x_{j}\gamma_{j})^{n_l} \prod_{j=1}^{r}d\mk_{\{(\prod_{j\leqslant i_l}\rho_{i_l})^{\gamma_j}\xi_j\}}(x_j).
\end{align*}
Thus, by Proposition \ref{T-multi-2}, we complete the proof.
\qed

\begin{remark}
It is worthy to restate  Corollary \ref{non-trivial property} saying that
the right-hand side of
the above equation \eqref{Th-main} is $p$-adically continuous not only
with respect to $n_1,\ldots,n_r$ but also with respect to $c$.
\end{remark}


Considering the Galois action of ${\rm Gal}(\overline{\mathbb{Q}}_p/{\mathbb{Q}}_p)$, 
we obtain the following result from Theorems \ref{Th-pMLF} and \ref{continuity theorem}.

\begin{corollary}\label{C-main-0}
For $n_1,\ldots,n_r\in \mathbb{N}_0$, $\gamma_1,\ldots,\gamma_r\in \mathbb{Z}_p$
and $c\in \mathbb{Z}_{p}^\times$,
\begin{align*}
& L_{p,r}(-n_1,\ldots,-n_r;\omega^{n_1},\ldots,\omega^{n_r};\gamma_1,\ldots,\gamma_{r};c)\in \mathbb{Z}_p,
\end{align*}
hence the right-hand side of \eqref{Th-main} is in $\mathbb{Z}_p$, though it includes the terms like $\left(-\frac{1}{p}\right)^d$. 
In particular when $\gamma_1,\ldots,\gamma_{r}\in \mathbb{Z}_{(p)}:=\{\frac{a}{b}\in \mathbb{Q}\,|\,a,b\in \mathbb{Z},\ (b,p)=1\}=\mathbb{Z}_p\cap \mathbb{Q}$, 
\begin{align*}
& L_{p,r}(-n_1,\ldots,-n_r;\omega^{n_1},\ldots,\omega^{n_r};\gamma_1,\ldots,\gamma_{r};c)\in \mathbb{Z}_{(p)}.
\end{align*}
\end{corollary}

The following examples are special cases ($r=1$ and $r=2$) of Theorem \ref{T-main-1}.

\begin{example}
For $n\in \mathbb{N}_0$ 
and $c\in \mathbb{N}_{>1}$ with $(c,p)=1$,

\begin{align}
L_{p,1}(-n;\omega^n;1;c)&=(1-p^n)\sum_{\xi^c=1 \atop \xi\neq 1}\aa_n(\xi),\label{KL-LP}
\end{align}
which recovers the known fact
$$
L_p(-n;\omega^{n+1})
=-(1-p^n)\frac{B_{n+1}}{n+1}$$
thanks to \eqref{2-0-1} and \eqref{p-L-vals}. 
\end{example}

\begin{example}\label{C-main-1}
For $n_1,n_2\in \mathbb{N}_0$, $\gamma_1,\gamma_2\in \mathbb{Z}_p$, 
and $c\in \mathbb{N}_{>1}$ with $(c,p)=1$, 
\begin{align}
L_{p,2}&(-n_1,-n_2;\omega^{n_1},\omega^{n_2};\gamma_1,\gamma_{2};c)\notag\\
& =\sum_{\xi_1^c=1 \atop \xi_1\not=1}\sum_{\xi_2^c=1 \atop \xi_2\not=1}\aa(n_1,n_2;\xi_1,\xi_2;\gamma_1,\gamma_2) \notag\\
& \quad -\frac{1}{p} \sum_{\rho_{1}^p=1}\sum_{\xi_1^c=1 \atop \xi_1\not=1}\sum_{\xi_2^c=1 \atop \xi_2\not=1}\aa(n_1,n_2;\rho_{1}^{\gamma_1}\xi_1,\xi_2;\gamma_1,\gamma_2)\notag\\
& \quad -\frac{1}{p} \sum_{\rho_{2}^p=1}\sum_{\xi_1^c=1 \atop \xi_1\not=1}\sum_{\xi_2^c=1 \atop \xi_2\not=1}\aa(n_1,n_2;\rho_{2}^{\gamma_1}\xi_1, \rho_{2}^{\gamma_2}\xi_2;\gamma_1,\gamma_2)\notag\\
& \quad +\frac{1}{p^2} \sum_{\rho_{1}^p=1}\sum_{\rho_{2}^p=1}\sum_{\xi_1^c=1 \atop \xi_1\not=1}\sum_{\xi_2^c=1 \atop \xi_2\not=1}\aa(n_1,n_2;(\rho_{1}\rho_2)^{\gamma_1}\xi_1,\rho_{2}^{\gamma_2}\xi_2;\gamma_1,\gamma_2). \label{C-1-formula}
\end{align}
\end{example}

More generally, we consider the generating function of $\{ L_{p,r}((-n_j);(\omega^{n_j});(\gamma_j);c)\}$, that is, 
\begin{align*}
&F_{p,r}(t_1,\ldots,t_r;(\gamma_j);c)=\sum_{n_1=0}^\infty \cdots\sum_{n_r=0}^\infty L_{p,r}((-n_j);(\omega^{n_j});(\gamma_j);c)\prod_{j=1}^{r}\frac{t_j^{n_j}}{n_j!}.
\end{align*}
Then we have the following.

\begin{theorem}\label{T-5-gene}
Let $c\in\mathbb Z_p^\times$.
For $\gamma_1,\ldots,\gamma_r\in \mathbb{Z}_p$,
\begin{align*}
&F_{p,r}(t_1,\ldots,t_r;(\gamma_j);c)\notag\\
& =\prod_{j=1}^{r}\left(\frac{1}{\exp\left(\gamma_j\sum_{k=j}^r t_k\right)-1}-\frac{c}{\exp\left(c\gamma_j\sum_{k=j}^r t_k\right)-1}\right)\notag\\
& +\sum_{d=1}^{r}\left(-\frac{1}{p}\right)^d \sum_{1\leqslant i_1<\cdots<i_d\leqslant r}\sum_{\rho_{i_1}^p=1}\cdots\sum_{\rho_{i_d}^p=1}\notag\\
& \quad\times \prod_{j=1}^{r}\left(\frac{1}{(\prod_{j\leqslant i_l}\rho_{i_l})^{\gamma_j}\exp\left(\gamma_j\sum_{k=j}^r t_k\right)-1}-\frac{c}{(\prod_{j\leqslant i_l}\rho_{i_l})^{c\gamma_j}\exp\left(c\gamma_j\sum_{k=j}^r t_k\right)-1}\right)\\
& =\prod_{j=1}^{r}\left(\sum_{m_j=0}^\infty \left(1-c^{m_j+1}\right)\frac{B_{m_j+1}}{m_j+1}\frac{\left(\gamma_j\sum_{k=j}^r t_k\right)^{m_j}}{m_j!}\right)\notag\\
&\quad +(-1)^r\sum_{d=1}^{r}\left(-\frac{1}{p}\right)^d \sum_{1\leqslant i_1<\cdots<i_d\leqslant r}\sum_{\rho_{i_1}^p=1}\cdots\sum_{\rho_{i_d}^p=1}\notag\\
& \qquad\times \prod_{j=1}^{r}\left(\sum_{m_j=0}^\infty (1-c^{m_j+1})\frak{B}_{m_j}((\prod_{j\leqslant i_l}\rho_{i_l})^{\gamma_j})\frac{\left(\gamma_j\sum_{k=j}^r t_k\right)^{m_j}}{m_j!}\right). 
\end{align*}
In particular when $\gamma_1\in \mathbb{Z}_p^\times$ and $\gamma_j\in p\mathbb{Z}_p$ $(2\leqslant j\leqslant r)$, 
\begin{align}
F_{p,r}(t_1,\ldots,t_r;(\gamma_j);c)
& =\bigg(\frac{1}{\exp\left(\gamma_1\sum_{k=1}^r t_k\right)-1}-\frac{c}{\exp\left(c\gamma_1\sum_{k=1}^r t_k\right)-1} \notag\\
& \qquad - \frac{1}{\exp\left(p\gamma_1\sum_{k=1}^r t_k\right)-1}+\frac{c}{\exp\left(cp\gamma_1\sum_{k=1}^r t_k\right)-1}\bigg)\notag\\
& \quad \times \prod_{j=2}^{r}\left(\frac{1}{\exp\left(\gamma_j\sum_{k=j}^r t_k\right)-1}-\frac{c}{\exp\left(c\gamma_j\sum_{k=j}^r t_k\right)-1}\right)\notag\\
& =\left( \sum_{n_1=0}^\infty \left(1-p^{n_1}\right)\left(1-c^{n_1+1}\right)\frac{B_{n_1+1}}{n_1+1}\frac{\left(\gamma_1\sum_{k=1}^{r}t_k\right)^{n_1}}{n_1!}\right)\notag\\
& \quad \times \prod_{j=2}^{r}\left(\sum_{n_j=0}^\infty \left(1-c^{n_j+1}\right)\frac{B_{n_j+1}}{n_j+1}\frac{\left(\gamma_j\sum_{k=j}^{r}t_k\right)^{n_j}}{n_j!}\right).\label{generating}
\end{align}
\end{theorem}

We stress here that
Theorem \ref{T-main-1} holds for  $c\in \mathbb{N}_{>1}$ with $(c,p)=1$,
in contrast,
Theorem \ref{T-5-gene} holds, more generally, for $c\in\mathbb Z_p^\times$.

\begin{proof}
For the moment assume that $c\in \mathbb{N}_{>1}$ with $(c,p)=1$. 
Combining \eqref{Fro-def-r} and \eqref{Th-main}, we have
\begin{align*}
F_{p,r}((t_j);(\gamma_j);c)
& = \sum_{\xi_1^c=1 \atop \xi_1\not=1}\cdots \sum_{\xi_r^c=1 \atop \xi_r\not=1}\prod_{j=1}^{r} \frac{1}{1-\xi_j \exp\left(\gamma_j \sum_{k=j}^r t_k\right)}\\
& \quad +\sum_{d=1}^{r}\left(-\frac{1}{p}\right)^d \sum_{1\leqslant i_1<\cdots<i_d\leqslant r}\sum_{\rho_{i_1}^p=1}\cdots\sum_{\rho_{i_d}^p=1}\sum_{\xi_1^c=1 \atop \xi_1\not=1}\cdots \sum_{\xi_r^c=1 \atop \xi_r\not=1}\notag\\
& \quad\times \prod_{j=1}^{r}\frac{1}{1-(\prod_{j\leqslant i_l}\rho_{i_l})^{\gamma_j}\xi_j\exp\left(\gamma_j\sum_{k=j}^r t_k\right)}\\
& = \prod_{j=1}^{r} \sum_{\xi_j^c=1 \atop \xi_j\not=1}\frac{1}{1-\xi_j \exp\left(\gamma_j \sum_{k=j}^r t_k\right)}\\
& \quad +\sum_{d=1}^{r}\left(-\frac{1}{p}\right)^d \sum_{1\leqslant i_1<\cdots<i_d\leqslant r}\sum_{\rho_{i_1}^p=1}\cdots\sum_{\rho_{i_d}^p=1}\\
& \quad \times \prod_{j=1}^{r}\left(\sum_{\xi_j^c=1 \atop \xi_j\not=1}\frac{1}{1-(\prod_{j\leqslant i_l}\rho_{i_l})^{\gamma_j}\xi_j\exp\left(\gamma_j\sum_{k=j}^r t_k\right)}\right).
\end{align*}
Hence, by \eqref{log-der}, we obtain the first assertion. Next we assume $\gamma_1\in \mathbb{Z}_p^\times$ and $\gamma_j\in p\mathbb{Z}_p$ $(2\leqslant j\leqslant r)$. Then 
$$(\prod_{j\leqslant i_l}\rho_{i_l})^{\gamma_j}=1\quad (2\leqslant j\leqslant r).$$
Hence we have
\begin{align*}
F_{p,r}((t_j);(\gamma_j);c)
& =\prod_{j=1}^{r}\left(\frac{1}{\exp\left(\gamma_j\sum_{k=j}^r t_k\right)-1}-\frac{c}{\exp\left(c\gamma_j\sum_{k=j}^r t_k\right)-1}\right)\notag\\
& \ +\sum_{d=1}^{r}\left(-\frac{1}{p}\right)^d \sum_{1\leqslant i_1<\cdots<i_d\leqslant r}\sum_{\rho_{i_1}^p=1}\cdots\sum_{\rho_{i_d}^p=1}\notag\\
& \quad\times \left(\frac{1}{(\prod_{i_l}\rho_{i_l})\exp\left(\gamma_1\sum_{k=1}^r t_k\right)-1}-\frac{c}{(\prod_{i_l}\rho_{i_l})^{c}\exp\left(c\gamma_1\sum_{k=1}^r t_k\right)-1}\right)\\
& \qquad \times \prod_{j=2}^{r}\left(\frac{1}{\exp\left(\gamma_j\sum_{k=j}^r t_k\right)-1}-\frac{c}{\exp\left(c\gamma_j\sum_{k=j}^r t_k\right)-1}\right).
\end{align*}
Using \eqref{log-der}, we can rewrite the second term on the right-hand side as
\begin{align*}
& \prod_{j=2}^{r}\left(\frac{1}{\exp\left(\gamma_j\sum_{k=j}^r t_k\right)-1}-\frac{c}{\exp\left(c\gamma_j\sum_{k=j}^r t_k\right)-1}\right)\\
& \quad \times \sum_{d=1}^{r}\left(-\frac{1}{p}\right)^d \sum_{1\leqslant i_1<\cdots<i_d\leqslant r}p^d\left(\frac{1}{\exp\left(p\gamma_1\sum_{k=1}^r t_k\right)-1}-\frac{c}{\exp\left(cp\gamma_1\sum_{k=1}^r t_k\right)-1}\right).
\end{align*}
Noting 
$$\sum_{d=1}^{r}\left(-\frac{1}{p}\right)^d \sum_{1\leqslant i_1<\cdots<i_d\leqslant r}p^d=\sum_{d=1}^{r}\binom{r}{d}(-1)^d=(1-1)^r-1=-1,$$
we obtain the second assertion.

Whence we get the formulas in the theorem for $c\in \mathbb{N}_{>1}$ with $(c,p)=1$.
It should be noted that each coefficient of the left-hand sides of the formulas is
expressed as  a polynomial on $c$.
On the other hand each  coefficient of the right-hand sides of the formulas is
continuous on $c$ by Theorem \ref{continuity theorem}.
Therefore the validity of the formulas is extended into  $c\in\mathbb Z_p^\times$.
\end{proof}


We consider the case $r=2$ with $(\gamma_1,\gamma_2)=(1,\eta)$ for $\eta\in \mathbb{Z}_p$
and $c\in \mathbb{N}_{>1}$ with $(c,p)=1$. 
From \eqref{Eur-exp1} and Example \ref{C-main-1}, we have 
\begin{align*}
& L_{p,2}(-k_1,-k_2;\omega^{k_1},\omega^{k_2};1,\eta;c)\\
& =\sum_{\xi_1^c=1 \atop \xi_1\not=1}\sum_{\xi_2^c=1 \atop \xi_2\not=1}\sum_{\nu=0}^{k_2}\binom{k_2}{\nu}\aa_{k_1+\nu}(\xi_1)\aa_{k_2-\nu}(\xi_2)\eta^{k_2-\nu}\\
& \ -\frac{1}{p} \sum_{\rho_{1}^p=1}\sum_{\xi_1^c=1 \atop \xi_1\not=1}\sum_{\xi_2^c=1 \atop \xi_2\not=1}\sum_{\nu=0}^{k_2}\binom{k_2}{\nu}\aa_{k_1+\nu}(\rho_{1}\xi_1)\aa_{k_2-\nu}(\xi_2)\eta^{k_2-\nu}\\
& \ -\frac{1}{p} \sum_{\rho_{2}^p=1}\sum_{\xi_1^c=1 \atop \xi_1\not=1}\sum_{\xi_2^c=1 \atop \xi_2\not=1}\sum_{\nu=0}^{k_2}\binom{k_2}{\nu}\aa_{k_1+\nu}(\rho_{2}\xi_1)\aa_{k_2-\nu}(\rho_{2}^{\eta}\xi_2)\eta^{k_2-\nu}\\
& \ +\frac{1}{p^2} \sum_{\rho_{1}^p=1}\sum_{\rho_{2}^p=1}\sum_{\xi_1^c=1 \atop \xi_1\not=1}\sum_{\xi_2^c=1 \atop \xi_2\not=1}\sum_{\nu=0}^{k_2}\binom{k_2}{\nu}\aa_{k_1+\nu}(\rho_{1}\rho_2\xi_1)\aa_{k_2-\nu}(\rho_{2}^{\eta}\xi_2)\eta^{k_2-\nu}
\end{align*}
for $k_1,k_2\in \mathbb{N}_0$. 
Similarly to \eqref{2-0-1}, we have
\begin{equation}
\sum_{\xi^c=1 \atop \xi\not=1} \aa_n(\alpha \xi)=
\begin{cases}
\left(1-c^{n+1}\right)\frac{B_{n+1}}{n+1} & (\alpha=1)\\
c^{n+1}\aa_{n}(\alpha^c)-\aa_n(\alpha) & (\alpha\not=1)
\end{cases}
\label{p-xi-0}
\end{equation}
for $n\in \mathbb{N}_0$.
Also, using \eqref{log-der} with $k=p$, we have 
\begin{equation}
\sum_{\rho^p=1}\sum_{\xi^c=1 \atop \xi\not=1}\aa_n(\rho\xi)=\sum_{\rho^p=1}\left(c^{n+1}\aa_{n}(\rho^c)-\aa_{n}(\rho)\right)=p^{n+1}\left(1-c^{n+1}\right)\frac{B_{n+1}}{n+1}\label{p-xi}
\end{equation}
for $n\in \mathbb{N}_0$. 
Hence, by the continuity with respect to $c$, we have the following.

\begin{example} \label{C-DZV-1}
For $k_1,k_2\in \mathbb{N}_0$, $\eta\in \mathbb{Z}_p$ and $c\in \mathbb{Z}_p^\times$,
\begin{align}
& L_{p,2}(-k_1,-k_2;\omega^{k_1},\omega^{k_2};1,\eta;c)\notag\\
& = \sum_{\nu=0}^{k_2}\binom{k_2}{\nu}\left(1-p^{k_1+\nu}\right)\left(1-c^{k_1+\nu+1}\right)\frac{B_{k_1+\nu+1}}{k_1+\nu+1}\left(1-c^{k_2-\nu+1}\right)\frac{B_{k_2-\nu+1}}{k_2-\nu+1}\eta^{k_2-\nu}\notag\\
& \ -\frac{1}{p} \sum_{\rho_{2}^p=1}\sum_{\nu=0}^{k_2}\binom{k_2}{\nu}\left(c^{k_1+\nu+1}\aa_{k_1+\nu}(\rho_{2}^c)-\aa_{k_1+\nu}(\rho_{2})\right)\left(c^{k_2-\nu+1}\aa_{k_2-\nu}(\rho_{2}^{c\eta})-\aa_{k_2-\nu}(\rho_{2}^{\eta})\right)\eta^{k_2-\nu}\notag\\
& \ +\frac{1}{p}\sum_{\rho_{2}^p=1}\sum_{\nu=0}^{k_2}\binom{k_2}{\nu}p^{k_1+\nu}\left(1-c^{k_1+\nu+1}\right)\frac{B_{k_1+\nu+1}}{k_1+\nu+1}\left(c^{k_2-\nu+1}\aa_{k_2-\nu}(\rho_{2}^{c\eta})-\aa_{k_2-\nu}(\rho_{2}^{\eta})\right)\eta^{k_2-\nu}. \label{DV-01}
\end{align}
In particular when $\eta\in p\mathbb{Z}_p$, 
\begin{align}
& L_{p,2}(-k_1,-k_2;\omega^{k_1},\omega^{k_2};1,\eta;c)\notag\\
& =\sum_{\nu=0}^{k_2}\binom{k_2}{\nu}\left(1-p^{k_1+\nu}\right)\left(1- c^{k_1+\nu+1} \right)\left( 1-c^{k_2-\nu+1}\right)\frac{B_{k_1+\nu+1}B_{k_2-\nu+1}}{(k_1+\nu+1)(k_2-\nu+1)}\eta^{k_2-\nu}, \label{L2-val}
\end{align}
given from \eqref{Ber-01} and \eqref{p-xi}, where the case $c=2$ was already obtained in our previous paper \cite[Section 4]{KMT-IJNT}.
In the special case when $k_1\in \mathbb{N}$, $k_2\in \mathbb{N}_0$ and $k_1+k_2$ is odd, 
\begin{equation}
L_{p,2}(-k_1,-k_2;\omega^{k_1},\omega^{k_2};1,\eta;c)=\displaystyle{\frac{c-1}{2}}\left( 1-c^{k_1+k_2+1}\right)\left(1-p^{k_1+k_2}\right)\frac{B_{k_1+k_2+1}}{k_1+k_2+1}. \label{odd-val}
\end{equation}
\end{example}

Furthermore, computing the coefficient of $t_1^{n_1}t_2^{n_2}t_3^{n_3}$ in \eqref{generating} with $r=3$, we obtain the following.

\begin{example}
For $k_1,k_2,k_3\in \mathbb{N}_0$, $\eta_1,\eta_2\in p\mathbb{Z}_p$ 
and $c\in \mathbb{Z}_p^\times$, 
\begin{align}
& L_{p,3}(-k_1,-k_2,-k_3;\omega^{k_1},\omega^{k_2},\omega^{k_3};1,\eta_1,\eta_2;c)\notag\\
& =\sum_{\nu_1=0}^{k_3}\sum_{\nu_2=0}^{k_3-\nu_1}\sum_{\kappa=0}^{k_2}\binom{k_2}{\kappa}\binom{k_3}{\nu_1\ \nu_2}\left(1-p^{k_1+\nu_2+\kappa}\right)\left(1-c^{k_1+\nu_2+\kappa+1}\right)\left(1-c^{k_2+\nu_1-\kappa+1}\right)\notag\\
& \qquad \times \left(1-c^{k_3-\nu_1-\nu_2+1}\right)\frac{B_{k_1+\nu_2+\kappa+1}}{k_1+\nu_2+\kappa+1}\frac{B_{k_2+\nu_1-\kappa+1}}{k_2+\nu_1-\kappa+1}\frac{B_{k_3-\nu_1-\nu_2+1}}{k_3-\nu_1-\nu_2+1}\notag\\
& \qquad \times 
\eta_1^{k_2+\nu_1-\kappa}\eta_2^{k_3-\nu_1-\nu_2},\label{tri-comv}
\end{align}
where $\binom{N}{\nu_1\ \nu_2}=\frac{N!}{\nu_1!\ \nu_2!\ (N-\nu_1-\nu_2)!}$.
\end{example}

\subsection{Multiple Kummer congruences
}\label{Kummer}

We will formulate  multiple Kummer congruences for twisted multiple Bernoulli numbers,
certain generalization of the Kummer congruences for Bernoulli numbers (see \eqref{ord-Kummer} below), from the viewpoint of $p$-adic multiple $L$-functions 
(Theorem \ref{Th-Kummer}).
Then we will extract specific congruences for only ordinary Bernoulli numbers
in depth 2 case (Example \ref{DKC}). 

First we recall the ordinary Kummer congruences. By \eqref{1-p-LF-gamma}, we know from \cite[p.\,31]{Kob} that 
\begin{equation}
L_{p,1}(1-m;\omega^{m-1};1;c) \equiv L_{p,1}(1-n;\omega^{n-1};1;c)\quad ({\rm mod\ }p^l) \label{L-K-r-1}
\end{equation}
for $m,n \in \mathbb{N}_0$ and $l\in \mathbb{N}$ with $m \equiv n\ ({\rm mod\ }(p-1)p^{l-1})$ and $c\in \mathbb{N}_{>1}$. Therefore, from \eqref{KL-LP}, we see that 
\begin{equation}
(1-c^m)(1-p^{m-1})\frac{B_m}{m}\equiv (1-c^n)(1-p^{n-1})\frac{B_n}{n}\quad ({\rm mod\ }p^l). \label{Ord-KC}
\end{equation}
Note that $c\in \mathbb{N}_{>1}$ can be chosen arbitrarily under the condition $(c,p)=1$. 
In particular when $p>2$,  
for even positive integers $m$ and $n$ satisfying $m \equiv n\ ({\rm mod\ }(p-1)p^{l-1})$ and $n \not\equiv 0\ ({\rm mod\ }p-1)$, 
we can choose $c$ satisfying 
$1-c^m \equiv 1-c^n\ ({\rm mod\ }p^l)$ and $(1-c^m,p)=(1-c^n,p)=1$, and consequently obtain the ordinary Kummer congruences (see \cite[p.\,32]{Kob}):
\begin{equation}
(1-p^{m-1})\frac{B_m}{m}\equiv (1-p^{n-1})\frac{B_n}{n}\quad ({\rm mod\ }p^l). \label{ord-Kummer}
\end{equation}

As a multiple analogue of \eqref{Ord-KC}, we have the following. 

\begin{theorem}[{\bf Multiple Kummer congruence}]
\label{Th-Kummer}
Let $m_1,\ldots,m_r,n_1,\ldots,n_r\in \mathbb{N}_0$ with 
$m_j\equiv n_j$\ $\text{\rm mod}\ (p-1)p^{l_j-1}$ 
for $l_j\in \mathbb{N}$ $(1\leqslant j\leqslant r)$. 
Then, for $\gamma_1,\ldots,\gamma_r\in \mathbb{Z}_p$ 
and $c\in \mathbb{N}_{>1}$ with $(c,p)=1$, 
\begin{align}
& L_{p,r}(-m_1,\ldots,-m_r;\omega^{m_1},\ldots,\omega^{m_r};\gamma_1,\ldots,\gamma_{r};c)\notag\\
& \equiv L_{p,r}(-n_1,\ldots,-n_r;\omega^{n_1},\ldots,\omega^{n_r};\gamma_1,\ldots,\gamma_{r};c)\quad  \left({\rm mod}\  p^{\min \{l_j\,|\,{1\leqslant j\leqslant r}\}}\right). \label{L-val}
\end{align}
In other words, 
\begin{align}
& \sum_{\xi_1^c=1 \atop \xi_1\not=1}\cdots\sum_{\xi_r^c=1 \atop \xi_r\not=1}\aa((m_j);(\xi_j);(\gamma_j)) \notag\\
& +\sum_{d=1}^{r}\left(-\frac{1}{p}\right)^d \sum_{1\leqslant i_1<\cdots<i_d\leqslant r}\sum_{\rho_{i_1}^p=1}\cdots\sum_{\rho_{i_d}^p=1}\sum_{\xi_1^c=1 \atop \xi_1\not=1}\cdots\sum_{\xi_r^c=1 \atop \xi_r\not=1}\aa((m_j);((\prod_{j\leqslant i_l}\rho_{i_l})^{\gamma_j}\xi_j);(\gamma_j))\notag\\
& \equiv \sum_{\xi_1^c=1 \atop \xi_1\not=1}\cdots\sum_{\xi_r^c=1 \atop \xi_r\not=1}\aa((n_j);(\xi_j);(\gamma_j)) \notag\\
& \qquad +\sum_{d=1}^{r}\left(-\frac{1}{p}\right)^d \sum_{1\leqslant i_1<\cdots<i_d\leqslant r}\sum_{\rho_{i_1}^p=1}\cdots\sum_{\rho_{i_d}^p=1}\sum_{\xi_1^c=1 \atop \xi_1\not=1}\cdots\sum_{\xi_r^c=1 \atop \xi_r\not=1}\aa((n_j);((\prod_{j\leqslant i_l}\rho_{i_l})^{\gamma_j}\xi_j);(\gamma_j))\notag\\
& \qquad \qquad \left({\rm mod}\  p^{\min \{l_j\,|\,{1\leqslant j\leqslant r}\}}\right). \label{Kummer-main}
\end{align}
\end{theorem}

\begin{proof}
The proof works in the same way as 
the case of $L_p(s;\chi)$ stated above (for the details, see \cite[Chapter 2,\,Section\,3]{Kob}).
As for the integrand of \eqref{integration},
we have
\begin{equation}
\left(\sum_{\nu=1}^{j}x_{\nu}\gamma_{\nu}\right)^{m_j}\equiv \left(\sum_{\nu=1}^{j}x_{\nu}\gamma_{\nu}\right)^{n_j}\quad \left(\textrm{mod\ } p^{l_j}\right)\quad (1\leqslant j\leqslant r) \label{MKC1}
\end{equation}
for $(x_j)\in \left(\mathbb{Z}_p^r\right)'_{\{\gamma_j\}}$. 
Hence \eqref{MKC1} directly yields \eqref{L-val}. 
It follows from Theorem \ref{T-main-1} that \eqref{Kummer-main} holds. 
\end{proof}

\begin{remark}
In the case $p=2$, since 
$$\left(\mathbb{Z}/2^l\mathbb{Z}\right)^\times\simeq \mathbb{Z}/2\mathbb{Z} \times \mathbb{Z}/2^{l-2}\mathbb{Z}\quad (l\in \mathbb{N}_{>2}),$$
\eqref{MKC1} holds under the condition
\begin{equation}
m_j\equiv n_j\ \text{\rm mod}\ 2^{l_j-2}\ \ (l_j\in \mathbb{N}_{>2};\ 1\leqslant j\leqslant r), \label{2-condition}
\end{equation}
so does \eqref{L-val} under \eqref{2-condition}. 
Hence the following examples also hold under \eqref{2-condition} in the case $p=2$.
\end{remark}




In the case $r=1$, Theorem \ref{Th-Kummer} is nothing but \eqref{L-K-r-1},
a reformulation of \eqref{Ord-KC}.

In the case $r=2$, we obtain the following. 

\begin{example}\label{Cor-Kummer-2}
For $m_1,m_2,n_1,n_2\in \mathbb{N}_0$ with $m_j\equiv n_j$\ $\text{\rm mod}\ (p-1)p^{l_j-1}$ $(l_1,l_2\in \mathbb{N})$, $\eta\in \mathbb{Z}_p$ and $c\in \mathbb{Z}_p^\times$, 
\begin{align}
& \sum_{\nu=0}^{m_2}\binom{m_2}{\nu}\left(1-p^{m_1+\nu}\right)\left(1-c^{m_1+\nu+1}\right)\frac{B_{m_1+\nu+1}}{m_1+\nu+1}\left(1-c^{m_2-\nu+1}\right)\frac{B_{m_2-\nu+1}}{m_2-\nu+1}\eta^{m_2-\nu}\notag\\
& \ -\frac{1}{p} \sum_{\rho_{2}^p=1}\sum_{\nu=0}^{m_2}\binom{m_2}{\nu}\left(c^{m_1+\nu+1}\aa_{m_1+\nu}(\rho_{2}^c)-\aa_{m_1+\nu}(\rho_{2})\right)\left(c^{m_2-\nu+1}\aa_{m_2-\nu}(\rho_{2}^{c\eta})-\aa_{m_2-\nu}(\rho_{2}^{\eta})\right)\eta^{m_2-\nu}\notag\\
& \ +\frac{1}{p}\sum_{\rho_{2}^p=1}\sum_{\nu=0}^{m_2}\binom{m_2}{\nu}p^{m_1+\nu}\left(1-c^{m_1+\nu+1}\right)\frac{B_{m_1+\nu+1}}{m_1+\nu+1}\left(c^{m_2-\nu+1}\aa_{m_2-\nu}(\rho_{2}^{c\eta})-\aa_{m_2-\nu}(\rho_{2}^{\eta})\right)\eta^{m_2-\nu}\notag\\
&\ \equiv \sum_{\nu=0}^{n_2}\binom{n_2}{\nu}\left(1-p^{n_1+\nu}\right)\left(1-c^{n_1+\nu+1}\right)\frac{B_{n_1+\nu+1}}{n_1+\nu+1}\left(1-c^{n_2-\nu+1}\right)\frac{B_{n_2-\nu+1}}{n_2-\nu+1}\eta^{n_2-\nu}\notag\\
& \ -\frac{1}{p} \sum_{\rho_{2}^p=1}\sum_{\nu=0}^{n_2}\binom{n_2}{\nu}\left(c^{n_1+\nu+1}\aa_{n_1+\nu}(\rho_{2}^c)-\aa_{n_1+\nu}(\rho_{2})\right)\left(c^{n_2-\nu+1}\aa_{n_2-\nu}(\rho_{2}^{c\eta})-\aa_{n_2-\nu}(\rho_{2}^{\eta})\right)\eta^{n_2-\nu}\notag\\
& \ +\frac{1}{p}\sum_{\rho_{2}^p=1}\sum_{\nu=0}^{n_2}\binom{n_2}{\nu}p^{n_1+\nu}\left(1-c^{n_1+\nu+1}\right)\frac{B_{n_1+\nu+1}}{n_1+\nu+1}\left(c^{n_2-\nu+1}\aa_{n_2-\nu}(\rho_{2}^{c\eta})-\aa_{n_2-\nu}(\rho_{2}^{\eta})\right)\eta^{n_2-\nu}\notag\\
& \qquad \qquad \left({\rm mod}\  p^{\min \{l_1,l_2 \}}\right).\label{gen-Kummer}
\end{align}
\end{example}

\begin{proof}
We obtain the claim
by setting $(\gamma_1,\gamma_2)=(1,\eta)$ for $\eta\in \mathbb{Z}_p$ in \eqref{L-val} and using Example \ref{C-DZV-1}. 
\end{proof}

Next we consider certain congruences between ordinary Bernoulli numbers. 

\begin{example}\label{DKC}
We set $\eta=p$ and $c\in \mathbb{N}_{>1}$ with $(c,p)=1$
in \eqref{gen-Kummer}. 
Note that this case was already calculated in \eqref{L2-val} with $\eta=p$. 
Hence we similarly obtain the following  {\bf double Kummer congruence}
for ordinary Bernoulli numbers:
\begin{align}
& \sum_{\nu=0}^{m_2}\binom{m_2}{\nu}\left(1-p^{m_1+\nu}\right)\left(1- c^{m_1+\nu+1} \right)\left( 1-c^{m_2-\nu+1}\right)\frac{B_{m_1+\nu+1}B_{m_2-\nu+1}}{(m_1+\nu+1)(m_2-\nu+1)}p^{m_2-\nu}\notag \\
& \equiv \sum_{\nu=0}^{n_2}\binom{n_2}{\nu}\left(1-p^{n_1+\nu}\right)\left(1- c^{n_1+\nu+1} \right)\left( 1-c^{n_2-\nu+1}\right)\frac{B_{n_1+\nu+1}B_{n_2-\nu+1}}{(n_1+\nu+1)(n_2-\nu+1)}p^{n_2-\nu}\notag \\
& \qquad \left({\rm mod}\  p^{\min \{l_1,l_2 \}}\right) \label{Kum-mod-d1}
\end{align}
for $m_1,m_2,n_1,n_2\in \mathbb{N}_0$ with $m_j\equiv n_j$\ $\text{\rm mod}\ (p-1)p^{l_j-1}$ $(l_j\in \mathbb{N};\ j=1,2)$.
In particular when $m_1$ and $m_2$ are of different parity, 
we have from \eqref{odd-val} that 
\begin{align}
&\displaystyle{\frac{c-1}{2}}\left( 1-c^{m_1+m_2+1}\right)\left(1-p^{m_1+m_2}\right)\frac{B_{m_1+m_2+1}}{m_1+m_2+1}\notag \\
& \equiv \displaystyle{\frac{c-1}{2}}\left( 1-c^{n_1+n_2+1}\right)\left(1-p^{n_1+n_2}\right)\frac{B_{n_1+n_2+1}}{n_1+n_2+1}\ \ \left({\rm mod}\  p^{\min \{l_1,l_2 \}}\right),  \label{Kum-mod-d2}
\end{align}
which is equivalent to \eqref{Ord-KC} in the case $(c,p)=1$ with $p^2\nmid (c-1)$. 
Therefore, choosing $c$ suitably, we obtain the ordinary Kummer congruence \eqref{ord-Kummer}. 
From this observation, we can regard \eqref{Kum-mod-d1} as a certain general form of the Kummer congruence including \eqref{ord-Kummer}. 
We do not know whether 
our \eqref{Kum-mod-d1} is a kind of new congruence of Bernoulli numbers or
it follows from the ordinary Kummer congruence \eqref{Ord-KC}.
\end{example}

In the case $r=3$, we obtain the following. 

\begin{example}\label{triple-KC}
Using \eqref{tri-comv} with $(\gamma_1,\gamma_2,\gamma_3)=(1,p,p)$, we can similarly obtain the following {\bf triple Kummer congruence}
for ordinary Bernoulli numbers:
\begin{align}
& \sum_{\nu_1=0}^{m_3}\sum_{\nu_2=0}^{m_3-\nu_1}\sum_{\kappa=0}^{m_2}\binom{m_2}{\kappa}\binom{m_3}{\nu_1\ \nu_2}\left(1-p^{m_1+\nu_2+\kappa}\right)\left(1-c^{m_1+\nu_2+\kappa+1}\right)\left(1-c^{m_2+\nu_1-\kappa+1}\right)\notag\\
& \quad \times \left(1-c^{m_3-\nu_1-\nu_2+1}\right)\frac{B_{m_1+\nu_2+\kappa+1}}{m_1+\nu_2+\kappa+1}\frac{B_{m_2+\nu_1-\kappa+1}}{m_2+\nu_1-\kappa+1}\frac{B_{m_3-\nu_1-\nu_2+1}}{m_3-\nu_1-\nu_2+1}p^{m_2+m_3-\kappa-\nu_2}\notag\\
& \equiv \sum_{\nu_1=0}^{n_3}\sum_{\nu_2=0}^{n_3-\nu_1}\sum_{\kappa=0}^{n_2}\binom{n_2}{\kappa}\binom{n_3}{\nu_1\ \nu_2}\left(1-p^{n_1+\nu_2+\kappa}\right)\left(1-c^{n_1+\nu_2+\kappa+1}\right)\left(1-c^{n_2+\nu_1-\kappa+1}\right)\notag\\
& \qquad \times \left(1-c^{n_3-\nu_1-\nu_2+1}\right)\frac{B_{n_1+\nu_2+\kappa+1}}{n_1+\nu_2+\kappa+1}\frac{B_{n_2+\nu_1-\kappa+1}}{n_2+\nu_1-\kappa+1}\frac{B_{n_3-\nu_1-\nu_2+1}}{n_3-\nu_1-\nu_2+1}p^{n_2+n_3-\kappa-\nu_2}\notag\\
& \qquad \left({\rm mod}\  p^{\min \{l_1,l_2,l_3 \}}\right) \label{Kum-mod-triple}
\end{align}
for $m_1,m_2,m_3,n_1,n_2,n_3\in \mathbb{N}_0$ with $m_j\equiv n_j$\ $\text{\rm mod}\ (p-1)p^{l_j-1}$ $(l_j\in \mathbb{N};\ j=1,2,3)$. It is also unclear whether this is new or follows from \eqref{Ord-KC}.
\end{example}

\subsection{Functional relations for $p$-adic multiple $L$-functions}\label{Funct-rel}
In this subsection, we will prove certain functional relations with a parity condition
among $p$-adic multiple $L$-functions (Theorem \ref{T-6-1}).
They are extensions of the vanishing property of the
Kubota-Leopoldt $p$-adic $L$-functions with odd characters
in the single variable case (cf. Proposition \ref{Prop-zero} and Example \ref{Rem-zero-2})
and  the functional relations among
$p$-adic double $L$-functions proved by  \cite{KMT-IJNT}
in the double variable case under the condition $c=2$
(cf. Example \ref{P-Func-eq-1}). 
Also they can be regarded as certain $p$-adic analogues of the parity result for MZVs (see Remark \ref{Rem-Parity-result}). 
It should be noted that the following functional relations 
are peculiar to the $p$-adic case,
which are derived from the relation \eqref{parity-2} among Bernoulli numbers.

Let $r \geqslant 2$. 
For each $J\subset\{1,\ldots,r\}$ with $1\in J$, we write 
$$J=\{ j_1(J),j_2(J),\ldots,j_{|J|}(J)\}\quad (1=j_1(J)<\cdots<j_{|J|}(J)),$$
where $|J|$ implies the number of elements of $J$. 
In addition, we formally let $j_{|J|+1}(J)=r+1$.

\begin{theorem}\label{T-6-1}
For $r\in \mathbb{N}_{>1}$, let $k_1,\ldots,k_r\in \mathbb{Z}$ with $k_1+\cdots+k_r\not\equiv r\pmod 2$, $\gamma_{1}\in \mathbb{Z}_p$, $\gamma_2,\ldots,\gamma_{r}\in p\mathbb{Z}_p$ and $c\in \mathbb{Z}_p^\times$. Then, for $s_1,\ldots,s_r\in \mathbb{Z}_p$, 
\begin{align}
& L_{p,r} (s_1,\ldots,s_r;\omega^{k_1},\ldots,\omega^{k_r};\gamma_1,\gamma_2,\ldots,\gamma_{r};c)\notag\\
& \ =-\sum_{J\subsetneq\{1,\ldots,r\} \atop 1\in J}
    \Bigr(\frac{1-c}{2}\Bigl)^{r-|J|}\notag\\
&   \qquad \times L_{p,|J|}\Bigl(\Bigl(\sum_{j_\mu(J)\leqslant l<j_{\mu+1}(J)}s_l\Bigr)_{\mu=1}^{|J|};\bigl(\omega^{\sum_{j_\mu(J)\leqslant l<j_{\mu+1}(J)}k_l}\bigr)_{\mu=1}^{|J|};(\gamma_{j})_{j\in J};c\Bigr). \label{main-1}
\end{align}
\end{theorem}

For the proof of this theorem, we prepare some notation.
Similarly to Definition \ref{define-tilde-H}, we consider $\widetilde{\mathfrak{H}}_1(t;\gamma;c)$ in $\mathbb{C}_p$, that is,  
\begin{equation}
\begin{split}
\widetilde{\mathfrak{H}}_1(t;\gamma;c)&=\frac{1}{e^{\gamma t}-1}-\frac{c}{e^{c\gamma t}-1}\\
&=\sum_{m=0}^\infty \left(1-c^{m+1}\right)B_{m+1}\frac{(\gamma t)^m}{(m+1)!}\\
&=\frac{c-1}{2}+\sum_{m=1}^\infty \left(c^{m+1}-1\right)\zeta(-m)\frac{(\gamma t)^m}{m!}
\end{split}
\label{3-3-2}
\end{equation}
for $c\in \mathbb{Z}_{p}^\times$ and $\gamma\in \mathbb{Z}_p$. 
For simplicity, we set 
\begin{equation}
\bb (n;\gamma;c):=\left(1-c^{n+1}\right)\frac{B_{n+1}}{n+1}\gamma^{n}\quad (n \in \mathbb{N}). \label{Ber-def}
\end{equation}
It follows from \eqref{2-1} and \eqref{2-1-2} that
\begin{equation}
\widetilde{\mathfrak{H}}_1(t;\gamma;c)=\sum_{m=0}^\infty \int_{\mathbb{Z}_p} x^{m}d\mm(x) \frac{(\gamma t)^m}{m!}
=\int_{\mathbb{Z}_p} e^{x\gamma t} d\mm(x),\label{3-4}
\end{equation}
\begin{equation}
\begin{split}
  \gh_1(t;\gamma;c):& =\widetilde{\mathfrak{H}}_1(t;\gamma;c)-\frac{c-1}{2}=\sum_{m=1}^\infty \bb (m;\gamma;c)\frac{t^m}{m!}\\
  &=\int_{\mathbb{Z}_p}(e^{x\gamma t}-1)d\widetilde{\mathfrak{m}}_c(x). 
\end{split}
\label{e-3-5}
\end{equation}
We know that $B_{2m+1}=0$, namely $\bb ({2m};\gamma;c)=0$ for $m\in \mathbb{N}$. Hence we have
\begin{equation}
\gh_1(-t;\gamma;c)=-\gh_1(t;\gamma;c). \label{e-3-6}
\end{equation}
Now we define certain multiple analogues of Bernoulli numbers as follows. 
For $r\in \mathbb{N}$ 
and $\gamma_1,\ldots,\gamma_r\in \mathbb{Z}_p$, we set
\begin{equation}
\label{Ber-def-r}
  \begin{split}
    \gh_r(t_1,\ldots,t_r;\gamma_1,\ldots,\gamma_r;c)
    &:=
    \prod_{j=1}^{r} \gh_1( \sum_{k=j}^r t_k;\gamma_j;c)
    \\
    &\left(=
    \prod_{j=1}^{r} \left\{\widetilde{\mathfrak{H}}_1(\sum_{k=j}^r t_k;\gamma_j;c)-\frac{c-1}{2}\right\}\right)
    \\
    &=
    \sum_{n_1=1}^\infty
    \cdots
    \sum_{n_r=1}^\infty
    \bb (n_1,\ldots,n_r;\gamma_1,\ldots,\gamma_r;c)
    \frac{t_1^{n_1}}{n_1!}
    \cdots
    \frac{t_r^{n_r}}{n_r!}.
  \end{split}
\end{equation}
By \eqref{e-3-6}, we have
\begin{equation}
  \gh_r(-t_1,\ldots,-t_r;\gamma_1,\ldots,\gamma_r;c)
  =(-1)^r
  \gh_r(t_1,\ldots,t_r;\gamma_1,\ldots,\gamma_r;c).
\label{parity-1}
\end{equation}
Hence, when $n_1+\cdots+n_r\not\equiv r\pmod 2$, we obtain
\begin{equation}
  \bb (n_1,\ldots,n_r;\gamma_1,\ldots,\gamma_r;c)=0.
\label{parity-2}
\end{equation}

\begin{remark}
From \eqref{def-tilde-H}, we have
\begin{align}
    & \widetilde{\mathfrak{H}}_r((t_j);(\gamma_j);c) 
    =
    \prod_{j=1}^{r} \widetilde{\mathfrak{H}}_1(\sum_{k=j}^r t_k;\gamma_j;c). \label{Ber-def-r-2}
\end{align}
Comparing \eqref{Ber-def-r} with \eqref{Ber-def-r-2}, we see that $\gh_r((t_j);(\gamma_j);c)$ is defined as a slight modification of $\widetilde{\mathfrak{H}}_r((t_j);(\gamma_j);c)$ 
so that \eqref{parity-1} holds.
\end{remark}

It follows from \eqref{Ber-def}, \eqref{e-3-5} and \eqref{Ber-def-r} that $\bb (n_1,\ldots,n_r;\gamma_1,\ldots,\gamma_r;c)$ can be expressed as a polynomial in $\{ B_{n}\}_{n\geqslant 1}$ and $\{\gamma_1,\ldots,\gamma_r\}$ with $\mathbb{Q}$-coefficients. 

\begin{lemma} \label{L-6-2}
For $n_1,\ldots,n_r\in \mathbb{N}$, $\gamma_1,\ldots,\gamma_r\in \mathbb{Z}_p$ and $c\in \mathbb{Z}_p^\times$, 
\begin{align}
\label{e-6-2}
  \bb (n_1,\ldots,n_r;\gamma_1,\ldots,\gamma_r;c)
    &=
    \sum_{J\subset\{1,\ldots,r\}\atop 1\in J}
    \Bigr(\frac{1-c}{2}\Bigl)^{r-|J|}
    \int_{\mathbb{Z}_p^{|J|}}
    \prod_{l=1}^r 
    (\sum_{\substack{j\in J\\j\leqslant l}}x_j \gamma_j)^{n_l}
    \prod_{j\in J} d\widetilde{\mathfrak{m}}_c(x_j),
\end{align}
where the empty product is interpreted as $1$.
\end{lemma}

\begin{proof}
From \eqref{e-3-5} and \eqref{Ber-def-r}, we have
\begin{align*}
  \gh_r&(t_1,\ldots,t_r;\gamma_1,\ldots,\gamma_r;c)\\
    &=
    \int_{\mathbb{Z}_p^r}\prod_{j=1}^r(e^{x_j \gamma_j (\sum_{l=j}^r t_l)}-1)\prod_{j=1}^r d\widetilde{\mathfrak{m}}_c(x_j)
    \\
    &=
    \int_{\mathbb{Z}_p^r}
    \sum_{J\subset\{1,\ldots,r\}}(-1)^{r-|J|}
    \exp\Bigl(\sum_{j\in J}x_j \gamma_j (\sum_{l=j}^r t_l)\Bigr)
    \prod_{j=1}^r d\widetilde{\mathfrak{m}}_c(x_j)
    \\
    &=
    \int_{\mathbb{Z}_p^r}
    \sum_{J\subset\{1,\ldots,r\}}(-1)^{r-|J|}
    \exp\Bigr(\sum_{l=1}^r t_l(\sum_{j\in J \atop j\leqslant l}x_j \gamma_j)\Bigl)
    \prod_{j=1}^r d\widetilde{\mathfrak{m}}_c(x_j)
    \\
    &=
    \sum_{n_1=0}^\infty
    \cdots
    \sum_{n_r=0}^\infty
    \sum_{J\subset\{1,\ldots,r\}}
    (-1)^{r-|J|}
    \int_{\mathbb{Z}_p^r}
    \prod_{l=1}^r 
    \frac{\left(t_l(\sum_{j\in J \atop j\leqslant l}x_j \gamma_j)\right)^{n_l}}{n_l!}
    \prod_{j=1}^r d\widetilde{\mathfrak{m}}_c(x_j)
    \\
    &=
    \sum_{n_1=0}^\infty
    \cdots
    \sum_{n_r=0}^\infty
    \sum_{J\subset\{1,\ldots,r\}}
    (-1)^{r-|J|}
    \int_{\mathbb{Z}_p^r}
    \prod_{l=1}^r 
    (\sum_{\substack{j\in J \\ j\leqslant l}}x_j \gamma_j)^{n_l}
    \prod_{j=1}^r d\widetilde{\mathfrak{m}}_c(x_j)
    \frac{t_1^{n_1}}{n_1!}
    \cdots
    \frac{t_r^{n_r}}{n_r!}.
\end{align*}
Here we consider each coefficient of $\frac{t_1^{n_1}}{n_1!} \cdots \frac{t_r^{n_r}}{n_r!}$ for $n_1,\ldots,n_r\in \mathbb{N}$. If $1\not\in J$ then 
$$\sum_{j\in J \atop j\leqslant 1}x_j \gamma_j$$
is an empty sum which implies $0$. Hence we obtain from \eqref{2-1-2} that 
\begin{align*}
  \bb &(n_1,\ldots,n_r;\gamma_1,\ldots,\gamma_r;c)\\
    &=
    \sum_{J\subset\{1,\ldots,r\}}
    (-1)^{r-|J|}
    \int_{\mathbb{Z}_p^r}
    \prod_{l=1}^r 
    (\sum_{\substack{j\in J\\j\leqslant l}}x_j \gamma_j)^{n_l}
    \prod_{j=1}^r d\widetilde{\mathfrak{m}}_c(x_j)
    \\
    &=
    \sum_{J\subset\{1,\ldots,r\} \atop 1\in J}
    (-1)^{r-|J|}
    \int_{\mathbb{Z}_p^r}
    \prod_{l=1}^r 
    (\sum_{\substack{j\in J\\j\leqslant l}}x_j \gamma_j)^{n_l}
    \prod_{j=1}^r d\widetilde{\mathfrak{m}}_c(x_j)
    \\
    &=
    \sum_{J\subset\{1,\ldots,r\} \atop 1\in J}
    \Bigr(\frac{1-c}{2}\Bigl)^{r-|J|}
    \int_{\mathbb{Z}_p^{|J|}}
    \prod_{l=1}^r 
    (\sum_{\substack{j\in J\\j\leqslant l}}x_j \gamma_j)^{n_l}
    \prod_{j\in J} d\widetilde{\mathfrak{m}}_c(x_j)
\end{align*}
for $n_1,\ldots,n_r\in \mathbb{N}$. 
Thus we complete the proof.
\end{proof}

\begin{proof}[Proof of Theorem \ref{T-6-1}] 
If $\gamma_1\in p\mathbb{Z}_p$, then it follows from Remark \ref{Rem-gamma1} that all functions on the both sides of \eqref{main-1} are zero-functions. Hence \eqref{main-1} holds trivially. Therefore we only consider the case 
$\gamma_1\in \mathbb{Z}_p^\times,\ \gamma_2,\ldots,\gamma_r\in p\mathbb{Z}_p$ 
in \eqref{e-6-2}. 

First we consider the case $\gamma_1=1$. 
Setting $\gamma_1=1$ and $\gamma_2,\ldots,\gamma_r\in p\mathbb{Z}_p$ in \eqref{e-6-2}, 
we have
\begin{align*}
  \bb &(n_1,\ldots,n_r;\gamma_1,\gamma_2,\ldots,\gamma_r;c)\\
    &\ +
    \sum_{J\subset\{1,\ldots,r\} \atop 1\in J}
    \Bigr(\frac{1-c}{2}\Bigl)^{r-|J|}
    \int_{\mathbb{Z}_p^{|J|}}
    \prod_{l=1}^r 
    (\sum_{\substack{j\in J\\j\leqslant l}}x_j \gamma_j)^{n_l}
    \prod_{j\in J} d\widetilde{\mathfrak{m}}_c(x_j)
\end{align*}
for $n_1,\ldots,n_r\in \mathbb{N}$. From the condition $\gamma_1=1$ and $\gamma_2,\ldots,\gamma_r\in p\mathbb{Z}_p$, we can see that
\begin{align*}
(\mathbb{Z}_p^{|J|})'_{\{\gamma_j\}_{j\in J}}&=\mathbb{Z}_p^\times \times \mathbb{Z}_p^{|J|-1}=\mathbb{Z}_p^{|J|}\setminus \left(p\mathbb{Z}_p \times \mathbb{Z}_p^{|J|-1}\right)
\end{align*}
for any $J\subset\{1,\ldots,r\}$ with $1\in J$. 
Therefore we have
\begin{align}
&   \sum_{J\subset\{1,\ldots,r\} \atop 1\in J}
    \Bigr(\frac{1-c}{2}\Bigl)^{r-|J|}\notag\\
&   \qquad \times L_{p,|J|}\Bigl(\Bigl(-\sum_{j_\mu(J)\leqslant l<j_{\mu+1}(J)}n_l\Bigr)_{\mu=1}^{|J|};\bigl(\omega^{\sum_{j_\mu(J)\leqslant l<j_{\mu+1}(J)}n_l}\bigr)_{\mu=1}^{|J|};(\gamma_{j})_{j\in J};c\Bigr)\notag\\
    &
    =
    \sum_{J\subset\{1,\ldots,r\}\atop 1\in J}
    \Bigr(\frac{1-c}{2}\Bigl)^{r-|J|}
    \int_{\mathbb{Z}_p^\times \times \mathbb{Z}_p^{|J|-1}}
    \prod_{l=1}^r 
    (\sum_{\substack{j\in J\\ j\leqslant l}}x_j \gamma_j)^{n_l}
    \prod_{j\in J} d\widetilde{\mathfrak{m}}_c(x_j)\notag\\
    &=
    \sum_{J\subset\{1,\ldots,r\} \atop 1\in J}
    \Bigr(\frac{1-c}{2}\Bigl)^{r-|J|}
    \int_{\mathbb{Z}_p^{|J|}}
    \prod_{l=1}^r 
    (\sum_{\substack{j\in J\\i\leqslant l}}x_j \gamma_j)^{n_l}
    \prod_{j\in J} d\widetilde{\mathfrak{m}}_c(x_j)\notag\\
&   \ \ -p^{\sum_{l=1}^{r}n_l}\sum_{J\subset\{1,\ldots,r\}\atop 1\in J}
    \Bigr(\frac{1-c}{2}\Bigl)^{r-|J|}
    \int_{\mathbb{Z}_p^{|J|}}
    \prod_{l=1}^r 
    (x_1+\sum_{\substack{j\in J\\1<j\leqslant l}}x_j \gamma_j/p)^{n_l}
    \prod_{j\in J} d\widetilde{\mathfrak{m}}_c(x_j)\notag\\
& =  \bb (n_1,\ldots,n_r;1,\gamma_2,\ldots,\gamma_r;c)
    -p^{\sum_{l=1}^{r}n_l}
     \bb (n_1,\ldots,n_r;1,\gamma_2/p,\ldots,\gamma_r/p;c), \label{eq-zero}
\end{align}
which is equal to $0$ when $n_1+\cdots+n_r\not\equiv r\pmod 2$ because of \eqref{parity-2}. 
Let $k_1,\ldots,k_r\in \mathbb{Z}$ with $k_1+\cdots+k_r\not\equiv r\pmod 2$. Then the above consideration gives that
\begin{align}
&   \sum_{J\subsetneq\{1,\ldots,r\}\atop 1\in J}
    \Bigr(\frac{1-c}{2}\Bigl)^{r-|J|}\notag \\
&   \times L_{p,|J|}\Bigl(\Bigl(-\sum_{j_\mu(J)\leqslant l<j_{\mu+1}(J)}n_l\Bigr)_{\mu=1}^{|J|};\bigl(\omega^{\sum_{j_\mu(J)\leqslant l<j_{\mu+1}(J)}k_l}\bigr)_{\mu=1}^{|J|};(\gamma_{j})_{j\in J};c\Bigr)\notag\\
& \ +L_{p,r}(-n_1,\ldots,-n_r;\omega^{k_1},\ldots,\omega^{k_r};\gamma_1,\gamma_2,\ldots,\gamma_{r};c)=0 \label{L-parity}
\end{align}
holds on 
$$S_{\{k_j\}}=\{(n_1,\ldots,n_r) \in \mathbb{N}^{r}\mid n_j\in k_j+(p-1)\mathbb{Z}\ (1\leqslant j \leqslant r)\}.$$
Since $S_{\{k_j\}}$ is dense in $\mathbb{Z}_p^r$, we obtain \eqref{main-1} in the case $\gamma_1=1$. 

Next we consider the case $\gamma_1\in \mathbb{Z}_p^\times$. 
We can easily check that if $\gamma_1\in \mathbb{Z}_p^\times$, 
\begin{align*}
& L_{p,r}((s_j);(\omega^{k_j});(\gamma_j);c) \\
& \quad  =\langle \gamma_1\rangle^{-s_1-\cdots -s_r} \omega^{k_1+\cdots +k_r}(\gamma_1) 
    \ L_{p,r}((s_j);(\omega^{k_j});(\gamma_j/\gamma_1);c),\\
& L_{p,|J|}\Bigl(\Bigl(\sum_{j_\mu(J)\leqslant l<j_{\mu+1}(J)}s_l\Bigr)_{\mu=1}^{|J|};\bigl(\omega^{\sum_{j_\mu(J)\leqslant l<j_{\mu+1}(J)}k_l}\bigr)_{\mu=1}^{|J|};(\gamma_{j})_{j\in J};c\Bigr)\\
& \quad  =\langle \gamma_1\rangle^{-s_1-\cdots -s_r} \omega^{k_1+\cdots +k_r}(\gamma_1) \\
& \qquad \times L_{p,|J|}\Bigl(\Bigl(\sum_{j_\mu(J)\leqslant l<j_{\mu+1}(J)}s_l\Bigr)_{\mu=1}^{|J|};\bigl(\omega^{\sum_{j_\mu(J)\leqslant l<j_{\mu+1}(J)}k_l}\bigr)_{\mu=1}^{|J|};(\gamma_{j}/\gamma_1)_{j\in J};c\Bigr)
\end{align*}
for $J\subsetneq \{1,\ldots,r\}$ with $1\in J$. 
Since we already proved \eqref{main-1} corresponding to the case $\{\gamma_j/\gamma_1\}_{j=1}^{r}$ in the above argument, we consequently obtain \eqref{main-1} in the case $\gamma_1\in \mathbb{Z}_p^\times$ by multiplying 
the both sides by
$$\langle \gamma_1\rangle^{-s_1-\cdots -s_r} \omega^{k_1+\cdots +k_r}(\gamma_1). $$
Thus we obtain the proof of Theorem \ref{T-6-1}.
\end{proof}

\begin{remark}\label{Rem-Parity-result}
Note that each $p$-adic multiple $L$-function appearing on the right-hand side of \eqref{main-1} is of depth lower  than $r$. 
The relation \eqref{main-1} reminds us of the \textit{parity result} for MZVs which implies that 
the MZV whose depth and weight are of different parity can be expressed as a polynomial in MZVs of lower depth with $\mathbb{Q}$-coefficients.
In fact, \eqref{L-parity} shows that 
$L_{p,r}((-n_j);(\omega^{k_j});(\gamma_j);c)$ can be expressed as a polynomial in $p$-adic multiple $L$-values of lower depth than $r$ at non-positive integers with $\mathbb{Q}$-coefficients, when $n_1+\cdots+n_r\not\equiv r\pmod 2$. This can be regarded as a $p$-adic version of the parity result for $p$-adic multiple $L$-values. 
By the density property, we obtain its continuous version, that is, the functional relation \eqref{main-1}.
\end{remark}

The following result corresponds to the case $r=1$ in Theorem \ref{T-6-1}.

\begin{example}\label{Rem-zero-2}
Setting $r=1$ in \eqref{eq-zero}, we obtain from Definition \ref{Def-Kubota-L} that 
\begin{align*}
    &
    \Bigr(\frac{1-c}{2}\Bigl)^{r-1}
    \int_{\mathbb{Z}_p^\times}
    x_1^{n_1}
    d\widetilde{\mathfrak{m}}_c(x_1)\left(=\Bigr(\frac{1-c}{2}\Bigl)^{r-1}
    \left(c^{n_1+1}-1\right)L_p(-n_1;\omega^{n_1+1})\right)\notag\\
& =  \bb (n_1;1;c)
    -p^{n_1}
     \bb (n_1;1;c)=0
\end{align*}
for $n_1\in \mathbb{N}$ satisfying $n_1\not\equiv 1\pmod 2$ namely $n_1+1$ is odd, because of \eqref{parity-2}. Hence, if $k$ is odd then $L_p(-n_1;\omega^{k})=0$ for $n_1\in \mathbb{N}$ satisfying $n_1+1\equiv k\ ({\rm mod}\ p-1)$. This implies 
$$
L_p(s,\omega^k)\equiv 0
$$
when $k$ is odd, the statement of Proposition \ref{Prop-zero}. Thus we can regard Theorem \ref{T-6-1} as a multiple version of Proposition \ref{Prop-zero}.
\end{example}

Next we consider the case $r=2$.

\begin{example} \label{P-Func-eq-1}
Let $k,l\in \mathbb{N}_0$ with $2\nmid (k+l)$,
$\gamma_1\in \mathbb{Z}_p^\times$, $\gamma_2\in p\mathbb{Z}_p$ and $c\in \mathbb{Z}_p^\times $. Then, for $s_1,s_2\in \mathbb{Z}_p$, we obtain from \eqref{main-1} and \eqref{1-p-LF-gamma} that 
\begin{align}
& L_{p,2}(s_1,s_2;\omega^{k},\omega^{l};\gamma_1,\gamma_2;c)\notag\\
& ={\frac{c-1}{2}}\langle \gamma_1\rangle^{-s_1-s_2} \omega^{k+l}(\gamma_1)\left( \langle c\rangle^{1-s_1-s_2}\omega^{k+l+1}(c)-1\right)L_p(s_1+s_2;\omega^{k+l+1}). \label{Func-eq-1}
\end{align}
Note that this functional relation in the case $p>2$, $c=2$ and $\gamma_1=1$ was already proved in \cite{KMT-IJNT} by a different method. 
\end{example}


\begin{remark}\label{R-3-2} 
As we noted above, when $k+l$ is even, $L_p(s;\omega^{k+l+1})$ is the zero-function, so is the right-hand of \eqref{Func-eq-1}. On the other hand, even if $k+l$ is even, the left-hand of \eqref{Func-eq-1} is not necessarily the zero-function. In fact, it follows from \eqref{L2-val} that
\begin{align*}
& L_{p,2}(-1,-1;\omega,\omega;1,\gamma_2;c)=\frac{(1-c^2)^2(1-p)}{4}B_2^2\gamma_2
\end{align*}
for $\gamma_2\in p\mathbb{Z}_p$, which does not vanish if $\gamma_2\not=0$. Therefore this case implies that $L_{p,2}(s_1,s_2;\omega,\omega;1,\gamma_2;c)$ 
is not the zero-function.
Therefore $L_{p,2}(s_1,s_2;\omega^{k},\omega^l;1,\gamma_2;c)$ 
seems to have more information beyond the Kubota-Leopoldt $p$-adic $L$-function.
\end{remark}

Further we consider the case $r=3$.

\begin{example}\label{E-6-3}
Let 
$k_1,k_2,k_3\in \mathbb{Z}$ with $2\mid (k_1+k_2+k_3)$, $\gamma_1\in \mathbb{Z}_p^\times$, $\gamma_2,\gamma_3\in p\mathbb{Z}_p$ and $c\in \mathbb{Z}_p^\times $. Then, for $s_1,s_2,s_3\in \mathbb{Z}_p$, we obtain from \eqref{main-1} that 
\begin{align*}
& {L_{p,3}(s_1,s_2,s_3;\omega^{k_1},\omega^{k_2},\omega^{k_3};\gamma_1,\gamma_2,\gamma_3;c)} \notag\\
& \quad {=\frac{1-c}{2}L_{p,2}(s_1,s_2+s_3;\omega^{k_1},\omega^{k_2+k_3};\gamma_1,\gamma_2;c)} {+\frac{1-c}{2}L_{p,2}(s_1+s_2,s_3;\omega^{k_1+k_2},\omega^{k_3};\gamma_1,\gamma_3;c)} \notag\\
& \quad \quad 
-\left(\frac{1-c}{2}\right)^2\langle \gamma_1\rangle^{-s_1-s_2-s_3}\omega^{k_1+k_2+k_3}(\gamma_1)\\
& \qquad \quad 
\times \left(\langle c\rangle^{1-s_1-s_2-s_3}\omega^{k_1+k_2+k_3+1}(c)-1\right)L_{p}(s_1+s_2+s_3;\omega^{k_1+k_2+k_3+1}).
\end{align*}
Note that from \eqref{1-p-LF-gamma} the third term on the right-hand side vanishes because $L_p(s;\omega^{k})$ is the zero function when $k$ is odd. 
\end{example}

Using Theorem \ref{T-6-1} and Examples \ref{P-Func-eq-1} and \ref{E-6-3}, we can immediately obtain the following result by induction on $r\geqslant 2$. 

\begin{corollary}\label{C-6-4} 
Let $r\in \mathbb{N}_{\geqslant 2}$, $k_1,\ldots,k_r\in \mathbb{Z}$ with $k_1+\cdots+k_r\not\equiv r\pmod 2$, $\gamma_1\in \mathbb{Z}_p$, $\gamma_2,\ldots,\gamma_{r}\in p\mathbb{Z}_p$ and $c\in \mathbb{Z}_p^\times$. Then 
$$L_{p,r} (s_1,\ldots,s_r;\omega^{k_1},\ldots,\omega^{k_r};\gamma_1,\ldots,\gamma_{r};c)$$
can be expressed as a polynomial in $p$-adic $j$-ple $L$-functions for $j\in \{1,2,\ldots,r-1\}$ satisfying $j \not\equiv r \pmod 2$, with $\mathbb{Q}$-coefficients.
\end{corollary}

\

\section{Special values of $p$-adic multiple $L$-functions at positive integers}\label{sec-5}
In our previous section, particularly in Theorem \ref{T-main-1}, we saw that 
the special values of our $p$-adic multiple $L$-functions at non-positive integers
are expressed  in terms of the twisted multiple Bernoulli numbers (Definition \ref{Def-M-Bern}),
which are
the special values of the complex multiple 
zeta-functions of generalized Euler-Zagier-Lerch type 
at non-positive integers
(cf. Theorem \ref{T-multiple}).
In contrast, in this section we will  discuss their special values 
at positive integers.
For this purpose, 
we will introduce a specific $p$-adic function in 
each subsection.
In the complex case, 
the special values 
of multiple zeta function (cf.\ \eqref{MZF-def})
at positive integers  
are given by the special values of multiple polylogarithms 
at unity (see \eqref{zeta=Li}).
Our main result is Theorem \ref{L-Li theorem-2}
where we will show a $p$-adic analogue of the equality \eqref{zeta=Li}:
We will establish a close relationship of  our $p$-adic multiple $L$-functions
with the $p$-adic TMPL's
\footnote{
In this paper,  TMPL stands for twisted multiple polylogarithm.
},
generalizations of  $p$-adic  multiple polylogarithm
introduced by the first-named author \cite{Fu1, F2}
and Yamashita \cite{Y} 
for a study of the $p$-adic realization of certain motivic fundamental group.
It is achieved 
by showing that the special values of $p$-adic multiple $L$-function at positive integers
are described by  
$p$-adic twisted multiple $L$-values;
the special values at unity
of the $p$-adic TMPL's (Definition \ref{Def-TMPL}, see also \cite{Fu1,Y})
constructed by Coleman's $p$-adic iterated integration theory \cite{C}.
Our theorem generalizes a previous work of Coleman \cite{C}.
To connect $p$-adic multiple $L$-functions with $p$-adic TMPL's,
we will introduce $p$-adic rigid TMPL's (Definition \ref{def of pMMPL})
and their partial versions (Definition \ref{def of pPMPL}) as intermediate objects
and investigate their several basic properties mainly in 
Subsections \ref{sec-5-3} and \ref{sec-5-4}. 

\subsection{$p$-adic rigid twisted multiple polylogarithms} \label{sec-5-3}
This subsection is to introduce $p$-adic rigid TMPL's
(Definition \ref{def of pMMPL})
and to give a description of special values of $p$-adic multiple $L$-functions at positive integers
by special values of $p$-adic rigid TMPL's at roots of unity
 (Theorem \ref{L-ell theorem})
which extends Coleman's result \eqref{Coleman equality}.

First, we briefly review the minimum basics of  rigid analysis in our specific case.

\begin{notation}[cf. \cite{BGR} etc.]\label{rigid-basics}
\label{basics on rigid}
Let $\alpha_1, \dots, \alpha_n\in {\mathbb C}_p$ and $\rs_1,\dots,\rs_n\in{\mathbb Q}_{>0}$
and $\rs_0\in{\mathbb Q}_{\geqslant 0}$.
The space
\begin{equation}\label{affinoid presentation}
X=\left\{z\in {\bf P}^1({\mathbb C}_p)\bigm| |z-\alpha_i|_p\geqslant \rs_i \ (i=1,\dots, n),
|z|_p\leqslant 1/\rs_0
\right\}
\end{equation}
is equipped with a structure of {\it affinoid}, a special type of rigid analytic space. 
A {\it rigid analytic function} on $X$ is 
a functions $f(z)$ on $X$
which admits  the convergent expansion
$$
f(z)=\sum_{m\geqslant 0}a_{m}(\infty;f)z^{m}
+\sum_{i=1}^n\sum_{m>0}\frac{a_{m}(\alpha_i;f)}{(z-\alpha_i)^{m}}
$$
with $\mathbb C_p$-coefficients.
The  expressions are actually unique 
(the Mittag-Leffler decompositions), 
which can be  shown from, for example \cite[I.1.3]{FvP}.
We denote the algebra of rigid analytic functions on $X$ by $A^{\text{rig}}(X)$.
\end{notation}

\begin{notation}
For $a$ in $\textbf{P}^{1}(\mathbb{C}_p)$, $\bar a$ means the image ${\rm red}(a)$ by the reduction map 
$${\rm red}:\textbf{P}^{1}(\mathbb{C}_p) \to \textbf{P}^{1}({\overline{\mathbb{F}}_p})\left(=\overline{\mathbb{F}}_p\cup \{\bar \infty\}\right),$$
where $\overline{\mathbb{F}}_p$ is the algebraic closure of ${\mathbb{F}}_p$. 
For a finite subset $D\subset \textbf{P}^{1}(\mathbb{C}_p)$, we define $\bar D={\rm red}(D)\subset \textbf{P}^{1}(\overline{\mathbb{F}}_p)$.
For $a_0 \in \textbf{P}^{1}(\overline{\mathbb{F}}_p)$, 
its tubular neighborhood $]a_0[$ means the inverse image ${\rm red}^{-1}(a_0)$ of the reduction map.
Namely, 
$]\bar a[=\{x\in \textbf{P}^{1}(\mathbb{C}_p)\,|\,|x-a|_p<1\}$
for $a\in {\mathbb{C}}_p$,  $]\bar 0[={\frak M}_{\mathbb{C}_p}$
and $]\bar \infty[=\textbf{P}^{1}(\mathbb{C}_p)\setminus {\mathcal O}_{\mathbb{C}_p}$. 
For a finite subset 
$S\subset  \textbf{P}^{1}(\overline{\mathbb{F}}_p)$,
we define $]S[:={\rm red}^{-1}(S)\subset \textbf{P}^{1}(\mathbb{C}_p)$.
By abuse of notation, 
we denote $A^{\text{rig}}({\bf P}^1({\mathbb C}_p)- ]S[)$
by $A^{\text{rig}}({\bf P}^1\setminus S)$.
\end{notation}

We remind two fundamental properties of rigid analytic functions:
\begin{proposition}[{\cite[Chapter 6]{BGR}, etc.}]\label{two fundamental properties of rigid analytic functions}
Let $X$ be as in \eqref{affinoid presentation}.
Then the following holds.

{\rm (i)}\ {\it Coincidence principle}:
If two rigid analytic functions $f(z)$ and $g(z)$ coincide in a subaffinoid of $X$,
then  they coincide on the whole of $X$. 

{\rm (ii)}\ The algebra $A^{\text{rig}}(X)$ forms a Banach algebra  with the supremum norm.
\end{proposition}

The following function plays the main role in this subsection.

\begin{definition}\label{def of pMMPL}
Let $n_1,\dots,n_r\in \mathbb{N}$ and $\xi_1,\dots,\xi_{r}\in\mathbb{C}_p$ with $|\xi_j|_p\leqslant 1$ ($1\leqslant j\leqslant r$).
The {\bf $p$-adic rigid TMPL} 
$\ell^{(p)}_{n_1,\dots,n_r}(\xi_1,\dots,\xi_{r};z)$ is defined by the following 
$p$-adic power series:
\begin{equation}\label{series expression for ell}
\ell^{(p)}_{n_1,\dots,n_r}(\xi_1,\dots,\xi_{r},z):=
\underset{(k_1,p)=\cdots=(k_r,p)=1}{\underset{0<k_1<\cdots<k_{r}}{\sum}}
\frac{\xi_1^{k_1}\cdots\xi_r^{k_r}}{k_1^{n_1}\cdots k_r^{n_r}}z^{k_r}
\end{equation}
which converges for $z\in ]\bar{0}[=\{x\in\mathbb{C}_p\bigm| \ |x|_p<1\}$ 
by $|\xi_j|_p\leqslant 1$ for $1\leqslant j\leqslant r$.
\end{definition}

It will be proved 
that  it is rigid analytic in Proposition \ref{rigidness} 
and is furthermore overconvergent in Theorem \ref{rigidness III}.
We remark that when $r=1$, $\ell^{(p)}_{n}(1;z)$
is equal to the $p$-adic polylogarithm
$\ell_n^{(p)}(z)$
in \cite[p.196]{C}.
The following integral expressions 
are generalization of \cite[Lemma 7.2]{C}.

\begin{theorem}\label{integral theorem}
Let $n_1,\dots,n_r\in \mathbb{N}$ and $\xi_1,\dots,\xi_{r}\in\mathbb{C}_p$ with $|\xi_j|_p\leqslant 1$ ($1\leqslant j\leqslant r$).
Set a finite subset $S$ of  $\textbf{P}^{1}(\overline{\mathbb{F}}_p)$ by
\footnote{
Here we ignore the multiplicity.
}
\begin{equation}\label{S0}
S=\{
\overline{\xi_{r}^{-1}},\overline{(\xi_{r-1}\xi_{r})^{-1}},
\dots,\overline{(\xi_1\cdots\xi_{r-1})^{-1}}\}.
\end{equation}
Then the   $p$-adic rigid TMPL
$\ell^{(p)}_{n_1,\dots,n_r}(\xi_1,\dots,\xi_{r};z)$ is 
extended  into $\textbf{P}^{1}({\mathbb{C}}_p)- ]S[$
as a function on $z$
by  the following $p$-adic integral expression:
\begin{align}\label{integral expression}
\ell&^{(p)}_{ n_1,\dots,n_r}  (\xi_1,\dots,\xi_{r};z)= \notag \\
&\int_{(\mathbb{Z}_p^r)'_{\{1\}}}\langle x_1\rangle^{-n_1}\langle x_1+x_2\rangle^{-n_2}\cdots\langle x_1+\cdots+x_r\rangle^{-n_r}\cdot 
\omega(x_1)^{-n_1}\omega(x_1+x_2)^{-n_2}\cdots\omega(x_1+\cdots+x_r)^{-n_r} \notag \\
&\qquad\qquad\qquad d{\frak m}_{\xi_1\cdots\xi_{r}z}(x_1)\cdots d{\frak m}_{\xi_{r} z}(x_{r}), 
\end{align}
where $(\mathbb{Z}_p^r)'_{\{1\}}=\Bigl\{(x_1,\dots,x_r)\in\mathbb{Z}_p^r\Bigm|p\nmid x_1, p\nmid (x_1+x_2),\dots, p\nmid (x_1+\cdots+x_r)\Bigr\}$
(cf. \eqref{region}).
\end{theorem}

\begin{proof}
Since $\langle x\rangle\cdot \omega (x)=x\neq 0$
for $x\in \mathbb{Z}_p^\times$,
the right-hand side of \eqref{integral expression} is
\begin{align}
\int_{(\mathbb{Z}_p^r)'_{\{1\}}}
& x_1^{-n_1}(x_1+x_2)^{-n_2}\cdots (x_1+\cdots+x_r)^{-n_r}
d{\frak m}_{\xi_1\cdots\xi_{r}z}(x_1)\cdots d{\frak m}_{\xi_{r} z}(x_{r}) 
\notag\\
&=\lim_{M\to\infty}
\underset{(l_1,p)=\cdots=(l_1+\cdots+l_r,p)=1}{\underset{0<l_1, \dots,l_{r}<p^M}{\sum}}
\frac{\xi_1^{l_1}\xi_2^{l_1+l_2}\cdots\xi_{r}^{l_1+\cdots+l_{r}}}{l_1^{n_1}(l_1+l_2)^{n_2}\cdots (l_1+\cdots+l_r)^{n_r}}z^{l_1+\cdots+l_r} \notag \\
&\qquad \qquad \cdot\frac{1}{1-(\xi_1\cdots\xi_{r} z)^{p^M}}
\cdots 
\frac{1}{1-(\xi_{r-1}\xi_r z)^{p^M}}\cdot
\frac{1}{1-(\xi_r z)^{p^M}}.
\label{limit expression}\\
&=\lim_{M\to\infty} \qquad g^M_{n_1,\dots,n_r}(\xi_1,\dots,\xi_{r};z)
\qquad \qquad\qquad \text{(say).} \notag
\end{align}
By direct calculations it can be shown that it is equal to
the right-hand side of \eqref{series expression for ell} 
when $|z|_p<1$.
\end{proof}

As for $g^M_{n_1,\dots,n_r}(\xi_1,\dots,\xi_{r};z)$
defined in the above proof, we have

\begin{lemma}\label{congruence}
Fix  $n_1,\dots, n_r, M\in \mathbb{N}$ and $\xi_1,\dots,\xi_{r}\in\mathbb{C}_p$ with $|\xi_j|_p=1$ ($1\leqslant j\leqslant r$).
Then, for 
$
z_0\in 
{\bf P}^1({\mathbb C}_p) -]\overline{\xi_{r}^{-1}},\overline{(\xi_{r-1}\xi_r)^{-1}},\dots,\overline{(\xi_1\cdots\xi_{r})^{-1}}[ 
$,  we have
$$
g^M_{n_1,\dots,n_r}(\xi_1,\dots,\xi_{r};z_0)\in{\mathcal O}_{{\mathbb C}_p}
$$
and
$$
g^{M+1}_{n_1,\dots,n_r}(\xi_1,\dots,\xi_{r};z_0)\equiv g^M_{n_1,\dots,n_r}(\xi_1,\dots,\xi_{r};z_0) \pmod{p^M}.
$$
\end{lemma}

\begin{proof}
When $z_0\in{\bf P}^1({\mathbb C}_p) 
-]\overline{\xi_{r}^{-1}},\overline{(\xi_{r-1}\xi_r)^{-1}},\dots,\overline{(\xi_1\cdots\xi_{r})^{-1}},
\overline{\infty}[$,
namely when $z_0\in{\mathcal O}_{{\mathbb C}_p}
-]S[
$, it is clear that
$g^M_{n_1,\dots,n_r}(\xi_1,\dots,\xi_{r};z_0)\in{\mathcal O}_{{\mathbb C}_p}$.
In the definition of $g^{M+1}_{n_1,\dots,n_r}(\xi_1,\dots,\xi_{r};z)$, 
$l_j$ ($1\leqslant j\leqslant r$) is running in the interval
$(0,p^{M+1})$.
Writing $l_j=l_j^{\prime}+kp^M$ ($0<l_j^{\prime}<p^M$ and $1\leqslant k\leqslant p-1$),
we have

\begin{align*}
g^{M+1}_{n_1,\dots,n_r}&(\xi_1,\dots,\xi_{r};z_0)
\equiv
\underset{(l_1^{\prime} ,p)=\cdots=(l_1^{\prime} +\cdots+l_r^{\prime} ,p)=1}
{\underset{0<l_1^{\prime} , \dots,l_{r}^{\prime} <p^M}{\sum}} 
\frac{\xi_1^{l_1^{\prime} }\xi_2^{l_1^{\prime} +l_2^{\prime} }
\cdots\xi_{r}^{l_1^{\prime} +\cdots+l_{r}^{\prime} }z_0^{l_1^{\prime} +\cdots+l_r^{\prime} }}{l_1^{\prime n_1}(l_1^{\prime} +l_2^{\prime} )^{n_2}\cdots (l_1^{\prime} +\cdots+l_r^{\prime} )^{n_r}} \\ 
&\cdot\{1+(\xi_1\cdots\xi_{r} z_0)^{p^M}+(\xi_1\cdots\xi_{r} z_0)^{2p^M}+\cdots+(\xi_1\cdots\xi_{r} z_0)^{(p-1)p^M}\} \\
&\cdot\{1+(\xi_2\cdots\xi_{r} z_0)^{p^M}+(\xi_2\cdots\xi_{r} z_0)^{2p^M}+\cdots+(\xi_2\cdots\xi_{r} z_0)^{(p-1)p^M}\} \\
&\qquad\qquad\qquad\qquad \cdots \\
&\cdot\{1+(\xi_{r} z_0)^{p^M}+(\xi_{r} z_0)^{2p^M}+\cdots+(\xi_{r} z_0)^{(p-1)p^M}\}  \\
&\cdot\frac{1}{1-(\xi_1\cdots\xi_{r} z)^{p^{M+1}}}
\cdots 
\frac{1}{1-(\xi_{r-1}\xi_r z)^{p^{M+1}}}\cdot
\frac{1}{1-(\xi_r z)^{p^{M+1}}}
\pmod{p^M} \\
&
\qquad\qquad\qquad\qquad
=g^M_{n_1,\dots,n_r}(\xi_1,\dots,\xi_{r};z).
\end{align*}

When $z_0\in]\bar{\infty}[$, put $\varepsilon=\frac{1}{z_0}\in]\bar{0}[$.
By direct calculations,
$g^M_{n_1,\dots,n_r}(\xi_1,\dots,\xi_{r};z_0)\in{\mathcal O}_{{\mathbb C}_p}$.
We have
\begin{align*}
g^{M}_{n_1,\dots,n_r}&(\xi_1,\dots,\xi_{r};z_0)
=
\underset{(l_1,p)=\cdots=(l_1+\cdots+l_r,p)=1}{\underset{0<l_1, \dots,l_{r}<p^M}{\sum}} 
\frac{(\frac{\xi_1\xi_2\cdots\xi_{r}}{\varepsilon})^{l_1}(\frac{\xi_2\cdots\xi_{r}}{\varepsilon})^{l_2}\cdots (\frac{\xi_{r-1}\xi_r}{\varepsilon})^{l_{r-1}}
(\frac{\xi_r}{\varepsilon})^{l_r}}
{l_1^{n_1}(l_1+l_2)^{n_2}\cdots (l_1+\cdots+l_r)^{n_r}} \\
&\qquad\qquad\cdot\frac{1}{1-(\frac{\xi_1\cdots\xi_{r}}{\varepsilon})^{p^M}}
\cdots 
\frac{1}{1-(\frac{\xi_{r-1}\xi_r}{\varepsilon})^{p^M}}\cdot
\frac{1}{1-(\frac{\xi_r}{\varepsilon})^{p^M}} \\
&=(-1)^r
\underset{(l_1,p)=\cdots=(l_1+\cdots+l_r,p)=1}{\underset{0<l_1, \dots,l_{r}<p^M}{\sum}} 
\frac{(\frac{\varepsilon}{\xi_1\xi_2\cdots\xi_{r}})^{p^M-l_1}(\frac{\varepsilon}{\xi_2\cdots\xi_{r}})^{p^M-l_2}\cdots (\frac{\varepsilon}{\xi_{r-1}\xi_r})^{p^M-l_{r-1}}
(\frac{\varepsilon}{\xi_r})^{p^M-l_r}}
{l_1^{n_1}(l_1+l_2)^{n_2}\cdots (l_1+\cdots+l_r)^{n_r}} \\
&\qquad\qquad\cdot\frac{1}{1-(\frac{\varepsilon}{\xi_1\cdots\xi_{r}})^{p^M}}
\cdots 
\frac{1}{1-(\frac{\varepsilon}{\xi_{r-1}\xi_r})^{p^M}}\cdot
\frac{1}{1-(\frac{\varepsilon}{\xi_r})^{p^M}} \\
&=(-1)^r
\underset{(l_1,p)=\cdots=(l_1+\cdots+l_r,p)=1}{\underset{0<l_1, \dots,l_{r}<p^M}{\sum}} 
\frac{(\frac{\varepsilon}{\xi_1\xi_2\cdots\xi_{r}})^{l_1}(\frac{\varepsilon}{\xi_2\cdots\xi_{r}})^{l_2}\cdots (\frac{\varepsilon}{\xi_{r-1}\xi_r})^{l_{r-1}}
(\frac{\varepsilon}{\xi_r})^{l_r}}
{(p^M-l_1)^{n_1}(2p^M-l_1-l_2)^{n_2}\cdots (rp^M-l_1-\cdots-l_r)^{n_r}} \\
&\qquad\qquad\cdot\frac{1}{1-(\frac{\varepsilon}{\xi_1\cdots\xi_{r}})^{p^M}}
\cdots 
\frac{1}{1-(\frac{\varepsilon}{\xi_{r-1}\xi_r})^{p^M}}\cdot
\frac{1}{1-(\frac{\varepsilon}{\xi_r})^{p^M}} \\
&\equiv (-1)^{r+n_1+\cdots+n_r}
\underset{(l_1,p)=\cdots=(l_1+\cdots+l_r,p)=1}{\underset{0<l_1, \dots,l_{r}<p^M}{\sum}} 
\frac{(\frac{\varepsilon}{\xi_1\xi_2\cdots\xi_{r}})^{l_1}(\frac{\varepsilon}{\xi_2\cdots\xi_{r}})^{l_2}\cdots (\frac{\varepsilon}{\xi_{r-1}\xi_r})^{l_{r-1}}
(\frac{\varepsilon}{\xi_r})^{l_r}}
{{l_1}^{n_1}(l_1+l_2)^{n_2}\cdots (l_1+\cdots+l_r)^{n_r}} \\
&\qquad\qquad\cdot\frac{1}{1-(\frac{\varepsilon}{\xi_1\cdots\xi_{r}})^{p^M}}
\cdots 
\frac{1}{1-(\frac{\varepsilon}{\xi_{r-1}\xi_r})^{p^M}}\cdot
\frac{1}{1-(\frac{\varepsilon}{\xi_r})^{p^M}} \pmod{p^M} \\
&= (-1)^{r+n_1+\cdots+n_r}
g^{M}_{n_1,\dots,n_r}(\xi_1^{-1},\dots,\xi_{r}^{-1};\varepsilon).
\end{align*}
Therefore by our previous argument and by $|\xi_j|_p=|\xi_j^{-1}|_p=1$,
it follows that 
\begin{align*}
g^{M+1}_{n_1,\dots,n_r}&(\xi_1,\dots,\xi_{r};z_0)
\equiv  (-1)^{r+n_1+\cdots+n_r}
g^{M+1}_{n_1,\dots,n_r}(\xi_1^{-1},\dots,\xi_{r}^{-1};\varepsilon) \\
&\equiv (-1)^{r+n_1+\cdots+n_r}
g^{M}_{n_1,\dots,n_r}(\xi_1^{-1},\dots,\xi_{r}^{-1};\varepsilon)
\equiv
g^{M}_{n_1,\dots,n_r}(\xi_1,\dots,\xi_{r};z_0)
\pmod{p^M}.
\end{align*}
\end{proof}


Theorem \ref{integral theorem} and Lemma \ref{congruence} imply the following:

\begin{proposition}\label{rigidness}
Fix  $n_1,\dots,n_r\in \mathbb{N}$ and $\xi_1,\dots,\xi_{r}\in\mathbb{C}_p$ with $|\xi_j|_p=1$ ($1\leqslant j\leqslant r$).
By our integral formula \eqref{integral expression},
the  $p$-adic rigid TMPL 
$\ell^{(p)}_{n_1,\dots,n_r}(\xi_1,\dots,\xi_{r};z)$ is a rigid analytic function on
$
{\bf P}^1({\mathbb C}_p) - ]S[
$.
Namely, 
$$
\ell^{(p)}_{n_1,\dots,n_r}(\xi_1,\dots,\xi_{r};z)
\in A^{\rm{rig}}( {\bf P}^1\setminus S
).
$$
\end{proposition}

\begin{proof}
Since the space ${\bf P}^1({\mathbb C}_p) - ]S[$
is an affinoid and
the algebra 
$A^{\text{rig}}( {\bf P}^1\setminus S)$
forms a Banach algebra by the supremum norm (cf. Notation \ref{basics on rigid}),
by Lemma \ref{congruence}, the rational functions
$$
g^{M}_{n_1,\dots,n_r}(\xi_1,\dots,\xi_{r};z)\in
A^{\text{rig}}( {\bf P}^1\setminus S)
$$
uniformly converge to a rigid analytic function 
$\ell(z)\in A^{\text{rig}}( {\bf P}^1\setminus S)$
when $M$ goes to $\infty$ thanks to Proposition \ref{two fundamental properties of rigid analytic functions}.
It is easy to see that the restriction of $\ell(z)$ into 
$\textbf{P}^{1}({\mathbb{C}}_p)- ]S[$ coincides with \eqref{limit expression}, 
hence with $\ell^{(p)}_{n_1,\dots,n_r}(\xi_1,\dots,\xi_{r};z)$.
Therefore the analytic continuation of  $\ell^{(p)}_{n_1,\dots,n_r}(\xi_1,\dots,\xi_{r};z)$
is given by $\ell(z)$.
\end{proof}

From now on we will employ the same symbol 
$\ell^{(p)}_{n_1,\dots,n_r}(\xi_1,\dots,\xi_{r};z)$
to denote its analytic continuation.

We note that, by \eqref{limit expression},

\begin{lemma}\label{ell at infinity}
For $n_1,\dots,n_r\in \mathbb{N}$ and $\xi_1,\dots,\xi_{r}\in\mathbb{C}_p$ with $|\xi_j|_p=1$ ($1\leqslant j\leqslant r$),
we have
$$
\ell^{(p)}_{n_1,\dots,n_r}(\xi_1,\dots,\xi_{r};\infty)=0.$$
\end{lemma}

\begin{proof}
By direct computations, $g^{M}_{n_1,\dots,n_r}(\xi_1,\dots,\xi_{r};\infty)=0$.
Then the claim is obtained 
because $\ell^{(p)}_{n_1,\dots,n_r}(\xi_1,\dots,\xi_{r};z)$ is defined to be the limit of
$g^{M}_{n_1,\dots,n_r}(\xi_1,\dots,\xi_{r};z)$.
\end{proof}

The special values of the $p$-adic multiple $L$-function at positive integer points
are described in terms of the special values of
$p$-adic rigid TMPL
at roots of unity:

\begin{theorem}\label{L-ell theorem}
For  $n_1,\dots,n_r\in \mathbb{N}$
and $c\in \mathbb{N}_{>1}$ with $(c,p)=1$,
\begin{equation}\label{L-ell-formula}
L_{p,r}(n_1,\dots,n_r;\omega^{-n_1},\dots,\omega^{-n_r};1,\dots,1;c)= \\
\underset{\xi_1\cdots\xi_r\neq 1, \ \dots, \ \xi_{r-1}\xi_r\neq 1, \ \xi_r\neq 1}{\sum_{\xi_1^c=\cdots=\xi_r^c=1}}
\ell^{(p)}_{n_1,\dots,n_r}(\xi_1,\dots,\xi_r;1).
\end{equation}
\end{theorem}

\begin{proof}
%
By definition,
\begin{align*}
L_{p,r}(n_1,\dots,&n_r;\omega^{-n_1},\dots,\omega^{-n_r};1,\dots,1;c)= \\
&\int_{(\mathbb{Z}_p^r)'_{\{1\}}}\langle x_1\rangle^{-n_1}\langle x_1+x_2\rangle^{-n_2}\cdots\langle x_1+\cdots+x_r\rangle^{-n_r}  \\
&\qquad\qquad
\cdot\omega(x_1)^{-n_1}\omega(x_1+x_2)^{-n_2}\cdots\omega(x_1+\cdots+x_r)^{-n_r} 
d{\widetilde{\frak m}}_c(x_1)\cdots d{\widetilde{\frak m}}_c(x_r)
\end{align*}
where 
${\widetilde{\frak m}}_c=\underset{\xi\neq 1}{\underset{{\xi^c=1}}{\sum}}{\frak m}_\xi$.
By \eqref{integral expression} and
\begin{align*}
({\widetilde{\frak m}}_c)^r
&=\Bigl\{\underset{\xi_1^{\prime}\neq 1}{\underset{{\xi_1^{\prime c}=1}}{\sum}}{\frak m}_{\xi_1^{\prime}} \Bigr\}\times
\Bigl\{\underset{\xi_2^{\prime}\neq 1}{\underset{{\xi_2^{\prime c}=1}}{\sum}}{\frak m}_{\xi_2^{\prime}} \Bigr\}\times\cdots\times
\Bigl\{\underset{\xi_r^{\prime}\neq 1}{\underset{{\xi_r^{\prime c}=1}}{\sum}}{\frak m}_{\xi_r^{\prime}} \Bigr\}
\\
&=\underset{\xi_1\cdots\xi_r\neq 1,\dots,\xi_{r-1}\xi_r\neq 1,\xi_r\neq 1}{\sum_{\xi_1^c=\cdots=\xi_r^c=1}}
{\frak m}_{\xi_1\cdots\xi_r}\times{\frak m}_{\xi_2\cdots\xi_r}\times
\cdots\times
{\frak m}_{\xi_{r-1}\xi_r}\times{\frak m}_{\xi_r}
\end{align*}
we get  our formula.
\end{proof}

\begin{remark}
It is worthy to note that  the right-hand side of \eqref{L-ell-formula} is $p$-adically
continuous not only with respect to $n_1,\ldots,n_r$ but also with respect to $c$
by Theorem \ref{continuity theorem}.
\end{remark}

As a special case when $r=1$ of Theorem \ref{L-ell theorem},
we recover Coleman's formula in \cite[p.203]{C} below
by Example \ref{example for r=1}.

\begin{example}
For $n\in \mathbb{N}_{>1}$
and $c\in \mathbb{N}_{>1}$ with $(c,p)=1$,
\begin{equation}\label{Coleman equality}
(c^{1-n}-1)\cdot L_{p}(n;\omega^{1-n})
=\underset{\xi\neq 1}{\sum_{\xi^c=1}}\ell^{(p)}_{n}(\xi;1).
\end{equation}
\end{example}

When $r=2$, we have 

\begin{example}
For  $n_1,n_2\in \mathbb{N}$
and $c\in \mathbb{N}_{>1}$ with $(c,p)=1$,
\begin{equation*}
L_{p,2}(n_1,n_2;\omega^{-n_1},\omega^{-n_2};1,1;c)= \\
\underset{\xi_1\xi_2\neq 1, \ \xi_2\neq 1}{\sum_{\xi_1^c=\xi_2^c=1}}
\ell^{(p)}_{n_1,n_2}(\xi_1,\xi_2;1).
\end{equation*}
\end{example}

\subsection{$p$-adic partial twisted multiple polylogarithms}\label{sec-5-4}
We will prove that
$p$-adic rigid TMPL's  (Definition \ref{def of pMMPL})
are overconvergent in Theorem \ref{rigidness II}.
In order to do that,
$p$-adic partial TMPL's will be introduced in Definition \ref{def of pPMPL}
and their properties will be investigated.

First,we recall overconvergent functions and rigid cohomologies
in our particular  case
(consult \cite{Ber} for a general theory)
%
%
%
%

\begin{notation}\label{overconvergent functions and associated cohomologies}
Let
$S=\{s_0,\dots,s_d\}$ (all $s_i$'s are distinct)
be a finite subset of $\textbf{P}^{1}(\overline{\mathbb{F}}_p)$.
An {\it overconvergent function} of  $\textbf{P}^{1}\setminus S$
is a function belonging to the ${\mathbb C}_p$-algebra
$$
A^\dag(  {\bf P}^1\setminus S)
:=
\underset{\lambda\to 1^{-}}{\rm ind\text{-}lim} \ A^\text{rig}(U_\lambda)
$$
where
$U_\lambda$ is the affinoid  
``obtained by removing all closed discs
of radius $\lambda$ around $\hat s_i$ (: a lift of $s_i$)
from
${\bf P}^1({\mathbb C}_p)$'', i.e.
\begin{equation}\label{removing all closed discs}
U_\lambda:={\bf P}^1({\mathbb C}_p)\setminus \bigcup_{0\leqslant i \leqslant d} 
z_i^{-1}\left(\{\alpha\in{\mathbb C}_p\bigm||\alpha|_p\leqslant\lambda\}\right)
\end{equation}
and $z_i$ is a local parameter
\begin{equation}\label{local parameter}
z_i:]s_i[\overset{\sim}{\to} ]\bar{0}[.
\end{equation}
(It is noted that the above $\hat s_i$ is equal to $z_i^{-1}(0)$.)
An overconvergent function of  $ {\bf P}^1\setminus S$ is, in short, 
a function 
which can be analytically extended into
an affinoid which is bigger than
${\bf P}^1({\mathbb C}_p)- ]s_0,s_1,\dots,s_d[$.
We note that the definition of the space of
overconvergent functions 
does not depend on the
choice of  local parameter $z_i$.

\end{notation}

The following lemmas are quite useful.

\begin{lemma}
Assume $s_0=\bar\infty$ and take $\hat{s_0}=\infty$.
Then 
we have a description:
\begin{equation}\label{description of A1}
\begin{split}
A^\dag(  {\bf P}^1\setminus S) \simeq
\Biggl\{f(z)=\sum_{r\geqslant 0}a_{r}(\hat s_0;f)z^{r}
+\sum_{l=1}^d\sum_{m>0}\frac{a_{m}(\hat s_l;f)}{(z-\hat{s_l})^{m}}
\in{\mathbb C}_p[[z,\frac{1}{z-\hat{s_l}}]] \Biggm| 
&\\ 
\frac{|a_{m}(\hat s_l;f)|_p}{\lambda^{m}}\to 0 \  (m\to \infty) 
 \ \text{  for some  } 0<\lambda <1 
\quad (0\leqslant l \leqslant & d) 
\Biggr\} . \\
\end{split}
\end{equation}
While we have
\begin{equation}\label{description of A2}
\begin{split}
A^{\text{rig}}({\bf P}^1\setminus S) \simeq
\Biggl\{f(z)=&\sum_{m\geqslant 0}a_{m}(\hat s_0;f)z^{m}
+\sum_{l=1}^d\sum_{m>0}\frac{a_{m}(\hat s_l;f)}{(z-\hat{s_l})^{m}}
\in{\mathbb C}_p[[z,\frac{1}{z-\hat{s_l}}]] \Biggm| \\
&
|a_{m}(\hat s_l;f)|_p\to 0 \  (m\to \infty) 
\text{ for }  0\leqslant l \leqslant d
\Biggr\}. 
\end{split}
\end{equation}
\end{lemma}

\begin{proof}
They follow from the definitions (cf. Notation \ref{basics on rigid}).
\end{proof}

We note that 
$$
A^\dag(  {\bf P}^1\setminus S)\subset
A^{\text{rig}}({\bf P}^1\setminus S) .
$$
The following is one of the most important properties of overconvergent functions.
\begin{lemma}\label{useful example}
Let $f(z)\in A^\dag(  {\bf P}^1\setminus S)$.
Under the above assumption,
there exists a (unique up to constant)
solution $F(z)\in A^\dag(  {\bf P}^1\setminus S)$ of the differential equation
$$
dF(z)=f(z)dz
$$
if and only if the {\it residues} of the differential $1$-form $f(z)dz$,
i.e. $a_1(\hat s_l;f)$ ($1\leqslant l\leqslant d$) are all $0$.
\end{lemma}

\begin{proof}
When $a_1(\hat s_l;f)$ ($1\leqslant l\leqslant d$) are all $0$,
integrations of $f(z)$ in \eqref{description of A1} 
are formally given by the following power series
$$
\sum_{r\geqslant 1}\frac{a_{r-1}(\hat s_0;f)}{r}\cdot z^{r}
+\sum_{l=1}^d\sum_{m>0}\frac{a_{m+1}(\hat s_l;f)}{-m}\cdot\frac{1}{(z-\hat{s_l})^{m}}+
\text{constant}.
$$
Then by replacing the $\lambda$  by $\lambda'$ such that $\lambda<\lambda'<1$
and using $\text{ord}_p(n)=O(\log n/\log p )$,
we get
$$
\frac{|a_{m+1}(\hat s_l;f)|_p}{|m|_p}\cdot \frac{1}{\lambda^{\prime m}}\to 0 \quad  (m\to \infty).
$$
Whence $F(z)$ belongs to  $A^\dag(  {\bf P}^1\setminus S)$.
The \lq if'-part is obtained.
The \lq only if'-part is easy.
\end{proof}

The lemma actually is a consequence of the fact that
$\dim H^1_\dag( {\bf P}^1\setminus S)=d$.

The following function is a main object in this subsection.

\begin{definition}\label{def of pPMPL}
Let  $n_1,\dots,n_r\in \mathbb{N}$ and $\xi_1,\dots,\xi_{r}\in\mathbb{C}_p$ with $|\xi_j|_p\leqslant 1$ ($1\leqslant j\leqslant r$).
Let  $\alpha_1,\dots,\alpha_r\in \mathbb{N}$ with $0<\alpha_j<p$
 ($1\leqslant j\leqslant r$).
The {\bf $p$-adic partial TMPL} 
$\ell^{\equiv (\alpha_1,\dots,\alpha_r),(p)}_{n_1,\dots,n_r}(\xi_1,\dots\xi_{r};z)$ is defined by the following 
$p$-adic power series:
\begin{equation}\label{series expression for partial double}
\ell^{\equiv (\alpha_1,\dots,\alpha_r),(p)}_{n_1,\dots,n_r}(\xi_1,\dots,\xi_{r};z):=
\underset{k_1\equiv \alpha_1,\dots ,k_r\equiv\alpha_r \bmod p
}{\underset{0<k_1<\cdots <k_r}{\sum}}
\frac{\xi_1^{k_1}\cdots\xi_{r}^{k_{r}}}{k_1^{n_1}\cdots k_r^{n_r}}z^{k_r}
\end{equation}
which converges for $z\in ]\bar{0}[$.
\end{definition}

Similarly to \eqref{limit expression}, we have a limit expression of them.

\begin{proposition}
Let  $n_1,\dots,n_r\in \mathbb{N}$,
$\xi_1,\dots,\xi_{r}\in\mathbb{C}_p$ with $|\xi_j|_p\leqslant 1$ ($1\leqslant j\leqslant r$)
and $\alpha_1,\dots,\alpha_r\in \mathbb{N}$ with $0<\alpha_j<p$
($1\leqslant j\leqslant r$).
When $z\in]\bar 0[$, the function $\ell^{\equiv (\alpha_1,\dots,\alpha_r),(p)}_{n_1,\dots,n_r}(\xi_1,\dots,\xi_{r};z)$ is expressed as
\begin{align}\label{limit expression for partial double}
\ell^{\equiv (\alpha_1,\dots,\alpha_r),(p)}_{n_1,\dots,n_r}&
(\xi_1,\dots,\xi_{r};z) 
=\lim_{M\to\infty}
\underset{l_1\equiv \alpha_1,l_1+l_2\equiv\alpha_2,\dots ,l_1+\cdots+l_r\equiv\alpha_r \bmod p}{\underset{0<l_1, \dots,l_{r}<p^M}{\sum}}
\frac{\xi_1^{l_1}\xi_2^{l_1+l_2}\cdots\xi_{r}^{l_1+\cdots+l_{r}}z^{l_1+\cdots+l_r}}
{l_1^{n_1}(l_1+l_2)^{n_2}\cdots (l_1+\cdots+l_r)^{n_r}} \notag \\
&\qquad \qquad \cdot\frac{1}{1-(\xi_1\cdots\xi_{r} z)^{p^M}}
\cdots 
\frac{1}{1-(\xi_{r-1}\xi_r z)^{p^M}}\cdot
\frac{1}{1-(\xi_r z)^{p^M}}.
\end{align}
\end{proposition}

\begin{proof}
It can be proved by direct calculations.
\end{proof}

It is worthy to note that here the condition 
$\alpha_j\neq 0$ ($1\leqslant j\leqslant r$)
is necessary to make the limit convergent.

\begin{proposition}\label{rigid extension}
Let  $n_1,\dots,n_r\in \mathbb{N}$, 
$\xi_1,\dots,\xi_{r}\in\mathbb{C}_p$ with $|\xi_j|_p= 1$ ($1\leqslant j\leqslant r$)
and $\alpha_1,\dots,\alpha_r\in \mathbb{N}$ with $0<\alpha_j<p$
($1\leqslant j\leqslant r$).
Set $S$
as in \eqref{S0}.
Then the function $\ell^{\equiv (\alpha_1,\dots,\alpha_r),(p)}_{n_1,\dots,n_r}(\xi_1,\dots,\xi_{r};z)$
is analytically extended into 
${\bf P}^1({\mathbb C}_p) - ]S[
$
as a rigid analytic function.
Namely, 
$$\ell^{\equiv (\alpha_1,\dots,\alpha_r),(p)}_{n_1,\dots,n_r}(\xi_1,\dots,\xi_{r};z)\in 
A^{\text{rig}}( {\bf P}^1\setminus S).
$$
\end{proposition}

\begin{proof}
%
The relation
\begin{equation}\label{partial ell-ell}
\ell^{\equiv (\alpha_1,\dots,\alpha_r),(p)}_{n_1,\dots,n_r}(\xi_1,\dots,\xi_{r};z)
=\frac{1}{p^r}\sum_{\rho_1^p=\cdots=\rho_r^p=1} \rho_1^{-\alpha_1}\cdots\rho_r^{-\alpha_r}
\ell^{(p)}_{n_1,\dots,n_r}(\rho_1\xi_1,\dots,\rho_{r}\xi_{r}; z)
\end{equation}
holds on $]\bar 0[$.
Since $\ell^{(p)}_{n_1,\dots,n_r}(\rho_1\xi_1,\dots,\rho_{r}\xi_{r}; z)$'s
are all rigid analytic on the above space
by Proposition \ref{rigidness},
the function $\ell^{\equiv (\alpha_1,\dots,\alpha_r),(p)}_{n_1,\dots,n_r}(\xi_1,\dots,\xi_{r};z)$
can be extended there as a rigid analytic function
to keep the equation.
\end{proof}

From now on we will employ the same symbol
$\ell^{\equiv (\alpha_1,\dots,\alpha_r),(p)}_{n_1,\dots,n_r}(\xi_1,\dots,\xi_{r};z)$
to denote its analytic continuation.

The following formulas are necessary to prove our Theorem \ref{rigidness II}.

\begin{lemma}\label{differential equations}
Let  $n_1,\dots,n_r\in \mathbb{N}$, 
$\xi_1,\dots,\xi_{r}\in\mathbb{C}_p$ with $|\xi_j|_p= 1$ ($1\leqslant j\leqslant r$)
and $\alpha_1,\dots,\alpha_r\in \mathbb{N}$ with $0<\alpha_j<p$
($1\leqslant j\leqslant r$).

{\rm (i)}\ For $n_r\neq 1$, 
$$\frac{d}{dz}\ell^{\equiv (\alpha_1,\dots,\alpha_r),(p)}_{n_1,\dots,n_r}(\xi_1,\dots,\xi_{r};z)
=\frac{1}{z}\ell^{\equiv (\alpha_1,\dots,\alpha_r),(p)}_{n_1,\dots,n_{r-1},n_r-1}(\xi_1,\dots,\xi_{r};z).
$$

{\rm (ii)}\ For $n_r=1$ and $r\neq 1$,
$$\frac{d}{dz}\ell^{\equiv (\alpha_1,\dots,\alpha_r),(p)}_{n_1,\dots,n_r}(\xi_1,\dots,\xi_{r};z)=
\begin{cases}
&\frac{\xi_r(\xi_r z)^{\alpha_r-\alpha_{r-1}-1}}{1-(\xi_r z)^p}
\ell^{\equiv (\alpha_1,\dots,\alpha_{r-1}),(p)}_{n_1,\dots,n_{r-1}}(\xi_1,\dots,\xi_{r-2},\xi_{r-1};\xi_r z)  \\
&\qquad\qquad\qquad\qquad\qquad\qquad \text{if}\quad \alpha_r>\alpha_{r-1} , \\
&\frac{\xi_r(\xi_r z)^{\alpha_r-\alpha_{r-1}+p-1}}{1-(\xi_r z)^p}\ell^{\equiv (\alpha_1,\dots,\alpha_{r-1}),(p)}_{n_1,\dots,n_{r-1}}(\xi_1,\dots,\xi_{r-2},\xi_{r-1};\xi_r z)  \\
&\qquad\qquad\qquad\qquad\qquad\qquad \text{if}\quad \alpha_r\leqslant\alpha_{r-1}. \\
\end{cases}
$$

(iii)\ For $n_r=1$ and $r=1$ with $\xi_1=\xi$ and  $\alpha_1=\alpha$,
 $$\frac{d}{dz}\ell^{\equiv \alpha, (p)}_{1}(\xi;z)=
\frac{\xi (\xi z)^{\alpha-1}}{1-(\xi z)^p}.
$$

\end{lemma}

\begin{proof}
They can be proved by direct computations.
\end{proof}

By Lemma \ref{ell at infinity} and \eqref{partial ell-ell},

\begin{remark}\label{easy remark}
For $n_1,\dots,n_r\in \mathbb{N}$, 
$\xi_1,\dots,\xi_{r}\in\mathbb{C}_p$ with $|\xi_j|_p= 1$ ($1\leqslant j\leqslant r$)
and $\alpha_1,\dots,\alpha_r\in \mathbb{N}$ with $0<\alpha_j<p$
($1\leqslant j\leqslant r$),
the following hold:

%
(i)\  $\ell^{\equiv (\alpha_1,\dots,\alpha_r),(p)}_{n_1,\dots,n_r}(\xi_1,\dots,\xi_{r};0)
=\ell^{\equiv (\alpha_1,\dots,\alpha_r),(p)}_{n_1,\dots,n_r}(\xi_1,\dots,\xi_{r};\infty)=0$.

(ii)\  $\ell^{(p)}_{n_1,\dots,n_r}(\xi_1,\dots,\xi_{r};z)=
\sum_{0<\alpha_1,\dots, \alpha_r<p}\ell^{\equiv (\alpha_1,\dots,\alpha_r),(p)}_{n_1,\dots,n_r}(\xi_1,\dots,\xi_{r};z).$
\end{remark}

Next we discuss a new property of our functions.

\begin{theorem}\label{rigidness II}
Let  $n_1,\dots,n_r\in \mathbb{N}$, 
$\xi_1,\dots,\xi_{r}\in\mathbb{C}_p$ with $|\xi_j|_p= 1$ ($1\leqslant j\leqslant r$)
and $\alpha_1,\dots,\alpha_r\in \mathbb{N}$ with $0<\alpha_j<p$
($1\leqslant j\leqslant r$).
Set $S$ as in \eqref{S0}.
The function 
$\ell^{\equiv (\alpha_1,\dots,\alpha_r),(p)}_{n_1,\dots,n_r}(\xi_1,\dots,\xi_{r}; z)$
is an overconvergent function of 
${\bf P}^1\setminus S$.
Namely, 
$$\ell^{\equiv (\alpha_1,\dots,\alpha_r),(p)}_{n_1,\dots,n_r}(\xi_1,\dots,\xi_{r};z)\in
 A^\dag( {\bf P}^1\setminus S
).$$
\end{theorem}

\begin{proof}
It is achieved by induction on weight $n_1+\cdots+n_r$.

(i). Assume that the weight is equal to $1$, i.e. $r=1$ and $n_1=1$.
By 
 changing variable $w(z)=\frac{1}{\xi z-1}$, we see that 
${\bf P}^1 \setminus\{{\bar \xi^{-1}}\}$
is identified with ${\bf P}^1 \setminus\{{\bar \infty}\}$.
By direct calculations,
it can be checked that  
$\frac{(\xi z)^{\alpha-1}}{1-(\xi z)^p}\cdot\frac{dz}{dw}$ as a function on $w$
belongs to $A^\dag({\bf P}^1\setminus\{{\bar \infty}\})$.
(We note that it  also follows from the fact that  that it is a rational function on $w$
whose poles are all of the form $w=\frac{1}{\zeta_p-1}$ with $\zeta_p\in\mu_p$.)
So by Lemma \ref{useful example}, 
there exists a unique (modulo constant) function
$F(w)$ in $A^\dag({\bf P}^1
\setminus\{{\bar \infty}\})$
such that 
$$
\frac{dF}{dw}=\frac{\xi (\xi z)^{\alpha-1}}{1-(\xi z)^p}\cdot\frac{dz}{dw}.
$$
Therefore, there exists a unique function $F(z)$ in $A^\dag({\bf P}^1
\setminus\{{\bar \xi^{-1}}\})$
such that 
$$
F(0)=0 \quad \text{and} \quad
\frac{dF(z)}{dz}=\frac{\xi (\xi z)^{\alpha-1}}{1-(\xi z)^p}.
$$

By Proposition \ref{rigid extension} and Lemma \ref{differential equations} (iii),
 $\ell^{\equiv \alpha,(p)}_{1}(\xi;z)$ is also a unique function in
$A^{\text{rig}}( {\bf P}^1
\setminus \{\bar\xi^{-1}\} )$
satisfying the above properties and
$F(z)|_{{\bf P}^1({\mathbb C}_p) -]{\bar \xi^{-1}}[}$ belongs to
$A^{\text{rig}}( {\bf P}^1 
 \setminus \{ {\bar \xi^{-1}}\} )$.
Thus  we have
$$
F(z)|_{{\bf P}^1({\mathbb C}_p) -]{\bar \xi^{-1}}[}\equiv\ell^{\equiv \alpha,(p)}_{1}(\xi;z).
$$
So by the coincidence principle of rigid analytic functions
we can say that  $\ell^{\equiv \alpha,(p)}_{1}(\xi;z)$ can be uniquely
extended into a rigid analytic space bigger than
${\bf P}^1({\mathbb C}_p) -]{\bar \xi^{-1}}[ $
by $F(z)\in A^\dag({\bf P}^1
\setminus\{{\bar \xi^{-1}}\} )$.

(ii). Assume that  $n_r \neq 1$.
We put
\begin{equation*}
S_\infty=S\cup \{\bar\infty\}
\qquad \text{ and } \qquad
S_{\infty,0}=S\cup \{\bar\infty\}\cup\{\bar 0\}
\end{equation*}
and take a lift $\{\hat{s_0}, \hat{s_1}, \dots, \hat{s_d}\}$  
of $S_{\infty,0}$  with
\begin{equation*}
\hat{s_0}=\infty \qquad \text{ and } \qquad
\hat{s_1}=0.
\end{equation*}

By our assumption
$$
\ell^{\equiv (\alpha_1,\dots,\alpha_r),(p)}_{n_1,\dots,n_{r-1},n_r-1}(\xi_1,\dots,\xi_{r-1},\xi_r;z)\in A^\dag( {\bf P}^1\setminus S
)
$$
and by $\frac{dz}{z}\in 
\Omega^{\dag,1}({\bf P}^1
\setminus\{\overline{\infty},\overline{0}\})$,
we have 
\begin{equation*}
\ell^{\equiv (\alpha_1,\dots,\alpha_r),(p)}_{n_1,\dots,n_{r-1},n_r-1}(\xi_1,\dots,\xi_{r-1},\xi_r;z)\frac{dz}{z}\in
\Omega^{\dag,1}( {\bf P}^1\setminus S_{\infty,0}
 ).
\end{equation*}
Here $\Omega^{\dag,1}( {\bf P}^1\setminus \tilde{S})$ 
($\tilde{S}$: a finite subset of $\textbf{P}^{1}(\overline{\mathbb{F}}_p)$)
denotes the space of rigid differential $1$-forms there, i.e.
$\Omega^{\dag,1}( {\bf P}^1\setminus \tilde{S})=
A^{\dag,1}( {\bf P}^1\setminus \tilde{S})dz$.

Put
\begin{equation}\label{put}
f(z):=\frac{1}{z}\ell^{\equiv (\alpha_1,\dots,\alpha_r),(p)}_{n_1,\dots,n_{r-1},n_r-1}(\xi_1,\dots,\xi_{r-1},\xi_r; z)\in
A^{\dag}( {\bf P}^1\setminus S_{\infty,0}
).
\end{equation}

Since $\ell^{\equiv (\alpha_1,\dots,\alpha_r),(p)}_{n_1,\dots,n_{r-1},n_r}(\xi_1,\dots,\xi_{r};z)$
belongs to
$
A^{\text{rig}}( {\bf P}^1\setminus S
) 
\Bigl(\subset
A^{\text{rig}}( {\bf P}^1\setminus S_{\infty,0}
 ) 
\Bigr)
$
by Proposition \ref{rigid extension}
and it satisfies the differential equation in Lemma \ref{differential equations}.(i),
i.e. its differential  is equal to $f(z)$,
we have particularly 
\begin{equation}\label{expansion at 0}
a_m(\hat s_1;f)=0 \qquad (m>0)
\end{equation}
(recall $\hat{s_1}=0$) and
\begin{equation}\label{other residues}
a_1(\hat s_l;f)=0 \qquad (2\leqslant l\leqslant d)
\end{equation}
by \eqref{description of A1} and \eqref{description of A2}.

By \eqref{put} and \eqref{expansion at 0},
$$
f(z)\in A^{\dag}( {\bf P}^1\setminus S_\infty
).
$$

By \eqref{other residues} and Lemma \ref{useful example},
there exists a unique function
$F(z)$ in 
$
A^{\dag}( {\bf P}^1\setminus S_\infty
)
$,
i.e. a function $F(z)$ which is rigid analytic on an affinoid  $V$ of
$$
{\bf P}^1({\mathbb C}_p) - ]S_\infty[  \quad = \quad
{\bf P}^1({\mathbb C}_p) - ]\bar{\infty},S[
$$
such that 
\begin{equation}\label{differential property}
F(0)=0 \quad \text{and} \quad
dF(z)=f(z)dz.
\end{equation}

Since $\ell^{\equiv (\alpha_1,\dots,\alpha_r),(p)}_{n_1,\dots,n_{r-1},n_r}(\xi_1,\dots,\xi_{r}; z)$
is also a unique function in
$
A^{\text{rig}}( {\bf P}^1\setminus S
) 
$
satisfying \eqref{differential property},
the restrictions of both $F(z)$ and
$\ell^{\equiv (\alpha_1,\dots,\alpha_r),(p)}_{n_1,\dots,n_{r-1},n_r}(\xi_1,\dots,\xi_{r}; z)$
into the subspace
${\bf P}^1({\mathbb C}_p) -]S_\infty[
$ must coincide, i.e.
$$
F(z)\Bigm|_{{\bf P}^1({\mathbb C}_p) - ]S_\infty[
}
\equiv
\ell^{\equiv (\alpha_1,\dots,\alpha_r),(p)}_{n_1,\dots,n_{r-1},n_r}(\xi_1,\dots,\xi_{r};z)\Bigm|_{{\bf P}^1({\mathbb C}_p) - ]S_\infty[
}.
$$

Hence by the coincidence principle of rigid analytic functions,
there is a rigid analytic function $G(z)$ on the union of $V$ and
${\bf P}^1({\mathbb C}_p) - ]S[$
whose restriction to $V$ is equal to $F(z)$
and whose restriction to 
${\bf P}^1({\mathbb C}_p) -]S[
$
is equal to
$\ell^{\equiv (\alpha_1,\dots,\alpha_r),(p)}_{n_1,\dots,n_{r-1},n_r}(\xi_1,\dots,\xi_{r}; z)$.
So we can say that  
$$\ell^{\equiv (\alpha_1,\dots,\alpha_r),(p)}_{n_1,\dots,n_{r-1},n_r}(\xi_1,\dots,\xi_{r};z)\in
A^{\text{rig}}( {\bf P}^1\setminus S
) 
$$
can be rigid analytically
extended into a bigger rigid analytic space
by $G(z)$.
Namely, 
$$\ell^{\equiv (\alpha_1,\dots,\alpha_r),(p)}_{n_1,\dots,n_r}(\xi_1,\dots,\xi_{r};z)\in
 A^\dag( {\bf P}^1\setminus S
).$$


(iii). Assume that $n_r=1$ ($r\geqslant 2$).
Put
$$
\beta(z):=
\begin{cases}
\frac{\xi_r (\xi_r z)^{\alpha_r-\alpha_{r-1}-1}}{1-(\xi_r z)^p} &\text{if}\quad \alpha_r>\alpha_{r-1} ,\\
\frac{\xi_r (\xi_r z)^{\alpha_r-\alpha_{r-1}+p-1}}{1-(\xi_r z)^p}
&\text{if}\quad \alpha_r\leqslant\alpha_{r-1} .\\
\end{cases}
$$
By our assumption
$$
\ell^{\equiv (\alpha_1,\dots,\alpha_{r-1}),(p)}_{n_1,\dots,n_{r-1}}(\xi_1,\dots,\xi_{r-2},\xi_{r-1};\xi_r z) 
\in A^\dag( {\bf P}^1 
\setminus\{
\overline{\xi_{r}^{-1}},\dots,\overline{(\xi_1\cdots\xi_{r})^{-1}}\} )
$$
and by
$\beta(z)dz\in
\Omega^{\dag,1}({\bf P}^1
\setminus\{\overline{\infty},\overline{\xi_r^{-1}}\})$,
we have
\begin{equation*}
\ell^{\equiv (\alpha_1,\dots,\alpha_{r-1}),(p)}_{n_1,\dots,n_{r-1}}(\xi_1,\dots,\xi_{r-2},\xi_{r-1};\xi_r z) \cdot \beta(z)dz
\in\Omega^{\dag,1}( {\bf P}^1 \setminus S_\infty
).
\end{equation*}
Put
$$
f(z):=\ell^{\equiv (\alpha_1,\dots,\alpha_{r-1}),(p)}_{n_1,\dots,n_{r-1}}(\xi_1,\dots,\xi_{r-2},\xi_{r-1};\xi_r z) \cdot \beta(z)
\in
A^{\dag}( {\bf P}^1 \setminus S_\infty
).
$$
Then  it follows from the same arguments as the ones in (ii) that
$$\ell^{\equiv (\alpha_1,\dots,\alpha_r),(p)}_{n_1,\dots,n_r}(\xi_1,\dots,\xi_{r}; z)\in
 A^\dag( {\bf P}^1\setminus S
).$$
\end{proof}


We saw in Proposition \ref{rigidness} that
$\ell^{(p)}_{n_1,\dots,n_r}(\xi_1,\dots,\xi_{r}; z)$
is a rigid analytic function, namely
$
\ell^{(p)}_{n_1,\dots,n_r}(\xi_1,\dots,\xi_{r}; z)
\in A^{\text{rig}}( {\bf P}^1\setminus S
).
$
But actually we can say more:

\begin{theorem}\label{rigidness III}
Let  $n_1,\dots,n_r\in \mathbb{N}$, 
$\xi_1,\dots,\xi_{r}\in\mathbb{C}_p$ with $|\xi_j|_p= 1$ ($1\leqslant j\leqslant r$).
Set $S$ as in \eqref{S0}.
The function 
$\ell^{(p)}_{n_1,\dots,n_r}(\xi_1,\dots,\xi_{r}; z)$
is an overconvergent function of 
${\bf P}^1\setminus S
$.
Namely, 
$$\ell^{(p)}_{n_1,\dots,n_r}(\xi_1,\dots,\xi_{r}; z)\in
 A^\dag( {\bf P}^1\setminus S
).$$

%
\end{theorem}

\begin{proof}
By our previous proposition, we have
$$\ell^{\equiv (\alpha_1,\dots,\alpha_r),(p)}_{n_1,\dots,n_r}(\xi_1,\dots,\xi_{r}; z)\in
 A^\dag( {\bf P}^1\setminus S
).$$
Then by Remark \ref{easy remark} (ii), we have
$\ell^{(p)}_{n_1,\dots,n_r}(\xi_1,\dots,\xi_{r};z)\in A^\dag( {\bf P}^1\setminus S)$
because $A^\dag( {\bf P}^1\setminus S)$ forms an algebra.
\end{proof}

\begin{example}
Especially when $r=1$
we have
$$\ell^{(p)}_n(1;z)\in A^\dag( {\bf P}^1\setminus \{\bar 1\}).$$
Actually in \cite[Proposition 6.2]{C}, Coleman showed that
$$
\ell^{(p)}_n(1;z)\in A^{\text{rig}}(\tilde X)
$$
with
$\tilde X=\left\{z\in {\bf P}^1({\mathbb C}_p)\bigm| |z-1|_p >p^{\frac{-1}{p-1}}\right\}$.
\end{example}

\begin{remark}
In \cite[Definition 2.13]{F2}, the first named author introduced the overconvergent
$p$-adic MPL $Li^\dag_{n_1,\dots,n_r}(z)$.
Asking a relationship between 
this $Li^\dag_{n_1,\dots,n_r}(z)$ and our $\ell^{(p)}_{n_1,\dots,n_r}(\xi_1,\dots,\xi_{r};z)$ would be one of the questions which we are interested in.

\end{remark}

\subsection{$p$-adic twisted multiple polylogarithms}\label{sec-5-5}
In \cite{Fu1, Y} 
$p$-adic TMPL's (see Definition \ref{Def-TMPL}) were introduced 
as a multiple generalization
of Coleman's $p$-adic polylogarithm \cite{C} and
$p$-adic multiple $L$-value is introduced as its special value at $1$. 
The aim of 
this subsection is to clarify a relationship between $p$-adic rigid TMPL's
(see Definition \ref{def of pMMPL})
and $p$-adic TMPL's  in Theorem \ref{ell-Li theorem}
and to establish a relationship between 
special values at positive integers of our $p$-adic multiple $L$-functions
(see Definition \ref{Def-pMLF})
and the $p$-adic twisted multiple $L$-values,
the special values of $p$-adic TMPL's at $1$,
in Theorem \ref{L-Li theorem}.
The theorem extends the previous result
\eqref{Coleman L-Li} of Coleman in depth $1$ case shown in \cite{C}.

First we recall Coleman's $p$-adic iterated integration theory \cite{C}
in our particular case.
For other nice expositions of his theory,
see \cite[Section 5]{B}, \cite[Subsection 2.2.1]{Br} and \cite[Subsection 2.1]{Fu1}.
The integration here is different from the integration using a certain $p$-adic measure
which is explained in Subsection \ref{sec-3-1}.

\begin{notation}
Fix $\varpi\in\mathbb C_p$.
The {\it $p$-adic logarithm} $\log^\varpi$ associated to 
the {\it branch} $\varpi$ means
the locally rigid analytic group homomorphism
$\mathbb C_p^\times\to\mathbb C_p$
with the usual Taylor expansion  for $\log$ on $]\bar 1[=1+{\frak M}_{\mathbb C_p}$.
It is uniquely characterized by $log^\varpi(p)=\varpi$
because ${\mathbb C}_p^\times\simeq 
p^{\mathbb Q}\times \mu_{\infty}\times (1+{\frak M}_{\mathbb C_p})$.
We call this $\varpi\in\mathbb C_p$ the {\it branch parameter}
of the $p$-adic logarithm. 

Let
$S=\{s_0,\dots,s_d\}$ (all $s_i$'s are distinct)
be a finite subset of $\textbf{P}^{1}(\overline{\mathbb{F}}_p)$.
Define
\[
A^\varpi_{\mathrm{loc}}({\bf P}^1\setminus S)
:=\prod_{x\in {\bf P}^1({\overline{\mathbb{F}}_p}) }A^\varpi_{log}(U(x)),\qquad
\Omega^\varpi_{\mathrm{loc}}({\bf P}^1\setminus S)
:=\prod_{x\in  {\bf P}^1({\overline{\mathbb{F}}_p}) }\Omega^\varpi_{log}(U(x))
\]
where
\begin{align*}
&A^\varpi_{log}(U(x)):=
\begin{cases}
A^\text{rig}(]x[) &  x\not\in S, \\
\underset{\lambda\to 1^-}{\lim}
A^\text{rig}\Bigl(]x[ \ \cap U_\lambda\Bigr)\Bigl[log^\varpi(z_i)\Bigr] &
x=s_i \quad (0\leqslant i \leqslant d),
\end{cases}\\
&\Omega^\varpi_{log}(U(x)):=
\begin{cases}
A^{\text{rig}}(]x[)\cdot dz_x
&  x\not\in S, \\
A^\varpi_{log}(U(x))dz_i &
x=s_i \quad (0\leqslant i \leqslant d).
\end{cases}\\
\end{align*}
Here $U_\lambda$ is the affinoid in \eqref{removing all closed discs},
$z_i$ is a local parameter \eqref{local parameter}
and $z_x$ is a local parameter of $]x[$. 
We note that $log^\varpi(z_i)$ is a locally analytic function defined on
$]s_i[-z_i^{-1}(0)$
whose differential is $\frac{dz_i}{z_i}$ and it is transcendental over 
$A^\text{rig}\Bigl(]s_i[ \ \cap U_\lambda\Bigr)$
and
$$
\underset{\lambda\to 1^-}{\lim}
A^\text{rig}\Bigl(]s_i[ \ \cap U_\lambda\Bigr)
\cong\Bigl\{f(z_i)=\sum\limits_{n=-\infty}^{\infty}a_nz_i^n  \ (a_n\in\mathbb C_p)
\text{ converging on }\lambda<|z_i|_p<1 \text{ for some } 0\leqslant \lambda<1\Bigr\}
$$
(see \cite{Be1}).
We remark that these definitions of $A^\varpi_{log}(U(x))$ and $\Omega^\varpi_{log}(U(x))$
are independent of any choice of local parameters 
modulo standard isomorphisms.
By taking a component-wise derivative, we obtain a $\mathbb C_p$-linear map
$d:A^\varpi_{\mathrm{loc}}({\bf P}^1\setminus S)\to
\Omega^\varpi_{\mathrm{loc}}({\bf P}^1\setminus S)$.
We may regard $A^\dag({\bf P}^1\setminus S)$
and $\Omega^{\dag,1}({\bf P}^1\setminus S)$
in Notation \ref{overconvergent functions and associated cohomologies}
to be a subspace of $A^\varpi_{\mathrm{loc}}({\bf P}^1\setminus S)$ 
and $\Omega^\varpi_{\mathrm{loc}}({\bf P}^1\setminus S)$ 
respectively.

By the work of Coleman \cite{C},  we have an 
$A^\dag({\bf P}^1\setminus S)$-subalgebra 
$$
A^\varpi_{\mathrm{Col}}({\bf P}^1\setminus S)$$
of $A^\varpi_{\mathrm{loc}}({\bf P}^1\setminus S)$,
which we call {\it the ring of Coleman functions} 
attached to a branch parameter $\varpi\in\mathbb C_p$,
and a $\mathbb C_p$-linear map
$$
\int_{(\varpi)}:A^\varpi_{\mathrm{Col}}({\bf P}^1\setminus S)\underset
{A^\dag({\bf P}^1\setminus S)}{\otimes}
\Omega^{\dag,1}({\bf P}^1\setminus S)
\to A^\varpi_{\mathrm{Col}}({\bf P}^1\setminus S)\Bigm/ \mathbb C_p\cdot 1
$$
satisfying 
$d\bigm|_{A^\varpi_{\mathrm{Col}}({\bf P}^1\setminus S)}\circ\int_{(\varpi)}=
id_{A^\varpi_{\mathrm{Col}}({\bf P}^1\setminus S)\otimes
\Omega^{\dag,1}({\bf P}^1\setminus S)}$,
which we call {\it the $p$-adic (Coleman) integration}
attached to a branch parameter $\varpi\in\mathbb C_p$.
We often drop  the subscript ${}_{(\varpi)}$.

Since
$A^\varpi_{\mathrm{Col}}({\bf P}^1\setminus S)$ is a subspace
of $A^\varpi_{\mathrm{loc}}({\bf P}^1\setminus S)$,
each element $f$ of
$A^\varpi_{\mathrm{Col}}({\bf P}^1\setminus S)$ can be seen as a map
$f:U\to \mathbb C_p$
by  $f|_{U\cap ]x[}\in A^\varpi_{\mathrm{log}}(U_x)$ for each residue $]x[$,
where $U$ is an affinoid bigger than
${\bf P}^1({\mathbb C}_p)- ]S[$.
This is how we regard $f$ as a function.
\end{notation}

Here we recall two important properties of Coleman functions below. 
%
%
%

\begin{proposition}[Branch Independency Principle {\cite[Proposition 2.3]{Fu1}}]
\label{branch independency principle}
For $\varpi_1,\varpi_2\in\mathbb C_p$
define the isomorphisms
$$
\iota_{\varpi_1,\varpi_2}:A^{\varpi_1}_{\mathrm{loc}}({\bf P}^1\setminus S)
\overset{\sim}{\to}A^{\varpi_2}_{\mathrm{loc}}({\bf P}^1\setminus S) \quad \text{ and } \quad
\tau_{\varpi_1,\varpi_2}:\Omega^{\varpi_1}_{\mathrm{loc}}({\bf P}^1\setminus S)\overset{\sim}{\to}\Omega^{\varpi_2}_{\mathrm{loc}}({\bf P}^1\setminus S)
$$
which is obtained
by replacing each $log^{\varpi_1}(z_{s_i})$ by $log^{\varpi_2}(z_{s_i})$
for $1\leqslant i\leqslant d$. Then
$$\iota_{\varpi_1,\varpi_2}(A^{\varpi_1}_{\mathrm{Col}}({\bf P}^1\setminus S))=A^{\varpi_2}_{\mathrm{Col}}({\bf P}^1\setminus S),$$
$$\tau_{\varpi_1,\varpi_2}(A^{\varpi_1}_{\mathrm{Col}}({\bf P}^1\setminus S)\otimes\Omega^{\dag,1}({\bf P}^1\setminus S))
=A^{\varpi_2}_{\mathrm{Col}}({\bf P}^1\setminus S)\otimes\Omega^{\dag,1}({\bf P}^1\setminus S)$$ 
and
$$\iota_{\varpi_1,\varpi_2}\circ\int_{(\varpi_1)}=\int_{(\varpi_2)}\circ\tau_{\varpi_1,\varpi_2} \mod \mathbb C_p\cdot 1.$$
Namely the following diagram is commutative:
\[
\begin{CD}
A^{\varpi_1}_{\mathrm{Col}}({\bf P}^1\setminus S)\underset{A^\dag({\bf P}^1\setminus S)}{\otimes}\Omega^{\dag,1}({\bf P}^1\setminus S)
@>{\tau_{\varpi_1,\varpi_2}}>>
A^{\varpi_2}_{\mathrm{Col}}({\bf P}^1\setminus S)\underset{A^\dag({\bf P}^1\setminus S)}{\otimes}\Omega^{\dag,1}({\bf P}^1\setminus S) \\
@V{\int_{(\varpi_1)}}VV
@VV{\int_{(\varpi_2)}}V \\
A^{\varpi_1}_{\mathrm{Col}}({\bf P}^1\setminus S)\Bigm/ \mathbb C_p\cdot 1
@>{\iota_{\varpi_1,\varpi_2}}>>
A^{\varpi_2}_{\mathrm{Col}}({\bf P}^1\setminus S)\Bigm/ \mathbb C_p\cdot 1. \\
\end{CD}
\]
\end{proposition}
%
%

\begin{proposition}[Coincidence Principle {\cite[Chapter IV]{C}}]\label{coincidence principle}
Put $\varpi\in\mathbb C_p$.
Let $f\in A^\varpi_{\mathrm{Col}}({\bf P}^1\setminus S)$.
Suppose that $f|_U\equiv 0$ for 
an admissible open subset $U$ (cf.\cite{B,C})
of ${\bf P}^1(\mathbb C_p)$
(that is,
$U$ is of the form 
$
\left\{z\in {\bf P}^1({\mathbb C}_p)\bigm| |z-\alpha_i|_p> \rs_i \ (i=1,\dots, n),
|z|_p< 1/\rs_0
\right\}
$
for some $\alpha_1, \dots, \alpha_n\in {\mathbb C}_p$ and $\rs_1,\dots,\rs_n\in{\mathbb Q}_{>0}$
and $\rs_0\in{\mathbb Q}_{\geqslant 0}$).
Then $f=0$.
\end{proposition}

It follows from this proposition that a locally constant 
Coleman function is globally constant.

%

\begin{notation}\label{Coleman-notation}
Fix a branch $\varpi\in\mathbb C_p$ and let
$\omega\in A^\varpi_{\mathrm{Col}}({\bf P}^1\setminus S)\underset{A^\dag({\bf P}^1\setminus S)}{\otimes}\Omega^{\dag,1}({\bf P}^1\setminus S)$.
Then by Coleman's integration theory, there exists
(uniquely modulo constant) a Coleman function 
$F_\omega\in A^\varpi_{\mathrm{Col}}({\bf P}^1\setminus S)$ such that
$\int\omega\equiv F_\omega$ (modulo constant).
For $x,y\in]{\bf P}^1\setminus S[$, 
we define 
$$\int_x^y\omega:=
F_\omega(y)-F_\omega(x).$$
It is clear that $\int_x^y\omega$ does not depend on any choice
of $F_\omega$ (although it may depend on the choice of a branch $\varpi\in\mathbb C_p$).
If both $F_\omega(x)$ and $F_\omega(y)$ make sense
when $x$ or $y$ belongs to  $]S_0[$,
we also denote $F_\omega(y)-F_\omega(x)$ by $\int_x^y\omega$.
By letting $x$ fixed and $y$ vary, we regard $\int_x^y\omega$ as the Coleman function
which is characterized by $dF_\omega=\omega$ and $F_\omega(x)=0$.
We note that
\begin{equation}
\iota_{\varpi_1,\varpi_2}\left(\int_{(\varpi_1),x}^y\omega\right)
=\int_{(\varpi_2),x}^y\tau_{\varpi_1,\varpi_2} (\omega).
\end{equation}
\end{notation}

We will apply Coleman's theory to the function defined below,
which is the main object in this subsection.

\begin{definition}\label{Def-TMPL}
Let $n_1,\dots,n_r\in \mathbb{N}$,
$\xi_1,\dots,\xi_{r}\in\mathbb{C}_p$
with $|\xi_j|_p\leqslant 1$ ($1\leqslant j\leqslant r$).
The {\bf $p$-adic TMPL} 
$Li^{(p)}_{n_1,\dots,n_r}(\xi_1,\dots,\xi_r; z)$ is defined by the following 
$p$-adic power series:
\begin{equation}\label{series expression}
Li^{(p)}_{n_1,\dots,n_r}(\xi_1,\dots,\xi_r ;z):=
{\underset{0<k_1<\cdots<k_{r}}{\sum}}
\frac{\xi_1^{k_1}\cdots\xi_{r}^{k_{r}} z^{k_r}}{k_1^{n_1}\cdots k_r^{n_r}}
\end{equation}
which converges for $z\in ]\bar{0}[$ 
by $|\xi_j|_p\leqslant 1$ for $1\leqslant j\leqslant r$.
\end{definition}

\begin{remark} \ 
\begin{enumerate}
\item We note that when $r=1$ and $\xi=1$, $Li^{(p)}_{n}(1;z)$
is equal to Coleman's $p$-adic polylogarithm
$\ell_n(z)$ in \cite[p.\,195]{C}.
\item
The special case when $\xi_1=\cdots=\xi_r=1$ is investigated
by the first-named author in \cite{F2} where
$Li^{(p)}_{n_1,\dots,n_r}(1,\dots,1;z)$ is
introduced as 
$p$-adic multiple polylogarithm. 
\item
Yamashita \cite{Y} treats the case when all $\xi_j$ are roots of unity
whose orders are prime to $p$.
\end{enumerate}
\end{remark}

The following can be proved by  direct computations.

\begin{lemma}\label{differential equations fo Li}
Let $n_1,\dots,n_r\in \mathbb{N}$,
$\xi_1,\dots,\xi_{r}\in\mathbb{C}_p$
with $|\xi_j|_p\leqslant 1$ ($1\leqslant j\leqslant r$).
\begin{enumerate}[{\rm (i)}]
\item $\displaystyle{
\frac{d}{dz}Li^{(p)}_{1}(\xi_1;z)=\frac{\xi_1}{1-\xi_1z}.}$
\item $\displaystyle{
\frac{d}{dz}Li^{(p)}_{n_1,\dots,n_r}(\xi_1,\dots,\xi_r ;z)=
\begin{cases}
\frac{1}{z}Li^{(p)}_{n_1,\dots,n_{r-1},n_r-1}(\xi_1,\dots,\xi_r;z)
&\text{if}\quad n_r\neq 1 ,\\
\frac{\xi_r}{1-\xi_r z} Li^{(p)}_{n_1,\dots,n_{r-1}}(\xi_1,\dots,\xi_{r-2},\xi_{r-1}\xi_r;z)&\text{if}\quad n_r=1. 
\end{cases}}$
\end{enumerate}
\end{lemma}

The  definition below is suggested by
the differential equations above.

\begin{theorem-definition}\label{Coleman function theorem}
Fix a branch of the $p$-adic logarithm by $\varpi\in\mathbb C_p$.
Let $n_1,\dots,n_r\in \mathbb{N}$,
$\xi_1,\dots,\xi_{r}\in\mathbb{C}_p$
with $|\xi_j|_p\leqslant 1$ ($1\leqslant j\leqslant r$).
Put 
$$
S_r:=\{\bar 0,\bar\infty, \overline{(\xi_r)^{-1}},\overline{(\xi_{r-1}\xi_r)^{-1}},\dots,\overline{(\xi_1\cdots\xi_r)^{-1}}\}
\subset  {\bf{P}}^{1}(\overline{\mathbb{F}}_p).
$$
Then we have the Coleman function attached to  $\varpi\in\mathbb C_p$
$$
Li^{(p),\varpi}_{n_1,\dots,n_r}(\xi_1,\dots,\xi_r; z)\in
A^\varpi_{\text{Col}}
({\bf P}^1\setminus S_r
)
$$
which is constructed
by the following iterated integrals:
\begin{align}
&Li^{(p),\varpi}_{1}(\xi_1;z)=-\log^\varpi (1-\xi_1z)=\int_0^z\frac{\xi_1}{1-\xi_1t}dt, 
\label{intLi1}
\\
&Li^{(p),\varpi}_{n_1,\dots,n_r}(\xi_1,\dots,\xi_r; z)=
\begin{cases}
\int_0^zLi^{(p),\varpi}_{n_1,\dots,n_{r-1},n_r-1}(\xi_1,\dots,\xi_r; t)\frac{dt}{t} &\text{if}\quad n_r\neq 1 ,\\
\int_0^zLi^{(p),\varpi}_{n_1,\dots,n_{r-1}}(\xi_1,\dots,\xi_{r-2},\xi_{r-1}\xi_r; t)\frac{\xi_r dt}{1-\xi_r t} &\text{if}\quad n_r=1. 
\end{cases}
\label{intLi2} 
\end{align}
\end{theorem-definition}

\begin{proof}
It is achieved by the induction on weight $w:=n_1+\cdots+ n_r$.

When $w=1$,  the integration starting from $0$ makes sense
because the differential form $\frac{\xi_1}{1-\xi_1t}dt$  has no pole at $t=0$.
Hence we have \eqref{intLi1}.

When $w>1$ and $n_r>1$, by our induction assumption, 
$Li^{(p),\varpi}_{n_1,\dots,n_r-1}(\xi_1,\dots,\xi_r; 0)$ 
is equal to $0$ because it is an integration from $0$ to $0$.
So $Li^{(p),\varpi}_{n_1,\dots,n_r-1}(\xi_1,\dots,\xi_r; t)$
has a zero of order at least $1$.
So the integrand of the right-hand side of \eqref{intLi2} has no pole at $t=0$.
The integration \eqref{intLi2} starting from $0$ makes sense.
The case when $n_r=1$ can be proved in a same  (or an easier) way.
%
\end{proof}

\begin{remark}\label{S_r-0}
It is easy to see that
$Li^{(p),\varpi}_{n_1,\dots,n_r}(\xi_1,\dots,\xi_r ;z)
=Li^{(p)}_{n_1,\dots,n_r}(\xi_1,\dots,\xi_r; z)$
when $|z|_p<1$ for all branch $\varpi\in\mathbb C_p$.
Thus the Coleman function 
$Li^{(p),\varpi}_{n_1,\dots,n_r}(\xi_1,\dots,\xi_r; z)
\in A^\varpi_{\text{Col}}
({\bf P}^1\setminus S_r
)$
is defined on an affinoid bigger than
${\bf P}^1({\mathbb C}_p) - ]{S_r}\setminus\{\overline{0}\} [.$
\end{remark}

\begin{proposition}\label{branch independent region}
Fix a branch  $\varpi\in\mathbb C_p$.
Let $n_1,\dots,n_r\in \mathbb{N}$,
$\xi_1,\dots,\xi_{r}\in\mathbb{C}_p$
with $|\xi_j|_p\leqslant 1$ ($1\leqslant j\leqslant r$).
Then the restriction of the $p$-adic TMPL
$Li^{(p),\varpi}_{n_1,\dots,n_r}(\xi_1,\dots,\xi_r ;z)$
into
${\bf P}^1({\mathbb C}_p) - ]{S_r}\setminus\{\overline{0}\} [$
does not depend on any choice of  the branch $\varpi\in\mathbb{C}_p$.
\end{proposition}

\begin{proof}
It is achieved by the induction on weight $w:=n_1+\cdots+ n_r$.
Take $z\in {\bf P}^1({\mathbb C}_p) - ]{S_r}\setminus\{\overline{0}\} [$.

When $w=1$,  it is easy to see that
$Li^{(p),\varpi}_{1}(\xi_1;z)=-\log^\varpi (1-\xi_1z)$ is free from
any choice of  the branch $\varpi\in\mathbb{C}_p$.

Consider the case when $w>1$ and $n_r>1$.
Due to the existence of a pole of $\frac{dt}{t}$ at $t=0$,
the integration \eqref{intLi2} might have a \lq log-term' on  its restriction into $]\bar 0[$.
However it is easy to see that the restriction of 
$Li^{(p),\varpi}_{n_1,\dots,n_r}(\xi_1,\dots,\xi_r; t)$
into $]\bar 0[$ has no log-term
because it is given by the series expansion \eqref{series expression}
by Remark \ref{S_r-0}.
By our induction assumption,
the restriction of
$Li^{(p),\varpi}_{n_1,\dots,n_r-1}(\xi_1,\dots,\xi_r; t)$
into ${\bf P}^1({\mathbb C}_p) - ]{S_r}\setminus\{\overline{0}\} [$
is independent of any choice of the branch $\varpi\in\mathbb{C}_p$.
Therefore $Li^{(p),\varpi}_{n_1,\dots,n_r}(\xi_1,\dots,\xi_r; t)$
has no log-term and does not depend on any choice of the branch.

The above proof also works for the case when $w>1$ and $n_r=1$.
It is noted that the branch independency follows since
two poles of $\frac{\xi_r dt}{1-\xi_r t} $ are outside of the region
${\bf P}^1({\mathbb C}_p) - ]{S_r}\setminus\{\overline{0}\} [$.
\end{proof}

The $p$-adic rigid TMPL
$\ell^{(p)}_{n_1,\dots,n_r}(\xi_1,\dots,\xi_{r}; z)$
in Subsection \ref{sec-5-3}
is described by our $p$-adic TMPL
$Li^{(p),\varpi}_{n_1,\dots,n_r}(\xi_1,\dots,\xi_r ;z)$ as follows:

\begin{theorem}\label{ell-Li theorem}
Fix a branch  $\varpi\in\mathbb C_p$.
Let $n_1,\dots,n_r\in \mathbb{N}$,
$\xi_1,\dots,\xi_{r}\in\mathbb{C}_p$
with $|\xi_j|_p= 1$ ($1\leqslant j\leqslant r$).
The equality
\begin{align}
\ell&^{(p)}_{n_1,\dots,n_r}(\xi_1,\dots,\xi_r ;z)  \notag \\
&=\frac{1}{p^r}\sum_{0<\alpha_1,\dots,\alpha_r<p}\sum_{\rho_1^p=\cdots=\rho_r^p=1}
\rho_1^{-\alpha_1}\cdots\rho_r^{-\alpha_r}
Li^{(p),\varpi}_{n_1,\dots,n_r}(\rho_1\xi_1,\dots,\rho_r\xi_r ;z) 
\label{ell-Li formula} 
\end{align}
holds for $z\in{\bf P}^1({\mathbb C}_p) - ]{S_r}\setminus\{\overline{0}\} [.$
\end{theorem}

\begin{proof}
By the power series expansion
\eqref{series expression for partial double}
and \eqref{series expression} and Remark \ref{easy remark}.(ii),
it is easy to see that the equality holds on $]\bar 0[$.
By Theorem \ref{rigidness III},
the left-hand side belongs to 
$A^\dag({\bf P}^1
\setminus{S_r})$
($\subset A^\varpi_{\text{Col}}({\bf P}^1 \setminus{S_r})$)
and by Theorem-Definition \ref{Coleman function theorem},
the right-hand side  belongs to $A^\varpi_{\text{Col}}({\bf P}^1  \setminus{S_r})$.
Therefore
by the coincidence principle (Proposition \ref{coincidence principle}),
the equality actually holds on the whole space of
${\bf P}^1({\mathbb C}_p) - ]{S_r}\setminus\{\overline{0}\} [$,
actually on an affinoid bigger than the space.
\end{proof}

The following is a reformulation of
the equation in Theorem \ref{ell-Li theorem}.

\begin{theorem}\label{ell-Li reformulation prop}
Fix a branch  $\varpi\in\mathbb C_p$.
Let $n_1,\dots,n_r\in \mathbb{N}$,
$\xi_1,\dots,\xi_{r}\in\mathbb{C}_p$
with $|\xi_j|_p= 1$ ($1\leqslant j\leqslant r$).
The equality
\begin{align}
&\ell^{(p)}_{n_1,\dots,n_r}(\xi_1,\dots,\xi_r; z) \notag \\
&=Li^{(p),\varpi}_{n_1,\dots,n_r}(\xi_1,\dots, \xi_r; z) 
+\sum_{d=1}^r\left(-\frac{1}{p}\right)^d
\sum_{1\leqslant i_1<\cdots<i_d\leqslant r}\sum_{\rho_{i_1}^p=1}\cdots\sum_{\rho_{i_d}^p=1}
Li^{(p),\varpi}_{n_1,\dots,n_r}\Bigl(\bigl((\prod_{l=1}^d\rho_{i_l}^{\delta_{i_lj}})\xi_j\bigr); z\Bigr)
\label{ell-Li reformulation} 
\end{align}
holds for $z\in{\bf P}^1({\mathbb C}_p) - ]{S_r}\setminus\{\overline{0}\} [$, 
where $\delta_{ij}$ is the Kronecker delta. 
\end{theorem}

\begin{proof}
It is a consequence of the following direct computation:
\begin{align*}
\ell^{(p)}_{n_1,\dots,n_r}&(\xi_1,\dots,\xi_{r}; z)
=
\underset{(k_1,p)=\cdots=(k_r,p)=1}{
{\underset{0<k_1<\cdots<k_{r}}{\sum}}}
\frac{\xi_1^{k_1}\cdots\xi_{r}^{k_{r}} z^{k_r}}{k_1^{n_1}\cdots k_r^{n_r}} 
=
{\underset{0<k_1<\cdots<k_{r}}{\sum}}
\left\{\prod_{i=1}^r
\left(1-\frac{1}{p}\sum_{\rho_i^p=1}\rho_i^{k_i}\right)
\frac{\xi_i^{k_i}}{k_i^{n_i}} \right\} z^{k_r}\\
=&Li^{(p)}_{n_1,\dots,n_r}(\xi_1,\dots,\xi_r; z)
+\sum_{d=1}^r\left(-\frac{1}{p}\right)^d
\sum_{1\leqslant i_1<\cdots<i_d\leqslant r}\sum_{\rho_{i_1}^p=1}\cdots\sum_{\rho_{i_d}^p=1}
Li^{(p)}_{n_1,\dots,n_r}\Bigl(\bigl((\prod_{l=1}^d\rho_{i_l}^{\delta_{i_lj}})\xi_j\bigr); z\Bigr).
\end{align*}
\end{proof}

\begin{example}
In the case when $r=1$ and $\xi_1=1$, \eqref{ell-Li reformulation} gives
$$
\ell_n^{(p)}(1;z)=Li_n^{(p)}(1;z)-\frac{1}{p}\sum_{\rho^p=1}Li_n^{(p)}(\rho;z).
$$
Combining this with
the distribution relation (cf.\cite{C})
\begin{equation}\label{distribution relation}
\frac{1}{r}\sum_{\eta\in\mu_r}Li^{(p)}_n(\eta; z)=
\frac{1}{r^n}Li_n^{(p)}(1;z^r)
\end{equation}
for $r\geqslant 1$,
we recover Coleman's formula 
$$
\ell_n^{(p)}(1;z)=Li_n^{(p)}(1;z)-\frac{1}{p^n}Li_n^{(p)}(1;z^p)
$$
shown in \cite[Section VI]{C}.
\end{example}

We define the $p$-adic twisted multiple $L$-values by
the special values of the $p$-adic TMPL's at unity
under a certain condition below.

\begin{theorem-definition}
Fix a branch  $\varpi\in\mathbb C_p$.
Let $n_1,\dots,n_r\in \mathbb{N}$,
$\rho_1,\dots,\rho_r\in\mu_p$ and
$\xi_1,\dots,\xi_r\in\mu_c$ with $(c,p)=1$.
Assume that 
\begin{equation}\label{assumption for pMLV}
\xi_1\cdots\xi_r\neq 1, \quad \xi_2\cdots\xi_r\neq 1,\quad
\dots, \quad  \xi_{r-1}\xi_r\neq 1,\quad \xi_r\neq 1.
\end{equation}
Then the special value of
\begin{equation}\label{function for pMLV}
Li^{(p),\varpi}_{n_1,\dots,n_r}(\rho_1\xi_1,\dots,\rho_r\xi_r; z)
\end{equation}
at $z=1$
is well-defined.
Furthermore it is free from any choice of the branch $\varpi\in\mathbb C_p$ and
it belongs to ${\mathbb Q}_p(\mu_{cp})$.
We call the value by {\bf $p$-adic twisted multiple $L$-value} 
and denote it by
\begin{equation}\label{FYpMLV}
Li^{(p)}_{n_1,\dots,n_r}(\rho_1\xi_1,\dots,\rho_r\xi_r).
\end{equation}
\end{theorem-definition}

\begin{proof}
By our assumption \eqref{assumption for pMLV},
$$1\in{\bf P}^1({\mathbb C}_p) - ]{S_r}\setminus\{\overline{0}\} [.$$
Hence the special value of \eqref{function for pMLV} at $z=1$ makes sense by Remark \ref{S_r-0}.

The branch independency of the special value follows from Proposition \ref{branch independent region}.

The points $0$ and $1$ and all the differential 1-forms
$\frac{\xi_r dt}{1-\xi_rt}$, 
$\frac{(\xi_{r-1}\xi_r) dt}{1-(\xi_{r-1}\xi_r)t}$, $\dots$,
$\frac{(\xi_1\cdots\xi_r) dt}{1-(\xi_1\cdots\xi_r)t}$
are defined over ${\mathbb{Q}}_p(\mu_{cp})$.
Then by the Galois equivalency stated in \cite{BdJ} Remark 2.3,
the special value of \eqref{function for pMLV}
for $z\in {\mathbb{Q}}_p(\mu_{cp})$ is
invariant under the action of the absolute Galois group 
${\rm Gal}(\overline{\mathbb{Q}}_p/{\mathbb{Q}}_p(\mu_{cp}))$
if we take $\varpi\in {\mathbb{Q}}_p(\mu_{cp})$.
 Since  we have shown that the special values at $z=1$  is free from the choice of $\varpi$,
the value \eqref{FYpMLV} must belong to ${\mathbb Q}_p(\mu_{cp})$.
\end{proof}

\begin{remark} \ 
\begin{enumerate}
\item
The $p$-adic multiple zeta values are introduced as
the special values at $z=1$ of $p$-adic MPL
$Li^{(p),\varpi}_{n_1,\dots,n_r}(1,\dots,1;z)$,
the function \eqref{function for pMLV} with all $\rho_i=1$  and $\xi_i=1$, 
and their basic properties are investigated   by the first-named author in  \cite{Fu1}.
\item
The $p$-adic multiple $L$-values introduced by Yamashita \cite{Y} are special values
at $z=1$ of  the function
\eqref{function for pMLV} with all $\rho_i=1$ and
$(\xi_r,n_r)\neq (1,1)$.
Unver's \cite{U} cyclotomic $p$-adic multi-zeta values might be closely related
to his values.
\end{enumerate}
\end{remark}

By Theorems \ref{L-ell theorem} and \ref{ell-Li theorem}, we have the following.
A very nice point here is that our assumption \eqref{assumption for pMLV}
appeared as the  condition of the summation in equation \eqref{L-ell-formula}. 

\begin{theorem}\label{L-Li theorem}
For  $n_1,\dots,n_r\in \mathbb{N}$
and $c\in \mathbb{N}_{>1}$ with $(c,p)=1$,
\begin{align*}
L_{p,r} & (n_1,\dots,n_r;\omega^{-n_1},\dots,\omega^{-n_r};1,\dots,1;c)= \notag \\
&\frac{1}{p^r}\sum_{0<\alpha_1,\dots,\alpha_r<p}\sum_{\rho_1^p=\cdots=\rho_r^p=1}
\underset{\xi_1\cdots\xi_r\neq 1,\dots,\xi_{r-1}\xi_r\neq 1,\xi_r\neq 1}{\sum_{\xi_1^c=\cdots=\xi_r^c=1}}
\rho_1^{-\alpha_1}\cdots\rho_r^{-\alpha_r} \cdot 
Li^{(p)}_{n_1,\dots,n_r}(\rho_1\xi_1,\dots,\rho_r\xi_r).
\end{align*}
%
\end{theorem}

The above theorem is reformulated to the following,
which is comparable to Theorem \ref{T-main-1}.

\begin{theorem}\label{L-Li theorem-2}
For  $n_1,\dots,n_r\in \mathbb{N}$
and $c\in \mathbb{N}_{>1}$ with $(c,p)=1$,
\begin{align}
L_{p,r} & (n_1,\dots,n_r;\omega^{-n_1},\dots,\omega^{-n_r};1,\dots,1;c) \notag \\
&=
\underset{\xi_1\neq 1}{\sum_{\xi_1^c=1}} \cdots
\underset{\xi_r\neq 1}{\sum_{\xi_r^c=1}} 
Li^{(p)}_{n_1,\dots,n_r}\left(\frac{\xi_1}{\xi_2},\frac{\xi_2}{\xi_3},\dots,\frac{\xi_{r}}{\xi_{r+1}}\right)   \notag \\
&+\sum_{d=1}^r\left(-\frac{1}{p}\right)^d
\sum_{1\leqslant i_1<\cdots<i_d\leqslant r}\sum_{\rho_{i_1}^p=1}\cdots\sum_{\rho_{i_d}^p=1}
\underset{\xi_1\neq 1}{\sum_{\xi_1^c=1}} \cdots
\underset{\xi_r\neq 1}{\sum_{\xi_r^c=1}} 
Li^{(p)}_{n_1,\dots,n_r}\Bigl(\bigl(\frac{\prod_{l=1}^d\rho_{i_l}^{\delta_{i_lj}}\xi_j}{\xi_{j+1}}\bigr)\Bigr),
\end{align}
where we put $\xi_{r+1}=1$.
\end{theorem}

\begin{proof}
It follows from 
Theorem \ref{L-ell theorem} and Theorem \ref{ell-Li reformulation prop}.
\end{proof}

Thus the positive integer values  of the $p$-adic multiple $L$-function
are described as linear combinations of 
the special values of the $p$-adic TMPL \eqref{function for pMLV}
at roots of unity.
This might be regarded as a $p$-adic analogue of the equality \eqref{zeta=Li}.

As  a special case of Theorem \ref{L-Li theorem}
we have the following.
When $r=1$, we have 

\begin{example}
For $n\in \mathbb{N}$
and $c\in \mathbb{N}_{>1}$ with $(c,p)=1$,
we have
\begin{equation*}
(c^{1-n}-1)\cdot L_{p}(n;\omega^{1-n})
=\frac{1}{p}\sum_{0<\alpha<p}\sum_{\rho^p=1}
\underset{\xi\neq 1}{\sum_{\xi^c=1}}
\rho^{-\alpha} Li^{(p)}_{n}(\rho\xi)
\end{equation*}
by Example \ref{example for r=1}.
This formula and  \eqref{distribution relation}
recover Coleman's equation \cite[(4)]{C}
\begin{equation}\label{Coleman L-Li}
L_p(n;\omega^{1-n})=(1-\frac{1}{p^n})Li_n^{(p)}(1).
\end{equation}
\end{example}

When $r=2$, we have
\begin{example}
For  $a,b\in \mathbb{N}$
and $c\in \mathbb{N}_{>1}$ with $(c,p)=1$,
\begin{align*} 
& L_{p,2}(a,b;\omega^{-a},\omega^{-b}; 1,1; c) \\
&=\frac{1}{p^2}\sum_{0<\alpha,  \beta <p} \ 
\sum_{\rho_1,\rho_2\in\mu_p} \ 
\underset{\xi_1\xi_2\neq 1,\xi_2\neq 1}{\sum_{\xi_1,\xi_2\in\mu_c}}
\rho_1^{-\alpha}\rho_2^{-\beta}Li_{a,b}(\rho_1\xi_1,\rho_2\xi_2) \\
&=
\underset{\xi_1\neq 1}{\sum_{\xi_1^c=1}} \ \underset{\xi_2\neq 1}{\sum_{\xi_2^c=1}} 
Li^{(p)}_{a,b}\left(\frac{\xi_1}{\xi_2},\xi_2 \right) 
-\frac{1}{p}\sum_{\rho^p=1} \ 
\underset{\xi_1\neq 1}{\sum_{\xi_1^c=1}} \  \underset{\xi_2\neq 1}{\sum_{\xi_2^c=1}} 
\left\{Li^{(p)}_{a,b}\left(\frac{\rho\xi_1}{\xi_2},\xi_2
\right) +
Li^{(p)}_{a,b}\left(\frac{\xi_1}{\xi_2},{\rho\xi_2}\right)\right\}
\\
&\qquad\qquad +\frac{1}{p^2}\sum_{\rho_{1}^p=1} \ \sum_{\rho_{2}^p=1} \
\underset{\xi_1\neq 1}{\sum_{\xi_1^c=1}} \ \underset{\xi_2\neq 1}{\sum_{\xi_2^c=1}} 
Li^{(p)}_{a,b}\left(\frac{\rho_1\xi_1}{\xi_2},{\rho_2\xi_2}\right).
\end{align*}
\end{example}

In our forthcoming paper \cite{FKMT02}, the above Theorem \ref{L-Li theorem-2}
will be extended into that for indices $(n_j)$ with non-all-positive integers. 
\ 

\noindent
{\bf Acknowledgements.}\ 
The authors wish to express their thanks to the Isaac Newton Institute for Mathematical Sciences, Cambridge, and the Max Planck Institute for Mathematics, Bonn, where parts of 
this work have been carried out.

\


\bibliographystyle{amsplain}


\ 

\begin{flushleft}
\begin{small}
H. Furusho\\
Graduate School of Mathematics \\
Nagoya University\\
Furo-cho, Chikusa-ku \\
Nagoya 464-8602, Japan\\
{furusho@math.nagoya-u.ac.jp}

\ 

Y. Komori\\
Department of Mathematics \\
Rikkyo University \\
Nishi-Ikebukuro, Toshima-ku\\
Tokyo 171-8501, Japan\\
komori@rikkyo.ac.jp

\ 

K. Matsumoto\\
Graduate School of Mathematics \\
Nagoya University\\
Furo-cho, Chikusa-ku \\
Nagoya 464-8602, Japan\\
kohjimat@math.nagoya-u.ac.jp

\ 

H. Tsumura\\
Department of Mathematics and Information Sciences\\
Tokyo Metropolitan University\\
1-1, Minami-Ohsawa, Hachioji \\
Tokyo 192-0397, Japan\\
tsumura@tmu.ac.jp

\end{small}
\end{flushleft}

\end{document}